\documentclass[a4paper,11pt,reqno]{article}
\usepackage[margin=25.4mm]{geometry}

\usepackage{epsfig, amssymb, amsmath,amsfonts,amssymb,amsthm,epsf,a4,color,graphicx,eucal,mathrsfs,bigfoot,paralist,upgreek,xcolor}
\usepackage[utf8]{inputenc}

\definecolor{ddmagenta}{rgb}{0.8,0,0.8}
\definecolor{dgreen}{rgb}{0.1,0.5,0.7}

\usepackage{mathrsfs,eucal}

\numberwithin{equation}{section}

\DeclareNewFootnote{AAffil}[arabic]
\DeclareNewFootnote{ANote}[fnsymbol]

\usepackage{etoolbox}
\makeatletter
\patchcmd\maketitle{\def\@makefnmark{\rlap{\@textsuperscript{\normalfont\@thefnmark}}}}{}{}{}
\makeatother

\makeatletter
\def\thanksAAffil#1{
  \footnotemarkAAffil\protected@xdef\@thanks{\@thanks%
        \protect\footnotetextAAffil[\the \c@footnoteAAffil]{#1}}%
}
\def\thanksANote#1{%
  \footnotemarkANote%
  \protected@xdef\@thanks{\@thanks%
        \protect\footnotetextANote[\the \c@footnoteANote]{#1}}%
}
\makeatother

\usepackage{tcolorbox}

\newcommand*{\colorboxed}{}
\def\colorboxed#1#{%
  \colorboxedAux{#1}%
}
\newcommand*{\colorboxedAux}[3]{%
  \begingroup
    \colorlet{cb@saved}{.}%
    \color#1{#2}%
    \boxed{%
      \color{cb@saved}%
      #3%
    }%
  \endgroup
}

\newtheorem{thmx}{Theorem}
\newtheorem{notation}{Notation}

\newtheorem{hypx}[thmx]{Hypothesis}

\usepackage{xcolor}
\usepackage{eucal,paralist,enumerate,verbatim}
\usepackage[textsize=footnotesize]{todonotes}
\usepackage{amsmath,amsfonts,amsthm,amssymb,amsxtra,dsfont}
\usepackage{mathrsfs, pstricks}
\usepackage[english]{babel}
\usepackage{amsxtra,latexsym}
\usepackage{exscale}
\usepackage{graphicx,mathrsfs}
\usepackage{psfrag}
\usepackage{eucal,enumerate,paralist}
\usepackage{accents}
\usepackage{color,graphics}
\usepackage{extpfeil}
\usepackage[T3,T1]{fontenc}

\definecolor{dblue}{rgb}{0,0,0.5}
\definecolor{hellblau}{rgb}{0.85,0.9,1}
\definecolor{bernstein}{rgb}{0.9,0.6,0.05}
\def\trait #1 #2 #3 {\vrule width #1pt height #2pt depth #3pt}
\def\fin{
    \trait .3 5 0
    \trait 5 .3 0
    \kern-5pt
    \trait 5 5 -4.7
    \trait 0.3 5 0
\medskip}

\newcommand{\QED}{\hfill $\square$}

  \def\bbC{{\mathbb C}}

\def\bbV{{\mathbb V}}

\def\calD{{\mathcal D}} \def\calE{{\mathcal E}}

\def\calP{{\mathcal P}}

  \def\bff{{\FG f}}
\def\bfg{{\FG g}}

\newcommand{\bele}{\begin{lemm}\begin{sl}}
\newcommand{\enle}{\end{sl}\end{lemm}}
\newcommand{\bedef}{\begin{defi}\begin{sl}}
\newcommand{\eddef}{\end{sl}\end{defi}}
\newcommand{\bete}{\begin{teor}\begin{sl}}
\newcommand{\ente}{\end{sl}\end{teor}}
\newcommand{\beos}{\begin{osse}\begin{rm}}
\newcommand{\eddos}{\end{rm}\end{osse}}
\newcommand{\bepr}{\begin{prop}\begin{sl}}
\newcommand{\empr}{\end{sl}\end{prop}}
\newcommand{\bepro}{\begin{prob}\begin{rm}}
\newcommand{\empro}{\end{rm}\end{prob}}
\newcommand{\bede}{\begin{defin}\begin{sl}}
\newcommand{\edde}{\end{sl}\end{defin}}
\newcommand{\beco}{\begin{coro}\begin{sl}}
\newcommand{\enco}{\end{sl}\end{coro}}
\newcommand{\behy}{\begin{hypo}\begin{sl}}
\newcommand{\enhy}{\end{sl}\end{hypo}}

\newcommand{\thspace}{\hspace{3mm}}

\newcommand{\beeq}[1]{\begin{equation}\label{#1}}
\newcommand{\eddeq}{\end{equation}}
\newcommand{\beeqa}[1]{\begin{eqnarray}\label{#1}}
\newcommand{\eddeqa}{\end{eqnarray}}
\newcommand{\beal}[1]{\begin{align}\label{#1}}
\newcommand{\eddal}{\end{align}}
\newcommand{\bespl}[1]{\begin{split}\label{#1}}
\newcommand{\edspl}{\end{split}}
\newcommand{\bega}[1]{\begin{gather}\label{#1}}
\newcommand{\edga}{\end{gather}}
\newcommand{\beeqax}{\begin{eqnarray*}}
\newcommand{\eddeqax}{\end{eqnarray*}}

\newcommand{\weaksto}{{\rightharpoonup^*}}
\newcommand{\weakto}{\rightharpoonup}

\newcommand{\itt}{\int_0^t}

\newcommand{\ito}{\itt\!\io}

\DeclareMathOperator{\spa}{span}

\newcommand{\tensore}{\varepsilon({\bf u})}

\newcommand{\foraa}{\text{for a.a.}}

\def\fine{\hfill\kern4pt \vrule height4pt depth0pt width4pt }

\numberwithin{equation}{section}
\numberwithin{equation}{section}

\newcommand{\dd}{\, \mathrm{d}}
\newcommand{\pairing}[4]{ \sideset{_{#1 }}{_{ #2}}  {\mathop{\langle #3 , #4  \rangle}}}

\newcommand{\eps}{\varepsilon}

\newlength{\dhatheight}
\newcommand{\doublehat}[1]{%
    \settoheight{\dhatheight}{\ensuremath{\hat{#1}}}%
    \addtolength{\dhatheight}{-0.35ex}%
    \hat{\vphantom{\rule{1pt}{\dhatheight}}%
    \smash{\hat{#1}}}}

\newcommand{\DDD}[3]{\begin{array}[t]{c}#1\vspace*{-1em}\\_{#2}\vspace*{-.5em}\\_{#3}\end{array}}
\newcommand{\ddd}[3]{\DDD{\begin{array}[t]{c}\underbrace{#1}\vspace*{.6em}\end{array}}{\text{\footnotesize #2}}{\text{\footnotesize #3}}}
\newcommand{\down}{\downarrow}

\newcommand{\rmC}{\mathrm{C}}

\newcommand{\EEE}{\color{black}}

\newcommand{\e}{\varepsilon}
\newcommand{\DIV}{\,\mathrm{div}}
\newcommand{\R}{\mathbb{R}}
\newcommand{\N}{\mathbb{N}}
\newcommand{\C}{\mathcal}
\newcommand{\dx}{\,\mathrm dx}

\newcommand{\dt}{\,\mathrm dt}
\newcommand{\ds}{\,\mathrm ds}
\newcommand{\dta}{\,\mathrm d\tau}
\newcommand{\dS}{\,\mathrm dS}
\newcommand{\dr}{\,\mathrm dr}
\newcommand{\dxt}{\,\mathrm dx\,\mathrm dt}

\newcommand{\weaklim}{\rightharpoonup}
\newcommand{\weakstarlim}{\stackrel{\star}{\rightharpoonup}}

\newcommand{\ff}{\f f}
\newcommand{\uu}{\f u}
\newcommand{\ut}{\uu_t}
\newcommand{\utt}{\uu_{tt}}
\newcommand{\vv}{\f v}

\newcommand{\chit}{\chi_t}
\newcommand{\io}{\int_\Omega}

\newcommand{\intt}{\int_0^t}

\newcommand{\dis}[3]{{#1}^{#2}_{#3}}

\newcommand{\chik}{\dis \chi k \tau}
\newcommand{\chikk}{\dis\chi {k-1}\tau}
\newcommand{\uk}{\dis{\mathbf{u}} k \tau}
\newcommand{\ukk}{\dis{\mathbf{u}}{k-1}\tau}
\newcommand{\ukkk}{\dis{\mathbf{u}}{k-2}\tau}
\newcommand{\chil}{\dis{\chi}\ell\tau}

\newcommand{\ul}{\dis{\mathbf{u}}\ell\tau}
\newcommand{\ull}{\dis{\mathbf{u}}{\ell-1}\tau}

\newcommand{\chij}{\dis {\chi} j\tau}
\newcommand{\chijj}{\dis{\chi}{j-1}\tau}
\newcommand{\uj}{\dis{\mathbf{u}} j \tau}
\newcommand{\ujj}{\dis{\mathbf{u}} {j-1}\tau}

\newcommand{\ph}{\text{\boldmath$\varphi$}}
\newcommand{\hpsi}{\widehat\psi}

\newcommand{\olt}{\overline{\mathsf{t}}_\tau}

\newcommand{\olu}{\overline{\mathbf{u}}_\tau}
\newcommand{\ulu}{\underline{\mathbf{u}}_\tau}
\newcommand{\olchi}{\overline{\chi}_\tau}
\newcommand{\ulchi}{\underline{\chi}_\tau}
\newcommand{\olv}{\overline{\mathbf{v}}_\tau}
\newcommand{\ulv}{\underline{\mathbf{v}}_\tau}

\newcommand{\ointe}[2]{\overline{#1}_{#2}}
\newcommand{\uinte}[2]{\underline{#1}_{#2}}
\newcommand{\linte}[2]{{#1}_{#2}}

\newcommand{\aein}{\text{ a.e. in }}

\newcommand{\cD}{\mathcal{D}}
\newcommand{\cE}{\mathcal{E}}

\newcommand{\CC}{\mathbb{C}}
\newcommand{\VV}{\mathbb{V}}

\DeclareSymbolFont{tipa}{T3}{cmr}{m}{n}
\DeclareMathAccent{\invbreve}{\mathalpha}{tipa}{16}

\newtheorem{thm}{Theorem}[section]
\newtheorem{remark}[thm]{{\textbf Remark}}
\newtheorem{lemma}[thm]{{\textbf Lemma}}
\newtheorem{theorem}[thm]{{\textbf Theorem}}
\newtheorem{proposition}[thm]{{\textbf Proposition}}
\newtheorem{definition}[thm]{{\textbf Definition}}
\newtheorem{corollary}[thm]{{\textbf Corollary}}

\DeclareMathOperator{\Vm}{\mathbb{V}}

\DeclareMathOperator{\dom}{dom}
\newcommand{\tu}{\tilde{\f u}}
\newcommand{\tchi}{\tilde{\chi}}
\newcommand{\tut}{\tilde{\f u}_t}
\newcommand{\f}[1]{{\pmb{ #1}}}
\newcommand{\Cm}[2]{\mathbb{C}\varepsilon({ #1}){:}\varepsilon({ #2})}
\newcommand{\di}{\nabla{\cdot}}

\allowdisplaybreaks

\newcommand{\Lquo}{L^2(\Omega)_{/ \R}}
\newcommand{\Hquo}{H^1(\Omega)_{/ \R}}
\newcommand{\AC}{\mathcal{AC}}

\newcommand{\cgamma}[1]{{\gamma}_{#1}}
\newcommand{\Gdir}{\Gamma_{\!\scriptscriptstyle{\rm D}}}
\newcommand{\Gneu}{\Gamma_{\!\scriptscriptstyle{\rm N}}}
\newcommand{\Spx}{\boldsymbol{X}}
\newcommand{\Spxw}{\boldsymbol{X}_{\mathrm{weak}}}
\renewcommand{\bff}{\boldsymbol{f}}
\renewcommand{\bfg}{\boldsymbol{g}}
\newcommand{\zz}{\boldsymbol{z}}

\newcommand{\betaup}{\upbeta}

\newcommand{\RNEW}{\color{black}} 

\newcommand{\wh}[1]{\widehat{#1}}

\begin{document}
	\title{\bf Existence and weak-strong uniqueness for damage systems in viscoelasticity}

\author{
  Robert Lasarzik%
  \thanksAAffil{Weierstrass Institute for Applied
Analysis and Stochastics, Mohrenstr.~39  D-10117 Berlin.
Email: {\ttfamily  robert.lasarzik@wias-berlin.de}}, \ 
  Elisabetta Rocca%
  \thanksAAffil{Dipartimento di Matematica ``F. Casorati'',
 Universit\`a di Pavia and IMATI -- C.N.R., Via Ferrata 5  I-27100 Pavia, Italy. Email: {\ttfamily elisabetta.rocca@unipv.it}}, \
  Riccarda Rossi%
  \thanksAAffil{Dipartimento di Ingegneria Meccanica e Industriale, Universit\`a di Brescia, Via Branze 38 I-25133 Brescia and IMATI -- C.N.R., Via Ferrata 5  I-27100 Pavia, Italy.
Email: {\ttfamily riccarda.rossi\,@\,unibs.it}}%
}

%
%
%
%
%

\maketitle



\begin{abstract}
\noindent 
In this paper we investigate the existence of solutions and their weak-strong uniqueness 
 property \EEE
for a PDE system 
 modelling \EEE
 damage in viscoelastic materials.
 \par
 In fact, we address two solution concepts, \emph{weak} and \emph{strong} solutions. For the former, we obtain a global-in-time existence result, but the highly nonlinear character of the system
 prevents us from proving their uniqueness. 
 For the latter,
 we prove local-in-time existence. Then, we show that the strong solution, as long as it exists, \EEE  is unique in the class of weak solutions. 
This  \emph{weak-strong uniqueness} statement is 
proved by means of a suitable relative energy inequality. 
\end{abstract}


  \vspace{4mm}

\noindent {\bf AMS (MOS) subject clas\-si\-fi\-ca\-tion:}\thspace
35D30,  
35D35, 
74G25, 
74A45. 

\section{Introduction}

In this paper we address the following PDE system
\begin{subequations}
	\label{PDEsystem}
	\begin{align}
		&\utt-\DIV\big(a(\chi) \CC \e(\uu)+ b(\chi)\VV\e(\ut)\big)=  \RNEW \bff\EEE
			&&\text{a.e. in }\Omega{\times}(0,T),\label{uEq}\\
		&\chi_t+\partial I_{(-\infty,0]}(\chi_t)-\Delta\chi+  \frac12 \EEE  a'(\chi) \e(\uu) \CC \e(\uu)+ \partial W(\chi) \EEE
		\ni 0
			&&\text{a.e. in }\Omega{\times}(0,T),
			\label{chiEq}\\
		&\uu(0)=\uu_0,\quad\ut(0)=\vv_0,\quad\chi(0)=\chi_0
			&&\text{a.e. in }\Omega,
			\label{initialCond}\\
		&\chi\geq 0,\quad\chi_t\leq 0
			&&\text{a.e. in }\Omega{\times}(0,T),
			\label{constraints}
			\intertext{coupled with homogeneous Neumann boundary conditions for $\chi$}
		&\partial_{\pmb n}\chi=0\hspace*{12.7em}
			&&\text{a.e. on }\partial\Omega{\times}(0,T),\hspace*{2.7em}
			\intertext{and with Robin-type boundary conditions for $\uu$}
			\label{Robin-intro}
			&  \cgamma{0}   \pmb n {\cdot} \left(a(\chi) \CC \eps(\f u) + b(\chi)\VV\eps(\f u_t) \right)+ \cgamma{1} \f u _t + \cgamma{2} \f u = \bfg  && \text{ a.e. on }\partial\Omega{\times}(0,T)\,,
	\end{align}
	\end{subequations}
	 tuned by coeffcients  
	\[
	 \cgamma{0} ,\,  \cgamma{1},\,   \cgamma{2} \geq 0\,.
	 \] 
System \eqref{PDEsystem}
models  damage processes in  a viscoelastic material
 occupying a   bounded Lipschitz domain in  $ \R^d$, $d=1,2,3$. We consider the evolution of the phenomenon in a time-interval $(0,T)$ and 
set \EEE
$Q:=\Omega{\times} (0,T)$  
 and $\Sigma := \partial \Omega {\times} (0,T)$. \EEE
 The state variables 
  are the  vector of small displacements  $\uu$, satisfying the momentum balance \eqref{uEq}, and the damage parameter $\chi$, representing the local proportion of damage: $\chi=1$ means that the material is completely safe, while $\chi=0$ means it is completely damaged.  We formulate the damage flow rule in the framework of the theory of 
  \textsc{M.\ Fr\'emond} \cite{fremond} and so we allow the phase parameter $\chi$ to assume also intermediate values inbetween $0$ and $1$ in the points of the domain $\Omega$ where only partial damage occurs.
\par
In  \eqref{uEq}, $\varepsilon( \uu)_{ij}:=(\uu_{i,j}+\uu_{j,i})/2$ denotes the linearized symmetric strain tensor, 
  while $\CC$ and $\VV$ are the elastic and viscosity tensors, respectively. The $\chi$-dependent coefficients
 $ a,b \in \mathrm{C}^1(\R)$ \EEE
mark
 the damage dependence of the elasticity and viscosity  modula, respectively;
  we will precisely specify our conditions on $\CC$, $\VV$, $a$, and $b$, in Section \ref{s:2} ahead. 
  The momentum balance is  supplemented by the 
the Robin-type boundary condition \eqref{Robin-intro},  \EEE where 
 the parameters $ \cgamma{0},  \cgamma{1},  \cgamma{2}$ in principle may be tuned in such a way as to yield a variety of boundary conditions for $\uu$, among which
 \[
 \begin{cases}
 \text{Neumann boundary conditions} & \text{for } \cgamma 0 \neq 0, \ \cgamma 1 =\cgamma 2 =0,
 \\
 \text{time-dependent Dirichlet boundary conditions} &  \text{for } \cgamma 0 =0,  \min \{  \cgamma 1 ,\cgamma 2\}>0\,. 
 \end{cases}
 \]
 Later on, we will point out to which extent we can encompass the \emph{general}
 conditions \eqref{Robin-intro} in our analysis. 
\par
 The damage flow rule~\eqref{chiEq}
has a doubly nonlinear structure. Indeed, it features the subdifferential term $\partial I_{(-\infty,0]}(\chi_t)$, with $\partial I_{(-\infty,0]}: \R \rightrightarrows \R$ the
(convex analysis) subdifferential of the indicator function $ I_{(-\infty,0]}$,  which serves to the purpose of enforcing unidirectionality of damage evolution via the constraint $\chi_t \leq 0$ a.e.\ in $Q$.  In turn, the \EEE  ``double-well'' type potential $W:=\breve{W}+\invbreve{W}$  is assumed to be the sum of a convex (possibly non-smooth) part $\breve{W}$ and non-convex (but regular) part $\invbreve{W}$. Typical choices for $W$ which we can include in our analysis are the logarithmic
potential
\begin{equation}
\label{logW} W(r):= r \ln (r) + (1-r) \ln(1-r) -c_1 r^2 -c_2 r -c_3
\quad \forall \, r \in (0,1),
\end{equation}
where $c_1$ and $ c_2$  are positive constants, as well as the sum of the
indicator function $\breve{W}=I_{[0,1]}$, 
 forcing $\chi$ to range between $0$ and $1$,  \EEE with a smooth non convex $\invbreve{W}$.  Therefore, 
the subdifferential $\partial W$ includes the (possibly) multivalued subdifferential $\partial\breve{W}$.  \EEE 
We note that the upper wall of the well at $1$ will already be respected by the unidirectional damage evolution $ \chi_t \leq 0$ together with the condition on the initial value $ \chi_0\leq 1$ in $\Omega$.
The coupling with 
\eqref{uEq} occurs through the term \EEE
 $\e(\uu) \CC \e(\uu)$, which is a short-hand for the
colon product $\tensore \colon \CC\tensore$. 

 System \eqref{PDEsystem} can be derived in the frame of the modelling approach by Fr\'emond
\cite{fremond} (cf.\ also~\cite{BoBo, bosch, bss})
 from of the following  choices \EEE  of the free-energy 
functional and of the pseudo-potential of dissipation: 
\begin{align*}
	\cE(\uu,\chi,\ut):={}&\io \left\{  \frac12|\ut|^2 {+}  \frac12  a(\chi) \CC \e(\uu){:} \e(\uu) {+} \frac12|\nabla\chi|^2{+}W(\chi)\right\} \dx
	  + \int_{\partial\Omega} \frac{\cgamma2}2 |\uu|^2 \dd S  \,, 
					\\
	 \cD(\chi,\ut,\chit):={}&\io \left\{ b(\chi)\VV\e(\ut){:}\e(\ut) {+} |\chit|^2{+}I_{(-\infty,0]}(\chit) \right\} \dx
	  +
	\int_{\partial\Omega} \cgamma1 |\uu_t|^2 \dd S  \,.
				\end{align*}

\subsection*{Mathematical difficulties}

The main mathematical hurdles encountered in the study of this system are related to the $\chi$-dependence in the viscosity and elastic coffiecients $a$ and $b$ in \eqref{uEq}, 
 and to the nonlinear features of equation \eqref{chiEq}. In particular, the  simultaneous \EEE
 presence of the non-smooth subdifferentials 
of $I_{(-\infty,0]}$, and  $W$,
 and the quadratic term $\frac12 a'(\chi)\eps(\uu)\CC\eps(\uu)$ occurring in \eqref{chiEq}, impart a strongly
nonlinear character to the system, so that the related analysis turns out to be nontrivial.
\par
 In the pioneering papers~\cite{bosch, bss}, the momentum balance equation
(with \emph{scalar} displacements)
 had a 
degenerating character due to the loss of ellipticity in regions where 
$a(\chi)=b(\chi) =0$. Consequently, only local-in-time existence results were proven. 
However, in most papers \emph{complete damage} is avoided, and  
non-degenerating coefficients in front of either the elasticity or viscosity tensors are considered:
we will also adopt this assumption hereafter.
\par
  Still, the highly nonlinear coupling between the momentum balance and the damage flow rule poses a major hurdle  to
  global-in-time existence  as already shown in \cite{BoBo}, where the coupling with thermal effects was also encompassed.
As a remedy to that,  the flow rule for  $\chi$ has been often 
 regularized by means of a nonlinear $p$-Laplacian operator, with  the exponent $p$  greater  than the space dimension
 (or a linear fractional Laplacian, \cite{KnRoZa13VVAR}),
  in place of the usual Laplacian acting on $\chi$. Indeed, this leads to higher
  spatial regularity for the damage variable and, as a consequence, paves the way for enhanced elliptic regularity estimates in the momentum balance, as
   well.
   This strategy has led to 
   \emph{global-in-time} existence for damage models in thermoviscoelastic materials
    \cite{HR, rocca-rossi-deg,rocca-rossi-full}, even encompassing phase separation 
    \cite{HKRR}. 
    \par
    Finally, let us also mention that in
\cite{rocca-rossi-full}
we addressed the asymptotic analysis of the damage system with $p$-Laplacian regularization,
 where the case of the Laplacian operator was considered as a limit for $p\searrow 2$  in the $p$-Laplacian term. 
 In that case,  we showed that
 that the limit damage system needs to
 be formulated in a weaker fashion. We will dwell on this solvability concept later on. 
 \medskip
 \par
 The main aim of this paper is to cope with the analysis of system \eqref{PDEsystem}  \emph{without} resorting to any higher-order regularization of the damage flow rule. 
 In this context:
 \begin{enumerate}
 \item We will contend with \emph{global-in-time solvability} for \eqref{PDEsystem}. As the literature available up to now suggest, global existence may be expected only for  weak solutions to \eqref{PDEsystem}: we will carefully introduce  our solvability concept and provide a set of conditions on the constitutive functions of the model, on the forces, and on the initial data, guaranteeing the existence of global-in-time solutions.
 \item We will then turn to handling \emph{strong} solutions, with  the displacement $\uu$ and the damage variable $\chi$ sufficiently regular in such a way as to satisfy 
  system  \eqref{PDEsystem} pointwise. We will prove that such solutions exist at least locally in time. 
  \item We finally show that strong solutions are unique, as long as they exist, within the class of weak solutions.
 \end{enumerate}
 The latter
 property goes under the name of \emph{weak-strong} uniqueness. In this regard, let us mention that there  \EEE
   is nowadays a consolidated literature on weak-strong uniqueness results in the context of fluid dynamics, 
such as \textsc{Serrin}'s uniqueness 
result~\cite{serrin} for \textsc{Leray}'s weak solutions~\cite{leray} to the
incompressible \textsc{Navier}--\textsc{Stokes} equation in three space dimensions, 
or the weak-strong uniqueness for suitable weak-solutions to the 
incompressible \textsc{Navier}--\textsc{Stokes} system~\cite{Feireislrelative} 
or to the full \textsc{Navier--Stokes--Fourier} system~\cite{novotny}. 
The formulation of a relative 
energy inequality entailing weak-strong uniqueness of solutions for  thermodynamical systems 
 goes back to \textsc{Dafermos}~\cite{dafermos}.  In the context of fluid dynamics, 
 this method
  has also been used to show the 
stability of a stationary solution~\cite{feireislstab}, 
the convergence to a singular limit~\cite{fei}, or to derive \textit{a posteriori} estimates
for simplified models~\cite{fischer}. 
Even though the method is consolidated, there are fewer articles dealing with the case of nonconvex energies, and most of them are related to liquid crystals models (cf., e.g., ~\cite{hyper, weakstrongweak, weakstrong, diss}). Finally we can quote the more recent papers \cite{LRS} and \cite{ALR}, where the weak-strong uniqueness of solutions is obtained for the first time for a Fr\'emond model of phase transitions accounting also for the temperature-evolution and for some Oldroyd-B type models for viscoelasticity at large strains, respectively. 
 \EEE
\subsection*{Our results}
 Firstly, let us specify the notion of weak solution we will address. Our concept \EEE
  couples a standard variational formulation of the momentum balance,
with the  damage flow rule weakly formulated in terms of a one-sided variational inequality 
\[
\io \left( \chi_t(t) \psi{+}\nabla\chi(t){\cdot}\nabla\psi{+} \tfrac12 a'(\chi(t)) \CC \e(\uu(t)) 
					{:} \e(\uu(t)) \psi{+}W '(\chi(t)) \psi  \right)\dx\geq 0;
	\]
						for all $\psi\in H^1(\Omega)\cap L^\infty(\Omega)$ with $\psi \leq 0$ a.e.\ in $\Omega$, which is coupled 
with an  energy-dissipation \EEE inequality 
\begin{align*}
						\begin{aligned}
					& \cE(\uu(t),\chi(t),\ut(t)) +\int_0^t \cD(\chi(s),\ut(s),\chi_t(s))\ds 
					\\
					& 
					\leq 	\cE(\uu_0,\chi_0,\vv_0)
					+\int_0^t \langle {\bff (s)}, {\ut(s)} \rangle_{H^1(\Omega)} \dd s
					 +\int_0^t \pairing{}{H^{1/2}(\partial\Omega)}{\bfg (s)}{\ut(s)} \dd s \,.	
					\end{aligned}				
				\end{align*}
In Section \ref{s:2} ahead  we will provide more insight into  this notion of solution, which was first introduced \cite{HK-1,HK-2} for  PDE systems modelling damage
in bodies  at elastic equilibrium
 (hence, without inertial and viscous terms in the displacement equation), undergoing phase separation. 	
Our first main result,  Theorem~\ref{thm:1} below, states the existence of {\em global-in-time weak solutions}. Its proof, carried out in Section \ref{ss:proof-thm-1},
 relies on a time discrete approximation scheme suitably tailored in order to obtain, as a byproduct, the non-negativity of the damage parameter $\chi$.  
\par 
 The existence of  local-in-time weak solutions, cf.\  Theorem~\ref{thm:2},  will be proved throughout Section \ref{s:proof-thm2}. It relies on 
careful estimates, yielding
higher spatial regularity for $\uu$ and $\chi$.  The latter cannot be rigorously rendered on a time-discretization scheme, as they rely on a  local-in-time  Gronwall estimate
that is not available on the time discrete level.
 In fact, we will resort to a different method based on 
\emph{spatial} discretization (via a Faedo-Galerkin scheme) for a suitable approximation of system \eqref{PDEsystem}. For this approximate system we will prove  local existence  
via a fixed point argument, and accordingly obtain local-in-time solutions to \eqref{PDEsystem} by passing to the limit. \EEE
\par
Our weak-strong uniqueness result, Theorem \ref{thm:3}, will be obtained  in the case of a regular potential $W$ by means of the proof of a suitable relative energy inequality (cf.~Proposition~\ref{prop:weakstrong}). The proof of such a result in case of a non-smooth potential $W$ is still an open problem even for simpler semilinear equations.

\section{Main results}
\label{s:2}
In this section we  lay the ground for  \EEE our main results,   stating the \EEE existence of global-in-time weak solutions and of   local-in-time strong solutions  to the
 damage PDE system,  as well as 
 the weak-strong uniqueness property for \eqref{PDEsystem}. 

 Preliminarily, \EEE  we  settle some general notation that will be used throughout the paper.
\begin{notation} 
\label{not:1.1}
\upshape
 Given a Banach space $\mathrm{X}$, we will
 denote by 
$\pairing{}{\mathrm{X}}{\cdot}{\cdot}$  both 
  the duality pairing between $\mathrm{X}^*$ and $\mathrm{X}$ and that between $(\mathrm{X}^d)^*$ and $\mathrm{X}^d$; we  will   just write $\pairing{}{}{\cdot}{\cdot}$ for the inner Euclidean product in $\R^d$. 
  Analogously, we  will
 indicate  by $\| \cdot \|_{\mathrm{X}}$ the norm in $\mathrm{X}$ and, most  often, use the same symbol for the norm in $\mathrm{X}^d$, while we will  just write $|\cdot| $ for the Euclidean norm in $\R^d$. 
\par
Hereafter, we will
 use the symbols
$c,\,c',\, C,\,C'$, etc., whose meaning may vary even within the same   line,   to denote various positive constants depending only on
known quantities. 
Furthermore, the symbols $I_i$,  $i = 0, 1,... $,
will be used as place-holders for several integral terms (or sums of integral terms) appearing in
the various estimates: we  
will
  not be
consistent with the numbering, so that, for instance, the
symbol $I_1$ will occur several times with different meanings.  
\end{notation}
 \EEE
%
\subsection{Existence of weak solutions}
We collect the first basic set of conditions on the 
tensors $\CC$ and $\VV$ and on the 
constitutive functions $a$, $b$ and $W$.
\begin{hypx}[Constitutive functions]
\label{h:1}
The elasticity and viscosity tensors 
$\CC, \, \VV \in  \R^{d\times d \times d \times d }$ are symmetric and positive definite in the sense that 
\begin{equation}
\label{asstensors}
\begin{aligned}
\begin{cases}
\displaystyle
\mathbb{E} _{ijkl}= \mathbb{E}_{klij}=\mathbb{E} _{jikl}=\mathbb{E}_{ijlk}  \text{ for } i,j,l,k \in \{1,\ldots, d\}
 \\ 
\displaystyle \exists\,  \eta_{\mathbb E}>0 \ \ \forall\,  A\in \R^{d\times d}  \, : \qquad 
 \mathbb{E}  A:A \geq  \eta_{\mathbb E}   | A|^2 
 \end{cases} \qquad \text{for } \mathbb{E} \in \{ \CC, \VV\}.
 \end{aligned}
 \end{equation}
 For the  coefficient $a$ we require that 
\begin{subequations}
\label{ass-a}
\begin{align}
&
\label{ass-a-1}
a \in \mathrm{C}^1(\R),   
\\
\label{ass-a-2}
&  a \text{ is non-decreasing}, \  a(r)\equiv 0 \text{ for }  r \in (-\infty, 0]
\\
& 
\label{ass-a-3}
a \text{ is convex},
\end{align}
\end{subequations}
while we impose that 
\begin{align}
&
\label{ass-b}
b \in \mathrm{C}^0(\R) 
\text{ and } \exists\, b_0>0  \ \forall\, r\in \R \, : \quad  b(r) \geq b_0. 
\end{align}
Finally, we assume that 
\begin{subequations}
\label{assW-new}
\begin{align}
&
\label{assW-1-new}
\text{$W \in \mathrm{C}^1(\R)$,} 
\\
&
\label{assW-2-new}
\exists\, \ell \geq 0 \, : \ \text{the mapping }  r\mapsto W(r) + \frac\ell 2 |r|^2 \text{ is convex,}
 \\
 &
 \label{assW-3-new}
W(0) \leq W(r)  \text{ for all } r \leq 0.
\end{align}
\end{subequations}
\end{hypx}

\noindent
We now specify our conditions on the volume force and on the initial data for the existence of weak solutions.
\begin{hypx}[Force and data]
\label{h:2}
We require that 
 \begin{subequations}
 \begin{align}
&
\label{force+data}
\begin{aligned}
  &  \bff \in L^2(0,T;H^{1}(\Omega;\R^d)^*), \qquad \bfg \in  L^2(0,T;H^{-1/2}(\partial\Omega;\R^d))\,,
\\
 &  \uu_0 \in H^1(\Omega;\R^d),  \quad \f v_0\in L^2(\Omega;\R^d),  \quad \chi_0 \in H^1(\Omega) \text{ with } \chi _0\in[0,1] \text{ a.e.~in $\Omega$.}
 \end{aligned}
 \end{align}
 \end{subequations}
	\end{hypx}
  \noindent Clearly, in the case of homogeneous Dirichlet boundary conditions for $\uu$, \eqref{force+data} should have to be suitably modified by requiring, for instance, $\uu_0 \in H_0^1(\Omega;\R^d)$. \EEE
	\begin{remark}
	\upshape
	\label{rmk:hyp-weak-sol}
	A few comments on Hyp.\ \textbf{\ref{h:1}} are in order:
	\begin{compactenum}
	\item  We have confined to spatially homogeneous tensors $\CC$ and $\VV$, but for the analysis of weak solutions  we could indeed 
	handle a suitable dependence on $x$, cf.\ Remark \ref{rmk:ext} ahead.
	\item Clearly, it follows from \eqref{ass-a-2} that $a(r)\geq 0$ for all $r\in \R$; the possible degeneracy $a(\chi)=0$ for the coefficient modulating the elasticity
	tensor is compensated  by the fact that  the coefficient $b(\chi)$ stays strictly positive by \eqref{ass-b}.
	\item It follows from \eqref{assW-2-new} that  $W$ admits the  \emph{convex/concave} decomposition 
	\begin{subequations}
	\begin{equation}
	\label{cvx-cnc-as-conseq}
	W = \breve{W} + \invbreve{W} \qquad \text{with } \begin{cases}
	\breve{W}(r) = W(r) + \frac{\ell}{2} r^2,
	\\
	\invbreve{W}(r) = -\frac{\ell}{2} r^2\,.
	\end{cases}
	\end{equation}
	Obviously, since $W \in  \rmC^1(\R)$, we have that 
	 $\breve{W},\, \invbreve{W} \in \rmC^1(\R)$; we remark for later use that 
	\begin{equation}
	 \label{assW-3}
	 \begin{cases}
	 \breve{W}(0) \leq \breve{W} (r) &  \text{for all } r \leq 0,
	 \\
	  \invbreve{W}'(r) \leq 0  &  \text{for all } r \in [0,1],
	  \end{cases}
	\end{equation}
	where the first of \eqref{assW-3} obviously derives from \eqref{assW-3-new}. 
	\end{subequations} 
	\item Our requirements on $a$ are designed in such a way as to  construct, via time discretization, 
	weak solutions $(\uu,\chi)$ to system \eqref{PDEsystem} 
	fulfilling 
	$\chi \geq 0$ a.e.\ in $Q$, as well as the associated  energy-dissipation inequality, cf.\ \eqref{UEDI} ahead. 
	In fact, while postponing all details to Section \ref{ss:proof-thm-1}, we may mention that  condition  \eqref{ass-a-2} 
	is exploited in the proof of the positivity of the discrete damage variable via a maximum principle argument, cf.\ Lemma \ref{lemma:existenceDiscrSol}. In turn,
	the convexity of $a$ allows us to tailor the time discretization scheme for \eqref{PDEsystem} in such a way as to 
	guarantee the validity of a discrete energy-dissipation inequality, cf.\ Lemma \ref{lemma:discrEI} ahead. 
	\item We will also resort to the convex/concave splitting \eqref{cvx-cnc-as-conseq} of $W$ in the  proof of  Lemma \ref{lemma:discrEI}, while  
	 properties \eqref{assW-3} 
	of $\breve W$
	and $\invbreve{W}$  will be used for the proof of the positivity of $\chi$. 
	\end{compactenum}
	\end{remark}
%
%
%
%
\EEE

\par
Our notion of weak solution features the following energy and dissipation functionals 
	\begin{align}
	\label{def:en}  
	\cE(\uu,\chi,\ut):={}&\io \left\{  \frac12|\ut|^2 {+}  \frac12 \EEE a(\chi) \CC \e(\uu){:} \e(\uu) {+} \frac12|\nabla\chi|^2{+}W(\chi)\right\} \dx
	  + \int_{\partial\Omega} \frac{\cgamma2}2 |\uu|^2 \dd S  \,, 
					\\
					\label{def:diss}
	 \cD(\chi,\ut,\chit):={}&\io \left\{ b(\chi)\VV\e(\ut){:}\e(\ut) {+} |\chit|^2{+}I_{(-\infty,0]}(\chit) \right\} \dx
	  +
	\int_{\partial\Omega} \cgamma1 |\uu_t|^2 \dd S  \,.
				\end{align}
 In fact, while $\cE$ subsumes  the  contributions of the kinetic and  elastic energies, of the volume force,  and of the 
gradient regularization and potential energy for the damage variable,  $\cD$  encompasses the dissipation due to viscous damping and the quadratic
dissipation for the damage gradient flow, with the indicator term enforcing unidirectionality. 
The weak solvability concept 
that we specify in Definition \ref{def:weakSol}
 below has been introduced, for (purely) elastic damage models possibly coupled  with other diffusion processes, in 
\cite{HK-1,HK-2}.  According to this notion,  the (standard variational formulation of) the momentum balance  is coupled
with the  damage flow rule, weakly formulated in terms of 
\eqref{weakChiIneq} \&
\eqref{UEDI}.  This formulation
 reflects the fact that, if the subdifferential  in \eqref{chiEq} is lifted to an operator
$ \partial I_{(-\infty,0]}: H^1(\Omega) \rightrightarrows  H^1(\Omega)^*$, then
\eqref{chiEq} rephrases as
\[
\begin{cases}
\pairing{}{H^1(\Omega)}{-\{ \chi_t {-}\Delta\chi {+}  \tfrac12 a'(\chi) \CC \e(\uu) 
					{:} \e(\uu) {+}W '(\chi) \}}{\psi} \leq \int_\Omega I_{(-\infty,0]}(\psi) \dd x  & \text{for all } \psi \in  H^1(\Omega),
					\\
\pairing{}{H^1(\Omega)}{-\{ \chi_t {-}\Delta\chi {+}  \tfrac12 a'(\chi) \CC \e(\uu) 
					{:} \e(\uu) {+}W '(\chi) \}}{\chi_t} \geq \int_\Omega I_{(-\infty,0]}(\chi_t) \dd x,  & 
\end{cases}
\]
both inequalities holding a.e.\ in $(0,T)$.  Note that we used the $1$-homogeneity of $ I_{(-\infty, 0 ]}$ in order to deduce the two above inequalities from~\eqref{chiEq}. Then,  restricting the first inequality to negative test functions $\psi \in L^\infty(\Omega)$ 
(in order to have the term $\int_\Omega \tfrac12 a'(\chi) \CC \e(\uu) 
					{:} \e(\uu) \psi \dd x $ well defined) yields \eqref{weakChiIneq}. Adding the second inequality 
				with the weak momentum balance tested by $\uu_t$ and integrating in time 
						leads 
					to  \eqref{UEDI}, which is termed an \emph{upper} energy-dissipation inequality to emphasize that
					the
				overall energy $\cE(\uu(t),\chi(t),\ut(t))$ at the current process time is estimated from above by the initial energy
				and the work of the external forces. 
\begin{definition}[Weak solution]
	\label{def:weakSol}
		We call a pair $(\uu,\chi)$ a weak solution  to the Cauchy problem for  system \eqref{PDEsystem} if  
		\begin{subequations}
		\label{regs-u-chi}
		\begin{align}
			&\uu\in H^1(0,T;  H^1(\Omega;\R^d))\cap W^{1,\infty}(0,T;L^2(\Omega;\R^d))\cap H^2(0,T;  H^1(\Omega;\R^d)^*\EEE),
			\label{reg-u}
			\\
			&\chi\in L^\infty(0,T;H^1(\Omega))\cap L^\infty(Q)\cap H^1(0,T;L^2(\Omega))
			\label{reg-chi}
		\end{align}
		\end{subequations}
	satisfy initial conditions \eqref{initialCond}, constraints \eqref{constraints}, and 
		\begin{itemize}
			\item the \emph{weak momentum balance} for almost all $t\in (0,T)$, \textit{i.e.},
				\begin{align}
					& 	\label{weakUEq}
					\begin{aligned}
					& \langle\utt(t),\ph\rangle_{H^1(\Omega)} +\io \big( b(\chi(t))\VV\e(\ut(t)):\e(\ph){+} a(\chi(t))\CC \e(\uu):\e(\ph) \big)\dx
					\\
					& \quad 
					 +\int_{\partial\Omega} \big( \cgamma 1\uu_t {+} \cgamma 2 \uu \big) \ph 
					=\langle\bff(t),\ph\rangle_{H^1(\Omega)} + 
				\langle\bfg(t),\ph\rangle_{H^{1/2}(\partial\Omega)}  \EEE
					\end{aligned}
				\end{align}
				 for all $ \ph\in H^1(\Omega;\R^d)$; \EEE
			\item the \emph{one-sided variational inequality} for the damage flow rule, \textit{i.e.},   for almost all $t\in (0,T)$
				\begin{align}
					&\io \left( \chi_t(t) \psi{+}\nabla\chi(t){\cdot}\nabla\psi{+} \tfrac12 a'(\chi(t)) \CC \e(\uu(t)) 
					{:} \e(\uu(t)) \psi{+}W '(\chi(t)) \psi  \right)\dx\geq 0;
						\label{weakChiIneq}
						\end{align}
						for all $\psi\in H^1(\Omega)\cap L^\infty(\Omega)$ with $\psi \leq 0$ a.e.\ in $\Omega$;
\item the (overall)  \emph{upper energy-dissipation inequality},  \textit{i.e.},   for all $t\in [0,T]$		
						\begin{align}
						\label{UEDI}
						\begin{aligned}
					& \cE(\uu(t),\chi(t),\ut(t)) +\int_0^t \cD(\chi(s),\ut(s),\chi_t(s))\ds 
					\\
					& 
					\leq 	\cE(\uu_0,\chi_0,\vv_0)
					+\int_0^t \langle {\bff (s)},{\ut(s)}  \rangle_{H^1(\Omega)} \dd s
					 +\int_0^t \pairing{}{H^{1/2}(\partial\Omega)}{\bfg (s)}{\ut(s)} \dd s \,.	 \EEE
					\end{aligned}				
				\end{align}
%
		\end{itemize}
	\end{definition}
\EEE


\begin{theorem}[Global existence of weak solutions]
\label{thm:1}
Assume Hypotheses \textbf{\ref{h:1}} \& \textbf{\ref{h:2}}. 
 Then, 
    the Cauchy problem for system  \eqref{PDEsystem} admits a weak solution $(\f u , \chi)$ in the sense of Definition~\ref{def:weakSol}.
\end{theorem}
\noindent
The \emph{proof} will be carried out in Section \ref{ss:proof-thm-1}. 

\begin{remark}[Extensions]
\label{rmk:ext}
\upshape
 Theorem \ref{thm:1} may be extended to 
  the non-homogeneous case, \textit{i.e.}, with \EEE 
 spatially dependent tensors $\VV, \CC  \in L^\infty(\Omega;\R^{d\times d\times d\times d})$.
 \par
  Let us emphasize that, so far, we have not specified other conditions on the parameters $\cgamma i$, $i \in \{0,1,2\}, $ besides $\cgamma 1, \cgamma 2 \geq 0$. Thus, 
  as pointed out in the Introduction, the existence statement of Thm.\ \ref{thm:1} in particular encompasses the case of null Dirichlet boundary conditions on $\partial\Omega$,
  correspoding to $\cgamma 0=\cgamma 1=0$, $\cgamma2>0$, $\bfg \equiv \mathbf{0}$ in \eqref{Robin-intro}. Clearly, in that case the weak momentum balance would feature test functions $\varphi \in H_0^1(\Omega;\R^d)$. 
We could also allow for a suitable  time-dependent 
Dirichlet loading $\mathbf{w} $
enforcing the condition 
\begin{equation}
\label{time-dep-Dir}
\uu = \mathbf{w} \qquad \text{ on $\Sigma=\partial\Omega{\times}(0,T)$.}
\end{equation}
 Indeed, \eqref{time-dep-Dir} would  correspond to 
 the case $\cgamma 0=0$,  $\bfg \equiv \mathbf{0}$ , $\min \{ \gamma_1,\gamma_2 \}>0$, with 
 \[
  \mathbf{w}(t) = \mathbf{c} \exp\left( {-}\frac{\gamma_2}{\gamma_1} t \right) \qquad \text{for some vector }  \mathbf{c}  \in \R^d\,.
 \]
To handle \eqref{time-dep-Dir}, \EEE  it would be sufficient to formulate the momentum balance \eqref{weakUEq} for $\uu = \widehat{\uu} + \mathbf{w}$, 
with $\widehat \uu (t) \in H_0^1(\Omega;\R^d)$ for all $t\in [0,T]$,  and  seek for a solution
$\widehat \uu $   complying with homogeneous Dirichlet boundary conditions. \EEE
\par
A closer perusal of the proof of Thm.\  \ref{thm:1}  also reveals that, since our estimates do not hinge on elliptic regularity arguments
for the displacement variable, mixed  boundary conditions could be also considered for $\uu$:   in particular, 
the body could be clamped on a portion
$\Gdir$ of the boundary, while an assigned traction could be applied on $\Gneu = \partial\Omega{\setminus}\Gdir$.  \EEE
\par
The extension to the case of a nonsmooth potential $W$ is more delicate; it will be addressed in Section \ref{ss:3-outlook} ahead.
\end{remark}
\EEE
%
%

\subsection{Existence of strong solutions}
 We start by specifying our notion of \emph{strong} solvability for system \eqref{PDEsystem}  which, we recall, we address  in the case
  of a possibly nonsmooth convex potential $\breve W$. \EEE
In Definition \ref{def:strongSol} below, we ask for 
 enhanced regularity and integrability properties  for $
\uu$,  which as a consequence ensure that the term $a'(\chi)\CC \e(\uu){:} \e(\uu)$ in the damage flow rule belongs to $L^2(\Omega)$.
Then, both the momentum balance and the flow rule for $\chi$ make sense pointwise in space and time.
Moreover,
 by comparison, $H^2(\Omega)$-regularity 
follows for  $\chi$. 
 \EEE
\begin{definition}[Strong solution]
	\label{def:strongSol}
		We call a pair $(\uu,\chi)$ a strong solution 
		if it enjoys the 
		regularity properties
		\begin{equation}
		\label{reg-strong-sols}
		\begin{aligned}
			&\uu\in H^1(0,T;H^3(\Omega;\R^d))\cap W^{1,\infty}(0,T;H^2(\Omega;\R^d)) \EEE \cap H^2(0,T;  H^1(\Omega;\R^d)), \EEE \\
			&\chi\in L^\infty(0,T;H^2(\Omega))\cap H^1(0,T;H^1(\Omega))\,,
		\end{aligned}
		\end{equation}
		  and system \eqref{PDEsystem}, with  the 
		 boundary condition  
		\begin{equation}
\label{homogNeu}
{\f n} {\cdot} \CC \eps(\f u)   = 0 \qquad \text{ a.e.~on }\Sigma\,,
\end{equation}
		  is satisfied  pointwise a.e.\ in $Q$, which for the damage  flow rule means that
		  \begin{equation}
		  \label{precise-pointwise-flow-rule}
		  \begin{aligned}
		  &
		  \exists\, \eta, \, \xi \in L^2(Q) \quad \text{with } \begin{cases}
		  \eta \in \partial  \partial I_{(-\infty,0]}(\chi_t),
		  \\
		  \xi \in \partial \breve{W}(\chi) 
		  \end{cases} \quad \text{a.e.~in $Q$, such that }
		  \\
		  &
		   \chi_t - \Delta \chi + \frac{1}{2}a'(\chi) \Cm{\f u}{\f u} + \xi + \eta + \invbreve W'(\chi) = 0  \quad \text{a.e.~in $Q$.}
\end{aligned}
\end{equation}
	\end{definition}		  
 \begin{remark}
\label{rmk:EDBstrong}
\upshape
It is straightforward to check that any strong solution satisfies  inequality
\eqref{UEDI}
(in which the coefficients $\gamma_1$ and $\gamma_2$ in the dissipation potential $\mathcal{D}$ and  in the energy functional $\mathcal{E}$, respectively, are null due to the boundary condition \eqref{homogNeu}), as an energy-dissipation \emph{balance}. 
\end{remark} 
\par
We will prove the existence of local-in-time strong solutions
under an additional  smoothness condition for the spatial domain $\Omega$ to allow for regularity estimates. Namely, we require that
\begin{equation}
\label{omega-smooth}
\tag{$\mathrm{H}_\Omega$}
\Omega \subset \R^d \text{ is  a bounded domain of class $\mathrm{C}^1$,}
\end{equation}
and under the following strengthened versions of Hypotheses \textbf{\ref{h:1}} and  \textbf{\ref{h:2}}  (although we no longer need to assume $a$ convex, cf.\
\eqref{ass-a-3-strong} below). 
\begin{hypx}[Constitutive functions]
\label{h:1-strong}
In addition to \eqref{asstensors}  for the elasticity and viscosity tensors, we assume that 
\begin{subequations}
\label{ass-a-strong}
\begin{align}
&
\label{ass-a-1-strong}
    \mathbb C = \mathbb V\,,
\qquad a \in \mathrm{C}^2(\R), 
\\  
& 
\label{ass-a-3-strong}
\exists\, \kappa_1>0 \  \exists\, p\geq1 \ \ \forall\, r 
\in \R \, : \quad  |a''(r)| \leq \kappa_1 (|r|^p{+}1),
\end{align}
\end{subequations}
and analogously we impose that, in addition to \eqref{ass-b},
\begin{align}
&
\label{ass-b-strong}
b \in \mathrm{C}^2(\R), 
\text{ and } \exists\, \kappa_2>0 \  \exists\, q\geq1 \ \ \forall\, r 
\in \R  \, : \quad  |b''(r)| \leq \kappa_2 (|r|^q{+}1).
\end{align}
        As for $W$, we require that $W:\R\to (-\infty, +\infty] $, with $\dom W \neq \emptyset$, is $\ell$-convex
        as in \eqref{assW-2-new}.
 \EEE
\end{hypx}x
\begin{remark}
\label{rmk:hyp-const-strong}
\upshape
Let us motivate the  above conditions and compare them with Hypothesis \ref{h:1}:
\begin{enumerate}
\item The enhanced regularity required of the coefficients $a$ and $b$ will be instrumental in performing enhanced regularity estimates for the solutions. 
To carry them out, we will also resort to the polynomial growth conditions \eqref{ass-a-3-strong} and \eqref{ass-b-strong}, which obviously imply analogous growth conditions
for $a',\, b'$ and $a,\, b $, namely 
\begin{equation}
\label{ultimate-growth-conditions}
\begin{cases}
\exists\, \widehat{\kappa}_i>0 \ \forall\, r \in \R \, : \quad |\zeta'(r)| \leq \widehat{\kappa}_i (|r|^{\rho+1}{+}1),
\\
\exists\, \doublehat{\kappa}_i>0 \ \forall\, r \in \R \, : \quad |\zeta(r)| \leq \doublehat{\kappa}_i (|r|^{\rho+2}{+}1),
\end{cases}
\quad \text{for } i \in \{1,2\}, \ \zeta \in \{a,b\}, \ \rho \in \{p,q\}\,.
\end{equation}
\item Let us emphasize that Hypothesis \ref{h:1-strong} allows for nonsmoothness of $W$ (or, equivalently, of $\breve W$): in particular, in this context we can encompass the case in which $\breve W = I_{[0,\infty)}$, and positivity of $\chi$ is automatically enforced. 
\item
Condition \eqref{assW-2-new} guarantees the convex/concave decomposition 
$ W = \breve{W} + \invbreve{W} $,  with $\invbreve{W}(r) = -\frac{\ell}{2} r^2$, 
which we are going to use for the analysis of strong solutions, too. 
\end{enumerate}
\end{remark}
Our conditions on the force and on the initial data will be enhanced as well.   The compatibility condition \eqref{compat-below} below reflects  that we  confine our analysis  of strong solutions 
        to  the  homogeneous Neumann boundary conditions \eqref{homogNeu}.         
 \EEE 
\begin{hypx}[Force and data]
\label{h:2-strong}
We require that 
 \begin{subequations}
 \begin{align}
&
\label{force+data-strong}
   \bff \in L^2(0,T; H^1(\Omega;\R^d)),\quad    \uu_0 \in H^3(\Omega;\R^d) \EEE    \quad \f v_0\in H^2(\Omega;\R^d),  
  \\
&
     \label{compat-below}
  \f n  {\cdot}    \CC \eps(\uu_0)  = 0  
 \quad \text{a.e.\ on $\partial\Omega$.} 
\EEE
\intertext{We  take $\gamma_1=\gamma_2 =0$, $\bfg \equiv \mathbf{0}$, and $\gamma_0 \neq 0$
in \eqref{Robin-intro}, and assume}
\label{chidata-strong}
 &  \chi_0 \in H^2(\Omega) \text{ with } 
 \begin{cases}
 \chi_0 (x)\leq 1   \text{ for all $x\in \overline\Omega$} 
 \\
|\partial \breve{W}^{\circ}|(\chi_0) \in L^2(\Omega) 
 \end{cases}
 \end{align}
 \end{subequations}
 where 
 \[
 |\partial \breve{W}^{\circ}|(\chi_0(x)) := \inf\{ |\xi|\, : \ \xi \in \partial  \breve{W}(\chi_0(x)) \} \qquad \foraa\, x \in \Omega\,.
 \]
 	\end{hypx}
\begin{remark}[On \eqref{chidata-strong}]
\upshape
First of all, it is immediate to check that the $\inf$ in the definition of $ |\partial \breve{W}^{\circ}|(\chi_0(x))$ is indeed a $\min$. Furthermore, the von Neumann-Aumann selection theorem yields that there exists a \emph{measurable} selection
 \begin{subequations}
 \label{xi-circ}
 \begin{align}
 &
 \Omega \ni x \mapsto \xi^\circ(x) \in \mathrm{Argmin}\{ |\xi|\, : \ \xi \in \partial  \breve{W}(\chi_0(x)) \}\,,
\intertext{so that \eqref{chidata-strong} is indeed equivalent to requiring that}
&
 \xi^\circ \in  L^2(\Omega ) \,.
 \end{align}
 \end{subequations}
 From this there follows that  $ W(\chi_0) \in L^1(\Omega) $: in fact, taking into account that $\invbreve{W}(\chi_0) \in L^\infty(\Omega)$
 by the quadratic growth of $\invbreve W$, it is sufficient to show that  $ \breve{W}(\chi_0) \in L^1(\Omega) $. This is a consequence of the estimate 
\[
 \int_\Omega [\breve W(\chi_0) - \breve W( r_o)] \dd x  \leq \int_\Omega \xi^\circ (x) (\chi _0(x) {-}r_o) \dd x 
 \]
 where $r_o$ is \emph{any} element in  $\dom \breve W $.
\end{remark}
\EEE
%

Throughout Section \ref{s:proof-thm2} we will prove the following result. 
\begin{theorem}[Local existence of strong solutions]
\label{thm:2}
Assume Hypotheses \textbf{\ref{h:1-strong}} \& \textbf{\ref{h:2-strong}}; let $\Omega$ fulfill condition \eqref{omega-smooth}.  
\par
Then, 
there exists $\widehat T \in (0,T]$ such that 
the Cauchy problem for system \eqref{PDEsystem} admits   strong solution $(\uu,\chi)$ in the sense of Definition \ref{def:strongSol} on the interval $(0,\widehat T)$.
\end{theorem}
\begin{remark}[Positivity for strong solutions]
\upshape
As previously pointed out, our analysis of strong solutions encompasses the choice of 
a nonsmooth potential $ \breve{W}$. In that  case, if  we additionally have  $\mathrm{dom}(\breve W) \subset [0,\infty)$, then we  immediately obtain that
$  \chi(x,t)\geq 0$ for all $(x,t)\in \overline\Omega \times [0,T]$. 
  \par
  An alternative way for obtaining nonnegativity of strong solutions is via  the   weak-strong uniqueness guaranteed  by Thm.\ \ref{thm:3} ahead: in this way,  we
deduce $\chi \geq 0$ for the strong solution, since it is coincides with the weak one, which is known to be  positive for instance under the assumptions of Thm.\ \ref{thm:1}. 
\par
Outside  these two cases, we do not claim positivity of $\chi$ and it is actually not needed in the analysis
of strong solutions. 
 Especially, in the case of nonmonotone $a$, we do not expect  such a  property due to the negative contribution on the left-hand side of the damage flow rule. 
\end{remark} 
 We conclude this section with  a consistency result, useful for the proof of Thm.\ \ref{thm:2}, showing that, for  a sufficiently regular pair $(\uu,\chi)$,  the pointwise flow rule 
may be proved  by just checking a variational inequality, cf.\ \eqref{weakChiIneqSub} below, joint with the energy-dissipation inequality \eqref{UEDI}. 
\begin{proposition}
\label{prop:consistency}
    Let $(\f u , \chi)$ enjoy the regularity properties \eqref{reg-strong-sols} and fulfill  the weak momentum balance~\eqref{weakUEq}  and the  energy-dissipation inequality \eqref{UEDI}.
    \par
    Then, $(\f u , \chi)$ satisfies the pointwise flow rule \eqref{precise-pointwise-flow-rule},  joint with a selection
    $\xi  \in  \partial \breve W (\chi)$ a.e.~in $Q$,
     if and  only if it complies with  the variational inequality
%
%
%
%
  \begin{align}
					&\io \left( \chi_t(t) \psi{+}\nabla\chi(t){\cdot}\nabla\psi{+} \tfrac12 a'(\chi(t)) \CC \e(\uu(t)) 
					{:} \e(\uu(t)) \psi{+}(\xi{+} \invbreve W'(\chi(t)) ) \psi  \right)\dx\geq 0
						\label{weakChiIneqSub}
						\end{align}
  for all $\psi\in H^1(\Omega)\cap L^\infty(\Omega)$ with $\psi \leq 0$ a.e.\ in $\Omega$. 
\end{proposition}
\begin{proof}
Clearly, it suffices to show that from \eqref{weakChiIneqSub} we can derive  \eqref{precise-pointwise-flow-rule}. For this, we start by observing that, 
by the assumed regularity~\eqref{reg-strong-sols}, \EEE
    \begin{equation}
    \label{reg-of-eta}
    \eta : = - \left(  \chi_t {-} \Delta \chi {+} \frac{1}{2}a'(\chi) \Cm{\f u}{\f u} {+} \xi {+} \invbreve W'(\chi)  \right) \in L^2(Q)\,.
    \end{equation}
Choosing $\ph = \uu $ in~\eqref{weakUEq} and subtracting  this from~\eqref{UEDI}, we find 
\[
\begin{aligned}
    & \int_\Omega \frac{1}{2}a(\chi) \Cm{\uu}{\uu} {+} \frac{1}{2}|\nabla \chi|^2 {+} W(\chi) \dx \Big|_0^t {+} \int_0^t \int_\Omega\left( |\chi_t|^2  {-} a(\chi) \Cm{\uu}{\uu_t} \right) \dx \ds 
    \\ & \leq - \int_0^t  I_{(-\infty,0]}(\chi_t) \dx \ds   \,.
\end{aligned}
\]
 Now, by the chain rule (which holds since $\eta$ and $\chi_t$ are in a duality pairing thanks to \eqref{reg-of-eta}), the above left-hand side equals $\int_\Omega ({-}\eta)\chi_t \dd x$. Thus, we deduce \EEE
\begin{align}
\label{converse4indicator}
    \int_0^t \int_\Omega  \eta \chi_t \dx \ds \geq  \int_0^t \int_\Omega I_{(-\infty,0]}(\chi_t) \dx \ds  \,
\end{align}
for a.e.~$t\in (0,T)$.
In turn,  inequality~\eqref{weakChiIneqSub}  and a density argument (again relying on \eqref{reg-of-eta}) \EEE implies  that 
    \begin{align}
    \label{straight4indicator}
     \int_\Omega  \eta \psi \dx  \leq   \int_\Omega I_{(-\infty,0]}(\psi) \dx  \qquad \text{for all } \psi \in L^2(\Omega) \,,
\end{align}
where we note that the inequality becomes empty in the case that  $\psi > 0$ on a set of positive measure in  $\Omega$.  
 Combining \eqref{converse4indicator} and \eqref{straight4indicator} we deduce $ \eta \in \partial I_{(-\infty, 0]}(\chi_t)$ a.e.\ in $Q$, \textit{i.e.},   \eqref{precise-pointwise-flow-rule}.
\end{proof}
\EEE
\subsection{Weak-strong uniqueness}
We will prove the weak-strong uniqueness property for the Cauchy problem for system \eqref{PDEsystem},  confining the discussion to 
 the case of the 
homogeneous Neumann boundary conditions \eqref{homogNeu}. 
We will work 
 under the following  conditions. 
 \begin{hypx}
\label{h:4}
We assume that $\bbC = \bbV$ complies with  \eqref{asstensors}, and that the nonlinear functions $a$, $b$, and $W$ satisfy
\begin{subequations}
\label{ass-WS}
\begin{align}
&
\label{ass-a-WS}
a \in \mathrm{C}^2(\R) 
\text{ is convex and  non-decreasing},
\\
&
\label{ass-b-WS}
b \in \mathrm{C}^1(\R) 
\text{ and } \exists\, b_0>0  \ \forall\, r\in \R \, : \quad  b(r) \geq b_0,
\\
\label{ass-W-WS}
&
\text{$W \in \mathrm{C}^2(\R)$ and } 
\exists\, \ell \geq 0 \, : \ \text{the mapping }  r\mapsto W(r) + \frac\ell 2 |r|^2 \text{ is convex.}
\end{align}
\end{subequations} 
\end{hypx}

\begin{theorem}[Weak-strong uniqueness]
\label{thm:3}
Under Hypothesis \textbf{\ref{h:4}},
let $ (\f u , \chi )$ be a weak solution  in the sense of Definition~\ref{def:weakSol} and $(\tu , \tchi)$ a strong solution in the sense of Definition~\ref{def:strongSol} emanating from the same initial data, with forcing term $\f f$ as in   \eqref{force+data}. 
Then it holds 
$$
\f u(t) = \tu(t) \quad \chi (t) = \tchi(t) \quad \text{for all }t\in[0, \hat T]\,.
$$
\end{theorem}
 The \emph{proof} will be carried out in Section~\ref{s:rel} ahead.

\section{Proof of Theorem \ref{thm:1}}
\label{ss:proof-thm-1}
	 We will prove the existence of weak solutions by resorting to a suitable time-discretization scheme. 
	Let $\tau=T/K$ be the time step size of  an equidistant partition
	$\{ 0 = \dis t0{\tau} < \dis t1\tau < \ldots <\dis tk \tau < \ldots < \dis t K \tau =T\}$
	of $[0,T]$ into $K$ subintervals. We will approximate the volume  and surface forces \EEE  by local means on the intervals $[\dis t {k-1}\tau, \dis t k \tau]$, by setting
	\begin{equation}
	\label{local-means-f}
	\dis \bff k \tau: = \frac1\tau \int_{\dis t{k-1}\tau}^{\dis t k \tau} \bff(s) \dd s\,, \qquad 	 \dis \bfg k \tau: = \frac1\tau \int_{\dis t{k-1}\tau}^{\dis t k \tau} \bfg(s) \dd s\,.
	\end{equation}
	\EEE Hence, the time discretization scheme for system \eqref{PDEsystem} reads,  in its strong formulation, \EEE
		\begin{subequations}
	\label{discrPDEsystem}
	\begin{align}
		&\frac{\uk-2\ukk+\ukkk}{\tau^2}-\DIV\Big(b(\chik)\VV\e\Big(\frac{\uk-\ukk}{\tau}\Big)+ a(\chik) \CC\e(\uk)\Big)=  \dis \bff  k \tau 
			\quad \text{in }\Omega,\\
		&\left.
			\begin{aligned}
				&\frac{\chik-\chikk}{\tau}+\partial I_{(-\infty,0]}\Big(\frac{\chik-\chikk}{\tau}\Big)-\Delta\chik\\&
				\qquad + \frac{  a'(\chik)
    }{2} \CC \e(\ukk){:}\e(\ukk)
				+  \breve W'(\chik) \EEE +\invbreve W'(\chikk)\ni 0
			\end{aligned}\;\;
			\right\}
			\quad \text{in }\Omega,\\
		&   \cgamma{0}   \pmb n {\cdot} \left(a(\chik) \CC \eps(\uk) {+} b(\chik)\VV\eps \left(\frac{\uk{-} \ukk}{\tau} \right)\right) + \cgamma{1} \frac{\uk{-} \ukk}{\tau} +
		 \cgamma{2} \uk = \dis \bfg k \tau \quad \text{on }\partial \Omega, \EEE
		 \\
		 &
		\partial_{\pmb n}\chik=0
			\quad \text{on }\partial \Omega, 
			\label{discrPDEsystem5}
	\end{align}
	supplemented with the initial conditions $\uu_0$,  $\dis \uu {-1} \tau: = \uu_0 - \tau \vv_0$, and $\chi_0$. \EEE
%
%
	\end{subequations}
	In the above scheme,   the convex/concave   splitting $W=\breve W+\invbreve W$ from  \eqref{cvx-cnc-as-conseq} 
	has been carefully combined with the choice of implicit/explicit terms in such a way as to yield the validity of a discrete energy-dissipation 
	inequality, cf.\ Lemma \ref{lemma:discrEI} ahead.
	\subsection{Existence and a priori estimates for time-discrete solutions}
	\label{ss:3.0}
	With our first result, we establish the existence of solutions to the weak formulation of system \eqref{discrPDEsystem}. Additionally, we prove the positivity 
	property $\chik \geq 0$ a.e.\ in $\Omega$ via a maximum principle argument mimicking that from \cite[Prop.\ 4.2]{KnRoZa13VVAR}. \EEE
	\begin{lemma}[Existence of time-discrete solutions]
	\label{lemma:existenceDiscrSol}
		Starting from  $\dis \uu {-1} \tau = \uu_0 - \tau \vv_0$,   $\dis u 0\tau: = u_0$, and $\dis\chi 0\tau: = \chi_0$, 
		there exists $\overline{\tau}>0$ such that for all $0<\tau<\overline{\tau}$ and 
		for every  $k=1,\ldots,K$ \EEE
		there exists a \emph{weak} solution  $\uk\in H^1(\Omega;\R^d)$  \EEE and $\chik\in H^1(\Omega)\cap L^\infty(\Omega)$ 
			to the time-discrete system \eqref{discrPDEsystem},
			fulfilling 
					\begin{subequations}
		\label{discrPDEsystemWeak}
		\begin{align}
			&
			\begin{aligned}
			&
			\io \left\{ \frac{\uk-2\ukk+\ukkk}{\tau^2}{\cdot}\ph{+}
			b(\chik)\VV\e\Big(\frac{\uk{-}\ukk}{\tau}\Big){:}\e(\ph)
			{+}a(\chik) \C\e(\uk){:}\e(\ph) 
			\right\}\dx
			\\
		 & 	 \quad +\int_{\partial\Omega} \left(  \gamma_1 \frac{\uk{-}\ukk}{\tau} {+} \gamma_2 \uk \right) \dd S 
			=
			 \langle{\dis\bff k\tau},{\ph} \rangle_{H^1(\Omega)}  +  \pairing{}{H^{1/2}(\partial\Omega)}{\dis\bfg k\tau}{\ph}  \EEE 
			 \end{aligned}
			 \EEE
				\label{discrUeqWeak}
			\intertext{for all $\ph\in H^1(\Omega;\R^d)$,}
			&
			\label{discrChiEqWeak}
				\begin{aligned}
					&\io \left\{ \frac{\chik{-}\chikk}{\tau}(\psi{-}\chik){+} \nabla\chik{\cdot}\nabla(\psi{-}\chik) 
       {+} \frac12 a'(\chik) \EEE
      \Cm{\ukk}{\ukk}(\psi{-}\chik) \right\}\dx
      \\
					&\hspace*{13.5em}+ \io\left\{  \breve{W}'(\chik)\EEE (\psi{-}\chik){+}\invbreve{W}'(\chikk)(\psi{-}\chik) \right\} \dx\geq 0
				\end{aligned}
				\intertext{for all $\psi\in X_\tau^{k-1}:=\big\{v\in H^1(\Omega)\cap L^\infty(\Omega):\;v\leq \chikk \aein\Omega\big\}$, as well as the constraints}
				& 
				\label{constraints-chi}
				0 \leq \chik \leq \chikk  \leq 1 \qquad \text{a.e.\ in } \Omega\,.
		\end{align}
		\end{subequations}
	\end{lemma}
	\begin{proof}
		Let $k=1,\ldots,K$ be given
	 and, accordingly,  $\ukk \in  H_0^1(\Omega;\R^d)$ and $\chikk\in H^1(\Omega)\cap L^\infty(\Omega)$. We first construct a solution to \eqref{discrChiEqWeak}
	by finding a 
		 minimizer $\chik\in H^1(\Omega)$  
		 for 
		 of the convex potential
		$\C P:H^1(\Omega) {\cap} L^\infty (\Omega) \to\R$ \EEE  defined by  
		\begin{equation}
		\label{P-functional}
		\begin{aligned}
			& \C P(\chi)= \mathcal{P}_1(\chi) + \mathcal{P}_2(\chi)  \qquad \text{with }
			\\
			& \qquad \begin{cases}
			 \displaystyle  \mathcal{P}_1(\chi)    {=} \io \left\{ \tfrac{\tau}{2}\Big|\frac{\chi{-}\chikk}{\tau}\Big|^2{+}\frac12|\nabla\chi|^2{+}  \breve{W}(\chi){+}\invbreve{W}'(\chikk)\chi \right\} \dx
			\\
			 \displaystyle  \mathcal{P}_2(\chi)  {=} \io  \tfrac12 a(\chi)  \Cm{\ukk}{\ukk} \dd x  +  I_{\tilde{X}_\tau^{k-1}}(\chi)\,,
			\end{cases}
			\end{aligned}
			\end{equation}
			with the set $ \tilde{X}_\tau^{k-1}:=\{v\in H^1(\Omega):  v\leq \chi^{k-1}\aein\Omega\}.$
			 We thus address the minimum problem 
			\begin{equation}
			\label{minimum-4-P}
			\min_{\chi \in \tilde{X}_\tau^{k-1}}  \C P(\chi). \qquad
			\end{equation} \EEE	
		First of all, observe that, for sufficiently small $\tau$ the functional   $\cal{P}$ is bounded from below   and suitably coercive.
		 To check this, we recall that,
		since $\breve{W}$ is convex, it is is bounded from below  by an affine function; combining this with the information 
		that $\invbreve{W}'(\chikk) \in L^\infty(\Omega)$  - since $0 \leq \chikk \leq \chi_0 \leq 1$ a.e.\ in 
		$\Omega$ -  we ultimately conclude that there exist  positive constants $c_W,\, C_W $, only depending on $W$, such that
		\begin{equation}
		\label{affine-W}
		\begin{aligned}
		\mathcal{P}_1(\chi)
		&  \geq \io \left\{ \frac{\tau}{2}\Big|\frac{\chi{-}\chikk}{\tau}\Big|^2{+}\frac12|\nabla\chi|^2 - c_W |\chi| - C_W \right\} \dd x  
		\stackrel{(1)}{\geq} c_W' \| \chi \|_{H^1(\Omega)}^2 - C_W'\,,
		\end{aligned}
		\end{equation}
		where {\footnotesize (1)} follows from absorbing $-c_W \|\chi\|_{L^1(\Omega)}$ into $\tfrac1 {2\tau}\|\chi{-}\chikk\|_{L^2(\Omega)}^2$, for sufficiently small $\tau$. 
		In turn, we observe that 
		if $
		\mathcal{P}_2(\chi) <+\infty
		$, then 
		$ \chi^+ \in L^\infty(\Omega)$ and, a fortiori, 
		 \[
		 a(\chi)  \Cm{\ukk}{\ukk}  = a((\chi)^+)  \Cm{\ukk}{\ukk} \in L^1(\Omega)
		 \] 
		 (the above equality holds because $a\equiv 0$ on $(-\infty, 0]$ by 
		\eqref{ass-a-2}). 
		All in all, we may conclude that 
		\[
		\forall S>0 \ \exists\, C_S>0 \ \forall \, \chi \in H^1(\Omega)\, : \qquad  \C P(\chi)\leq S \,  \Longrightarrow \, 
		 \begin{cases}
		 \| \chi \|_{H^1(\Omega)} \leq C_S
		 \\
		 \int_\Omega  a'(\chi)  \Cm{\ukk}{\ukk} \dd x \leq C_S
\end{cases}\,.
		\]
Therefore, any minimizing sequence for $\cal{P}$ is bounded in $H^1(\Omega)$ and thus  weakly converges, up to a subsequence, to some $\bar \chi \in H^1(\Omega) \cap \tilde{X}_\tau^{k-1}$;  by standard lower semicontinuity arguments we conclude that $\bar\chi$ is a minimizer for $\cal{P}$ and we set $\chik: = \bar \chi$.
		\par
	 We now show that \emph{any} solution $ \chik$ for the minimum problem \eqref{minimum-4-P} fulfills \EEE
	$\chik \geq 0$ a.e.\ in $\Omega$.  With this aim, we observe  that 
		the truncated function $(\chik)^+:=\max\{\chik,0\}$ fulfills a.e.\ in $\Omega$
		\[
		\begin{cases}
			 \Big\|\frac{(\chik)^+-\chikk}{\tau}\Big\|_{L^2}^2 & \leq \Big\|\frac{\chik-\chikk}{\tau}\Big\|_{L^2}^2,
			 \\
		 \|\nabla(\chik)^+\|_{L^2}^2 & \leq \|\nabla\chik\|_{L^2}^2,
			\\  
    a((\chik)^+)   & \leq 
   a(\chik), 
   \end{cases}
		\]
		where the latter estimate is again due to \eqref{ass-a-2}. 
		Furthermore the splitting $W=\breve W+\invbreve W$ is constructed in a way such that,  by \eqref{assW-3}, 
		\begin{align*}
			&&&&\breve W((\chik)^+)\leq{}& \breve W(\chik) && \aein\, \Omega, &&\\
			&&&&\invbreve W'(\chi^{k-1})(\chik)^+\leq{}& \invbreve W'(\chi^{k-1})\chik && \aein\, \Omega,
				&&\text{since $\invbreve W'(\chi^{k-1})\leq 0$.}&&&&
		\end{align*}
		Therefore,
		$
			\C P((\chik)^+)\leq \C P(\chik).
		$
		Due to the strict convexity of $\C P$, minimizers are unique and thus
		$\chik=(\chik)^+$. All in all, we have obtained \eqref{constraints-chi}. 
		\par
		 A fortiori, we have that any solution   $ \chik$  of \eqref{minimum-4-P} is indeed in $L^\infty(\Omega)$. Therefore, 
		\begin{equation}
		\label{auxiliary-min-prob}
		\chik \in \mathrm{Argmin}_{\chi \leq \chikk} \widetilde{\calP}(\chi),
		\end{equation}
		with $ \widetilde{\calP}: H^1(\Omega) {\cap}L^\infty(\Omega) \to \R$ the G\^{a}teaux-differentiable functional defined by 
		\[
		 \widetilde{\calP}(\chi):=  \mathcal{P}_1(\chi) + \widetilde{\mathcal{P}}_2(\chi)   \text{ and  }
		  \widetilde{\mathcal{P}}_2(\chi)  {=} \io  \tfrac12 a(\chi)  \Cm{\ukk}{\ukk} \dd x\,.
		  \]
		  We may apply, e.g., \cite[Lemma 2.21, p.\ 63]{Troltzsch-book}
		   to the auxiliary minimum problem \eqref{auxiliary-min-prob}, 
	and 	  
 the variational inequality \eqref{discrChiEqWeak} then follows 
		as first-order necessary condition.  \EEE
		\par	
		Finally, equation \eqref{discrUeqWeak},   with $\chik $ given as a datum,  \EEE  can be solved for $\uk$ by the Lax-Milgram lemma.
	\end{proof}
	\begin{lemma}[Time-discrete   energy-dissipation \EEE inequality]
	\label{lemma:discrEI}
		It holds for all $0\leq \ell\leq k\leq K$
		\begin{equation}
		\label{discrEDI}
		\begin{aligned}
			 & \cE\Big(\uk,\chik,\frac{\uk-\ukk}{\tau}\Big)
			+\tau\sum_{j=\ell}^k\cD\Big(\chij,\frac{\uj-\ujj}{\tau},\frac{\chij-\chijj}{\tau}\Big)
				\\ & \leq \cE\Big(\ul,\chil,\frac{\ul-\ull}{\tau}\Big)
				+ \tau \sum_{j=\ell}^k \langle{\bff_j},{\uj-\ujj} \rangle_{H^1(\Omega)}\,.
		\end{aligned}
		\end{equation}
	\end{lemma}
	\begin{proof}
		Testing \eqref{discrUeqWeak} with $\uj-\ujj$ and \eqref{discrChiEqWeak} with $\chijj$
		and using standard convexity estimates yield
		\begin{subequations}
		\begin{align}
			&\left.
				\begin{aligned}
				& 	\frac12\|\uj\|_{L^2(\Omega)}^2-\frac12\|\ujj\|_{L^2 (\Omega)}^2+  \io \EEE a(\chij) \CC\eps (\uj) {:}\e(\uj{-}\ujj)\dx&\\
					& \quad+\tau\io b(\chij)\VV\frac{\e(\uj){-}\e(\ujj)}{\tau}{:}\frac{\e(\uj){-}\e(\ujj)}{\tau}\dx
					\\
					& \quad  + \tau\int_{\partial\Omega} \gamma_1 \left| \frac{\uj{-}\ujj}{\tau}\right|^2\dS 
					+	\frac{\gamma_2}2\|\uj\|_{L^2(\partial \Omega)}^2-\frac{\gamma_2}2\|\ujj\|_{L^2(\partial \Omega)}^2				\\
					&\leq  \langle{\dis \bff j \tau},{\uj{-}\ujj} \rangle_{H^1(\Omega)} +   \pairing{}{H^{1/2}(\partial\Omega)}{\dis \bfg j \tau}{\uj{-}\ujj},   \EEE
				\end{aligned}\hspace*{4em}
				\right\}
				\label{discrUeqEst}\\
			&\left.
				\begin{aligned}
					&\tau\Big\|\frac{\chij{-}\chijj}{\tau}\Big\|_{L^2}^2+\frac12\|\nabla\chij\|_{L^2}^2-\frac12\|\nabla\chijj\|_{L^2}^2
     \\
					&+\io \left[ \tfrac{  a'(\chij) \EEE
      }{2} \Cm{\ujj}{\ujj} {+} \breve W'(\chij){+}\invbreve W'(\chijj)\right](\chij{-}\chijj)\dx\leq 0.
				\end{aligned}
				\right\}
				\label{discrChieqEst}
		\end{align}
		\end{subequations}
	By the convexity of 
  $\f u \mapsto  \tfrac12 \EEE  \Cm{\f u}{\f u}$  and $a$ \EEE 
		we have
		\begin{align*}
			\io  a(\chij) \Cm{\uj}{\uj-\ujj} \dd x 
			 	\geq{}& \io   \tfrac12 \EEE   a(\chij)  \Cm{\uj}{\uj} \dd x
				\\
				& \qquad \qquad  -\io   \tfrac12 \EEE  a(\chij)  \Cm{\ujj}{\ujj}  \dd x\,,
      \\
		 \io  \tfrac{a'(\chij)}{2} \Cm{\ujj}{\ujj} (\chij{-}\chijj) \dd x 
      \geq {}& \io \tfrac12  (a(\chij){-} a(\chijj) ) \Cm{\ujj}{\ujj}  \dd x \,.
 		\end{align*}
		In the same spirit,
		convexity of $\breve W$ and concavity of $\invbreve W$ yield
		\begin{align*}
			\begin{aligned}
				\io \breve W'(\chij)(\chij{-}\chijj) \dd x 
					\geq{}&\io [ \breve W(\chij){-}\breve W(\chijj)] \dx ,\\
				\io \invbreve W'(\chijj)(\chij{-}\chijj) \dx 
					\geq{}& \io  ]\invbreve W(\chij){-}\invbreve W(\chijj)] \dx\,,
			\end{aligned}\quad
	\end{align*}
	so that 
		\[
				\io \big(\breve W'(\chij){+}\invbreve W'(\chijj)\big)(\chij{-}\chijj) \dd x
				\geq \io  W(\chij) \dd x - \io W(\chijj) \dd x \,.
		\]
		\par
		All in all, adding the inequalities \eqref{discrUeqEst} and \eqref{discrChieqEst},
		applying the above estimates
		and summing over  index
		 $ j \in \{ \ell+1,\ldots,k\}$ \EEE  proves the assertion.
	\end{proof}
	\par
	 From Lemma \ref{lemma:discrEI} we deduce  the basic a priori estimates for the 
	families 
	 $(\olu)_\tau$, $(\ulu)_\tau$, $(\olchi)_\tau$, $(\ulchi)$
	of the (left and right continuous) piecewise constant interpolants	  of the discrete solutions,  as well as for their
	piecewise linear interpolants
	 $(\uu_\tau)_\tau$, $(\chi_\tau)_\tau$. Furthermore, we will consider the interpolants $(\olv)_\tau$, $(\ulv)_\tau$ and $(\vv_\tau)_\tau$ of 
	 the difference quotients $(\vv^k:=\frac{\uk-\ukk}{\tau})_{k=1}^K$, and  the (left continuous) piecewise constant   interpolants
	  $\ointe{\bff}\tau:[0,T] \to H^{1}(\Omega;\R^d)^*$ and  $\ointe{\bfg}\tau:[0,T] \to H^{-1/2}(\partial\Omega;\R^d)$ 
	       of the values $(\dis {\bff} k\tau)_{k=1}^K$ and  $(\dis {\bfg} k\tau)_{k=1}^K$ , respectively. \EEE
	  We record for later use that, as $
	\tau \to 0$, we have 
	  \begin{equation}
	  \label{strong-cvg-ftau}
	  \ointe{\bff}\tau \to \bff \qquad \text{in } L^2(0,T;H^1(\Omega;\R^d)^*), \qquad    \ointe{\bfg}\tau \to \bfg \qquad \text{in } L^2(0,T;H^{-1/2}(\partial\Omega;\R^d)). \EEE
	  \end{equation}
\begin{proposition}[A priori estimates]
\label{prop:aprio}
There exists a constant $S>0$ such that the following estimates hold for all $0<\tau<\bar\tau$
	\begin{subequations}\label{aprioriEst}
		\begin{align}
			&\label{AEST:1} \|\uu_\tau\|_{H^1(0,T;H^1(\Omega;\R^d))\cap W^{1,\infty}(0,T;L^2(\Omega;\R^d))}\leq S,\\
			&\label{AEST:2}  \|\olu\|_{L^\infty(0,T;H^1(\Omega;\R^d))}\leq S,\\
			& \label{AEST:3} \|\ulu\|_{L^\infty(0,T;H^1(\Omega;\R^d))}\leq S,\\
			& \label{AEST:4}\|\chi_\tau\|_{L^\infty(0,T;H^1(\Omega))\cap L^{\infty}(0,T;L^\infty(\Omega))\cap H^1(0,T;L^2(\Omega))}\leq S,\\
			&\label{AEST:5} \|\olchi\|_{L^\infty(0,T;H^1(\Omega))\cap L^{\infty}(0,T;L^\infty(\Omega))}\leq S,\\
			&\label{AEST:6} \|\ulchi\|_{L^\infty(0,T;H^1(\Omega))\cap L^{\infty}(0,T;L^\infty(\Omega))}\leq S,\\ 
			&\label{AEST:8} \|\olv\|_{L^\infty(0,T;L^2(\Omega;\R^d))}\leq S,\\
			&\label{AEST:9} \|\ulv\|_{L^\infty(0,T;L^2(\Omega;\R^d))}\leq S, \\
			&\label{AEST:7} \|\vv_\tau\|_{H^1(0,T;H^{-1}(\Omega;\R^d))}\leq S.\\
		\end{align}
				\end{subequations}
\end{proposition}
	\begin{proof}
	 Clearly, the discrete energy-dissipation inequality \eqref{discrEDI} rephrases as 
	\begin{equation}
	\label{DEDI-interp}
	\begin{aligned}
	 & \mathcal{E}(\olu(t), \olchi(t), \olv(t)) + \int_{\olt(s)}^{\olt(t)} \mathcal{D} \left(\olchi(r), \olv(r), \chi_\tau'(r) \right)  \dd r 
	\\ & \leq 
\mathcal{E}(\olu(s), \olchi(s), \olv(s)) +  \int_{\olt(s)}^{\olt(t)} \langle{\ointe {\bff}\tau(r)},{\olv(r)} \rangle_{H^1(\Omega)} \dd r  
 +  \int_{\olt(s)}^{\olt(t)} \pairing{}{H^{1/2}(\partial\Omega)}{\ointe {\bfg}\tau(r)}{\olv(r)} \dd r    \EEE
	\end{aligned}	
	\end{equation}
	for all $  0 \leq s \leq t \leq T$,
	where $\olt: [0,T]\to [0,T]$  is  the
	(left-continuous) piecewise constant interpolant of the nodes of the partition $(\dis t k\tau)_{k=1}^K$,  with $\olt(0):=0$. 
Taking into account the coercivity properties of $\calE$ and $\calD$ (based on the positive definiteness of the tensors $\bbC$ and 
$\bbV$, on Korn's inequality, 
  on the positivity properties  $a\geq 0$, $b\geq b_0>0$,
 and on the fact that $W$ is bounded from below by 
an  affine function \eqref{affine-W}), from 
\eqref{DEDI-interp} we immediately deduce 
\[
\begin{aligned}
& 
\frac12 \| \olv(t)\|_{L^2}^2 + \frac12 \|\nabla \olchi(t)\|_{L^2}^2 + \int_0^{\olt(t)} \left( \|\chi_\tau'(r)\|_{L^2}^2 {+} c \| \olv(r)\|_{H^1}^2 \right) \dd r 
\\
&
\leq 
\mathcal{E}(\uu_0, \chi_0, \vv_0) + \int_0^{\olt(t)} \langle{\ointe {\bff}\tau(r)},{\olv(r)} \rangle_{H^1(\Omega)} \dd r
 +  \int_0^{\olt(t)} \pairing{}{H^{1/2}(\partial\Omega)}{\ointe {\bfg}\tau(r)}{\olv(r)} \dd r    \EEE
\\
& \quad  +  c_W \| \olchi(t)\|_{L^1} + C_W. 
\end{aligned}
\]
Now, by \eqref{force+data} we gather that $|\mathcal{E}(\uu_0, \chi_0, \vv_0) |\leq C$; 
in turn, we have 
\[
\begin{cases}
\|\ointe {\bff}\tau\|_{L^2(0,T;H^{1}(\Omega;\R^d))^*} \leq \|\bff \|_{L^2(0,T;H^{1}(\Omega;\R^d))^*}
\\
 \|\ointe {\bfg}\tau\|_{L^2(0,T;H^{-1/2}(\partial\Omega;\R^d))} \leq \|\bfg \|_{L^2(0,T;H^{-1/2}(\partial\Omega;\R^d))} \EEE
\end{cases}
\]
   for every $\tau>0$, so that  we may immediately absorb the second integral
term on the right-hand side into the left-hand side. Finally, since 
by construction $0 \leq \olchi \leq 1$ a.e.\ in $Q$, we clearly have $ c_W \| \olchi(t)\|_{L^1} \leq c_W |\Omega|$. 
All in all, we conclude estimates  \eqref{AEST:1}--\eqref{AEST:6} and \eqref{AEST:8}--\eqref{AEST:9}. 
\par
Finally, \eqref{AEST:7} follows from a comparison argument in
 equation \eqref{discrUeqWeak}.
	\end{proof} \EEE
	
\subsection{Conclusion of the proof of Theorem \ref{thm:1}}
\label{ss:3.1}
	 Let us a consider a null sequence $(\tau_j)_j$ of time steps. By well known
compactness theorems we find  a pair $(\uu,\chi)$ as in \eqref{regs-u-chi} and a (not relabeled) subsequence of $(\tau_j)_j$ \EEE
		such that 		 the following 
		convergences hold as $j\to \infty$
		\begin{subequations}
		\label{weak-cvg-sez3}
		\begin{align}
			\uu_{\tau_j}\weakstarlim{}& \uu
				&&\text{ weakly-star in }H^1(0,T;H^1(\Omega;\R^d))\cap W^{1,\infty}(0,T;L^2(\Omega;\R^d)),\\
			\ointe {\uu}{\tau_j},\uinte {\uu}{\tau_j}\weakstarlim{}& \uu
				&&\text{ weakly-star in }L^{\infty}(0,T;H^1(\Omega;\R^d)),\\
				\label{weak-cvg-sez3-3}
			\vv_{\tau_j}\weaklim{}& \partial_t\uu
				&&\text{ weakly in }H^1(0,T;  H^{1}(\Omega;\R^d)^*), \EEE \\
					\label{weak-cvg-sez3-4}
			\ointe {\vv}{\tau_j},\uinte {\vv}{\tau_j},  \linte{\vv}{\tau_j} \EEE \weakstarlim{}& \partial_t\uu
				&&\text{ weakly-star in }L^\infty(0,T;L^2(\Omega;\R^d)),\\
			\chi_{\tau_j}\weakstarlim{}& \chi
				&&\text{ weakly-star in }L^\infty(0,T;H^1(\Omega))\cap L^{\infty}(0,T;L^\infty(\Omega))\cap H^1(0,T;L^2(\Omega)),\\
			\ointe {\chi}{\tau_j},\uinte {\chi}{\tau_j}\weakstarlim{}& \chi
				&&\text{ weakly-star in }L^\infty(0,T;H^1(\Omega))\cap L^{\infty}(0,T;L^\infty(\Omega))\,.
		\end{align}
		\end{subequations}
		From Aubin-Lions type compactness results we see that
		\begin{equation}
		\label{strong-cvg-chi}
		\begin{aligned}
			\chi_{\tau_j},\ointe {\chi}{\tau_j},\uinte {\chi}{\tau_j}\to{}& \chi
				\quad\text{ strongly in }L^\infty(0,T;L^p(\Omega))\text{ for all }p\in(0,2^*),\\
			\chi_{\tau_j},\ointe {\chi}{\tau_j},\uinte {\chi}{\tau_j}\to{}& \chi
				\quad\text{ a.e. in }Q.
		\end{aligned}
		\end{equation}
		 Finally, combining \eqref{weak-cvg-sez3-3} \& \eqref{weak-cvg-sez3-4} we also gather 
		\begin{equation}
		\label{pointwise-v}
		 \linte{\vv}{\tau_j}(t) \weakto \vv(t)=\partial_t \uu (t)  \qquad \text{ in } L^2(\Omega;\R^d) \qquad \text{for all } t \in [0,T]\,.
		\end{equation}
	\par
	Obviously, the pair $(\uu,\chi)$ complies with initial conditions \eqref{initialCond}, constraints \eqref{constraints}.
In order to check the validity of  \eqref{weakUEq}, 
		let us write the discrete weak momentum balance
		\eqref{discrUeqWeak}
		 in a time-integrated version:
		 \[
		\begin{aligned}
			&\intt\langle\partial_t\vv_{\tau_j},\ph\rangle_{H_0^1}\dt+
			\int_0^t \int_\Omega \left\{ 
			b(\ointe {\chi}{\tau_j})\VV\e(\partial_t\uu_{\tau_j}){:}\e(\ph){+} a(\ointe {\chi}{\tau_j}) \Cm{\ointe {\uu}{\tau_j}}{\ph}  \right\} \dx \dd r
			\\
			&  \quad + \int_0^t \int_{\partial\Omega} \left\{ \gamma_1
			\partial_t\uu_{\tau_j}
			{\cdot}\ph{+} \gamma_2 \ointe {\uu}{\tau_j}  {\cdot}\ph  \right\} \dd S \dd r
			\\
			 &  =  
			 \int_0^t \langle{\ointe {\bff}{\tau_j}},{\partial_t \linte {\uu}{\tau_j}} \rangle_{H^1(\Omega)} \dd r  +
			  \int_0^t \pairing{}{H^{1/2}(\partial\Omega)}{\ointe {\bfg}{\tau_j}}{\partial_t \linte {\uu}{\tau_j}} \dd r. \EEE 
		\end{aligned}
		\]
		for all   $\ph\in L^\infty(0,T;H^1(\Omega;\R^d))$. \EEE
	Convergences
	\eqref{weak-cvg-sez3}, \eqref{strong-cvg-chi}, and 
	 \eqref{strong-cvg-ftau}
	allow us to
		take the limit as $\tau_j\to 0$, and conclude the time-integrated version of  \eqref{weakUEq}.  Hence, the weak momentum balance is shown.
		\par
		In the next step we aim to obtain the integral inequality \eqref{weakChiIneq}.   For this, we observe that,
		 choosing the admissible test-function $\psi=\ointe \chi{\tau_j}(t)+\hpsi$
		with $\hpsi\in H_-^1(\Omega)\cap L^\infty(\Omega)$ 
		 (where 	$H_-^1(\Omega)$ is the cone of negative functions in $H^1(\Omega)$), \EEE
		 \eqref{discrChiEqWeak} rewrites  for almost all $t\in (0,T)$ as 
		 \[
		 \begin{aligned}
		&\io \left\{ \linte\chi{\tau_j}'(t) \hpsi{+} \nabla\ointe\chi{\tau_j}(t) {\cdot}\nabla \hpsi
						{+} 
       \frac12 a'(\linte\chi{\tau_j}(t))
      \Cm{\uinte \uu{\tau_j} (t)}{\uinte u{\tau_j}(t)}\hpsi \right\}\dx
      \\
					&\hspace*{13.5em}+ \io\left\{  \breve{W}'(\ointe\chi{\tau_j}(t)) \hpsi {+}\invbreve{W}'(\uinte\chi{\tau_j}(t))\hpsi \right\} \dx\geq 0\,.
					\end{aligned}
					\]
		Thus, integrating in time we obtain
		\begin{align}\notag
			&\iint_Q \left[ \partial_t \linte \chi{\tau_j}
			\hpsi{+}\nabla\ointe {\chi}{\tau_j}{\cdot}\nabla\hpsi{+}\frac{
            {a}'(\ointe {\chi}{\tau_j})
            }{2} \Cm{\uinte {\uu}{\tau_j}}{\uinte {\uu}{\tau_j}}
  \hpsi {+}   \breve{W}' (\ointe {\chi}{\tau_j}) \hpsi{+} \invbreve W'(\uinte {\chi}{\tau_j})\hpsi \right]    \dxt\geq 0
		\end{align}
		for all $\hpsi\in L^\infty(0,T;H_-^1(\Omega))\cap L^\infty(0,T;L^\infty(\Omega))$.
		In order to take the limit as $j\to\infty$ we rely on convergences \eqref{weak-cvg-sez3}  observe that, by \eqref{strong-cvg-chi} and the fact that $W'\in \mathrm{C}^0(\R)$, we immediately have, for instance, that
		\[
		  \breve{W}' (\ointe {\chi}{\tau_j})  + \invbreve W'(\uinte {\chi}{\tau_j}) \to \breve{W}'(\chi)+  \invbreve W'(\chi) \qquad \text{in } L^\infty(0,T;L^2(\Omega))\,.
	\]
Moreover, we combine the information that
	\[
   \e(\uinte {\uu}{\tau_j})\weaklim{} 
   \e(\uu)\quad
				\text{ weakly in }L^2(0,T;L^2(\Omega;\R^{d\times d}))
			\]
			with the fact that $a'(\ointe {\chi}{\tau_j})\to a'(\chi)$, e.g. in $L^\infty(0,T;L^2(\Omega))$.
			Then,  well-known lower semicontinuity results
			 (cf.\ the Ioffe theorem \cite{Ioff77LSIF}) \EEE
			 give 		\begin{equation*}
			\liminf_{j\to \infty}-\iint_Q  \frac{
   {a}'(\ointe {\chi}{\tau_j})
   }{2}\Cm{\uinte {\uu}{\tau_j}}{\uinte {\uu}{\tau_j}}\hpsi\dxt
				\geq{} \int_Q \frac{a'(\chi)}{2}\CC\e(\uu):\e(\uu)(-\hpsi)\dxt\,.
		\end{equation*}
	In this way, we conclude the time-integrated version of the one-sided variational inequality \eqref{weakChiIneq}, tested by
	an arbitrary  $\hpsi\in L^\infty(0,T;H_-^1(\Omega))\cap L^\infty(0,T;L^\infty(\Omega))$. Ultimately, we conclude 
	 \eqref{weakChiIneq}.
	 \par
	 Eventually,  also relying on the pointwise-in-time convergence \eqref{pointwise-v}, we are in a position to 
	take the limit as $\tau_j\downarrow 0$ in the discrete energy-dissipation inequality \eqref{DEDI-interp}, written for $s=0$ and arbitrary $t \in [0,T]$.
We thus conclude 
		the time-continuous energy inequality \eqref{UEDI}. 
	\QED
	\EEE
\subsection{Outlook to nonsmooth potential energies}
\label{ss:3-outlook} 
In this section, we provide a possible extension of the existence 
 result for weak solutions, to the case in which the convex part $\breve W$  of  potential energy $W$  is  nonsmooth. A prototypical example will be provided by 
$
 \breve W  = I_{[0,\infty)}
$, cf.\ Remark \ref{rmk:comparison_HK} below,
but we will indeed allow for more general potentials. 
Our standing assumptions are collected in the following
\begin{hypx}\label{h:5}
The elasticity and the viscosity tensors, and the constitutive functions $a$ and $b$, comply with \eqref{asstensors}, 
\eqref{ass-a}, and \eqref{ass-b}, respectively. We suppose that $W$ is $\ell$-convex, \textit{i.e.},~\eqref{assW-2-new} holds, and $W$ fulfills~\eqref{assW-3-new}. 

\end{hypx}
Note that in comparison to the assumption~\eqref{assW-new} of Hypothesis~\textbf{\ref{h:1}}, we relaxed the smoothness assumptions~\eqref{assW-1-new} on the convex part $\breve W$.
To handle nonsmoothness of  $\breve W$, \EEE  we  propose a novel generalized formulation
which turns out to be consistent with that of 
Definition \ref{def:weakSol}. In fact, in Def.\ \ref{def:weaknonsmooth} below inequality
\eqref{weakChiIneq} is replaced by another one-sided variational inequality, \eqref{weakChiIneq_3} below,
featuring  the derivative of the concave part, only.  On the other hand, \eqref{weakChiIneq_3} is in the same spirit as the one-sided variational inequality 
proposed for the damage flow rule in \cite{HK-1,HK-2} in the specific case $W= \breve W =  I_{[0,\infty)}$. 
\begin{definition}[Weak solution for nonsmooth potential]
\label{def:weaknonsmooth}
 In the setting of Hypothesis \textbf{\ref{h:5}},
we call a pair $(\f u , \chi)$ 
 as in \eqref{regs-u-chi}
a weak solution to~\eqref{PDEsystem}, if it satisfies 
 initial conditions \eqref{initialCond}, constraints \eqref{constraints}, the weak formulation \eqref{weakUEq} of the momentum balance, 
  the  upper energy-dissipation inequality   \eqref{UEDI}, and 
  \EEE
      \begin{align}
\begin{split}
    \iint_Q \left( \chi_t\varphi {+}\nabla\chi{\cdot}\nabla\varphi {+} \tfrac12  a'(\chi) \varphi  \CC \e(\uu) {:} \e(\uu){+}\breve W(\chi{+}\varphi ){-}\breve W(\chi){+} \invbreve W'(\chi) \varphi   \right) \dx\dt
         \geq 0
									\label{weakChiIneq_3}
\end{split}
						\end{align}
      for all   $\varphi \in L^\infty(0,T; H^1(\Omega)) \cap L^\infty(Q)$ with $\varphi \leq 0 $ a.e.~in $Q$. \EEE
\end{definition}
\begin{remark}
\label{rmk:comparison_HK}
\upshape
In the specific case in which $W = \breve W =  I_{[0,\infty)}$, \eqref{weakChiIneq_3} reduces to 
\begin{equation}
\label{particular-case-weak3}
   \iint_Q \left( \chi_t\varphi {+}\nabla\chi{\cdot}\nabla\varphi {+} \tfrac12  a'(\chi) \varphi  \CC \e(\uu) {:} \e(\uu) \right) \dx\dt \geq 0
\end{equation}
for all $\varphi \in L^\infty(0,T; H^1(\Omega))$ with $-\chi \leq \varphi \leq 0 $ a.e.\ in $Q$ (where the constraint $\varphi \geq -\chi$ ensures that 
$\iint_Q \breve W(\chi+\varphi )\dxt <+\infty$ and thus the inequality is non-trivial). 
In fact, 
 \eqref{particular-case-weak3} is  in accord 
with the one-sided variational inequality from  \cite[Def.\ 4.5]{HK-1},  \cite[Def.\ 2.3]{HK-2}.
\end{remark}
  Our first result shows that, as soon as  also the convex part of 
 $W$  is smooth, any weak solution in the sense of Definition \ref{def:weaknonsmooth} is also a weak solution according to Definition~\ref{def:weakSol}. 
 \begin{proposition}
 In addition to Hypothesis \ref{h:5}, suppose that $\breve W\in \rmC^1(\R)$. Then, any  $(\uu,\chi)$ as in  \eqref{regs-u-chi} fulfilling \eqref{weakChiIneq_3}
 also complies with \eqref{weakChiIneq}.
 \end{proposition}
  \begin{proof}
We choose the test function for \eqref{weakChiIneq_3} in the form 
\[
 \varphi = \phi \zeta \in L^\infty(0,T; H^1(\Omega)) \quad \text{ with }
 \begin{cases}
  \phi \in L^\infty(0,T), 
\quad 0 \leq \phi \leq 1 \text{ a.e.\ in } (0,T),
\\
 \zeta \in H^1(\Omega), \quad  \zeta \leq 0  \text{ a.e.\ in } \Omega.
 \end{cases}
 \]
 We now estimate from above the term $\breve W (\chi {+} \phi \zeta ) - \breve W(\chi)$  that features on the left-hand side of 
 \eqref{weakChiIneq_3} by \EEE 
 \[
\begin{aligned}
\breve W (\chi {+} \phi \zeta ) - \breve W(\chi) & =  \breve W\EEE ( \phi ( \chi {+} \zeta) {+} ( 1{-}\phi) \chi) - \breve W(\chi)
\\
& \stackrel{(1)}\leq  \phi \breve W (\chi {+} \zeta ) + (1{-}\phi) \breve W(\chi)  - \breve W(\chi)
\\ & = 
    \phi \left(  \breve W (\chi {+} \zeta ) {-} \breve W ( \chi) \right) \qquad \text{a.e.\ in } Q\,,
\end{aligned}
\]
 where {\footnotesize (1)} follows from the convexity of $\breve W$. We use   the above inequality to estimate from above the 
 left-hand side of 
 \eqref{weakChiIneq_3}, thus obtaining 
 \[
     \int_0^T \phi \left(\int_\Omega \chi_t \zeta  {+}\nabla\chi{\cdot}\nabla\zeta {+} \tfrac12  a'(\chi) \zeta  \CC \e(\uu) {:} \e(\uu){+} \breve W (\chi {+} \zeta ) {-} \breve W ( \chi) +\invbreve W'(\chi) \zeta   \dx \right)\dt
         \geq 0\,.
         \]
By the arbitrariness of $\phi$  we thus infer the pointwise in time formulation
\begin{equation}
				\begin{aligned}
					&\io \left( \chi_t(t)\zeta{+}\nabla\chi(t){\cdot}\nabla\zeta{+} \tfrac12 a'(\chi(t))    \CC \e(\uu(t)) {:} \e(\uu(t))\zeta {+} \invbreve W'(\chi(t))  \zeta \right) \dx
					\\
					& 
					+\int_\Omega ( \breve W(\chi(t){+}\zeta){-}\breve W(\chi(t)) ) \dd x \geq 0
						\end{aligned}
						\end{equation}
for all $\zeta\in H_-^1(\Omega)\cap L^\infty(\Omega)$ 
and for almost all  $t\in(0,T)$. 
Let us now choose  $\zeta =  h \psi $ with an arbitrary  $\psi \in  H_-^1(\Omega)\cap L^\infty(\Omega) $. 
 Dividing the resulting inequality by $h$ and sending   $h\to 0$  we obtain
  \eqref{weakChiIneq}.
  \end{proof}

  We now show that, with minor changes, the argument developed in Secs.\ \ref{ss:3.0}--\ref{ss:3.1} also yields the existence of solutions in the sense of
  Def.\ \ref{def:weaknonsmooth}.

\begin{theorem}[Existence of weak solutions for nonsmooth potentials]
Assume Hypotheses \textbf{\ref{h:2}} \& \textbf{\ref{h:5}}. 
    Then, there exists a weak solution in the sense of Definition~\ref{def:weaknonsmooth} to the Cauchy problem for system \eqref{PDEsystem}. 
\end{theorem}
\EEE
\begin{proof}
    The proof is very similar to the argument for  Theorem~\ref{thm:1}, hence we only comment on the relevant changes. 
     We construct time-discrete solutions $(\uk,\chik)_{k=1}^K$ as in Lemma \ref{lemma:existenceDiscrSol}.
    From the information that $\chik$ is a minimizer for the functional $\mathcal{P}$, cf.\ \eqref{P-functional},  as a first order optimality condition \EEE we gather that     
%
\begin{align}\label{nonsmoothchi}
\begin{split}
    					&\io \left\{ \frac{\chik{-}\chikk}{\tau}(\psi{-}\chik){+} \nabla\chik{\cdot}\nabla(\psi{-}\chik) 
       {+} \frac12 a'(\chik) 
      \Cm{\ukk}{\ukk}(\psi{-}\chik) \right\}\dx
      \\
					&\hspace*{13.5em}+ \io\left\{  \breve W (\psi) {-} \breve{W}(\chik) {+}\invbreve{W}'(\chikk)(\psi{-}\chik) \right\} \dx\geq 0\,,
				\end{split}
\end{align}
				{for all $\psi\in X_\tau^{k-1}:=\big\{v\in H^1(\Omega)\cap L^\infty(\Omega):\;v\leq \chikk \aein\Omega\big\}$, as well as the constraints}
				$0 \leq \chik \leq \chikk  \leq 1 $ a.e. in $ \Omega$. 
 The discrete energy inequality of Lemma~\ref{lemma:discrEI}
is then obtained by testing  \eqref{discrUeqWeak} with $\uj-\ujj$ and \eqref{nonsmoothchi} with $\chijj$; from 
\eqref{discrEDI} there stem the a priori estimates of Proposition \ref{prop:aprio} and, a fortiori, 
convergences \eqref{weak-cvg-sez3}--\eqref{pointwise-v}.
\par
In order to prove  \eqref{weakChiIneq_3},  first of all we sum  \eqref{nonsmoothchi}  
 over all  the time intervals
 induced by the partition, 
thus obtaining \EEE
	\begin{align}
		\label{varIneq}
  \begin{split}
      \iint_Q\partial_t \linte \chi{\tau_j}
			(\widehat{\psi}
   {-}\ointe\chi{\tau_j}  ) +\nabla\ointe {\chi}{\tau_j}{\cdot}\nabla(\widehat{\psi}
   {-}\ointe\chi{\tau_j})+\frac{
            {a}'(\ointe {\chi}{\tau_j})
            }{2} \Cm{\uinte {\uu}{\tau_j}}{\uinte {\uu}{\tau_j}}
  (\widehat{\psi}
  {-}\ointe\chi{\tau_j}) \dxt& \\ +\iint_Q  \breve W(\widehat{\psi}
  ) {-} \breve{W} (\ointe {\chi}{\tau_j}) + \invbreve W'(\uinte {\chi}{\tau_j})(\widehat{\psi}
  {-}\ointe\chi{\tau_j})   \dxt&\geq 0\,,
  \end{split}
		\end{align}
 for all test functions $ \widehat{\psi} \in L^\infty( 0,T; H^1(\Omega)) $ with $ \widehat\psi \leq\ointe {\chi}{\tau_j} $ a.e. in $Q$. 
With the convergences inferred in~\eqref{weak-cvg-sez3}, we may pass to the limit is the above formulation. 

For any $ \psi \in L^\infty(0,T;H^1(\Omega)) \cap L^\infty(Q ) $ with $\psi \leq \chi $ 
we construct a sequence of recovery test functions in the following way.
For any $j\in \N$
we define 
\[
 \psi_j (x,t):= \min \{ \psi(x,t) , \ointe {\chi}{\tau_j} (x,t) \} \quad \text{and} \quad 
  A_j := \{ (x,t) \in Q \mid   \psi(x,t) \leq \ointe {\chi}{\tau_j} (x,t) \} \,.
  \]
  With the  same arguments as in 
  \cite[Thm.\ 3.14]{ThoMie09DNEM} we can prove that
  \[
  (\psi_j)_j \text{ is bounded in }  L^\infty(0,T;H^1(\Omega))\cap L^\infty(Q ) \text{ and } \psi_j(t) \weakto \psi(t) \text{ in } H^1(\Omega) \ \foraa\,  t \in (0,T)\,,
  \]
  so that it is not difficult to deduce that 
  \[
  \psi_j \to \psi \quad \text{in }L^r(Q) \text{ for all } r \in [1,\infty)\,, \qquad
   \psi_j \weakstarlim{} \psi \quad \text{weakly-star in }L^\infty(0,T;H^1(\Omega))\cap L^\infty(Q )\,.
  \]
  We now choose $\widehat \psi = \psi_j $ in \eqref{varIneq}. Since 
  $\psi_j = \linte \chi{\tau_j}$ on $Q{\setminus} A_j$, we thus obtain
  \begin{align}\notag
  \begin{split}
      \iint_Q \mathds{1}_{A_j} \left(\partial_t \linte \chi{\tau_j}
			(\psi
   {-}\ointe\chi{\tau_j}  ) {+}\nabla\ointe {\chi}{\tau_j}{\cdot}\nabla(\psi
   {-}\ointe\chi{\tau_j}){+}\frac{
            {a}'(\ointe {\chi}{\tau_j})
            }{2} \Cm{\uinte {\uu}{\tau_j}}{\uinte {\uu}{\tau_j}}
  (\psi
 {-}\ointe\chi{\tau_j}) \right) \dxt& \\ +\iint_Q  \mathds{1}_{A_j} \left( \breve W(\psi
  ){-} \breve{W} (\ointe {\chi}{\tau_j}) {+} \invbreve W'(\uinte {\chi}{\tau_j})(\psi
  {-}\ointe\chi{\tau_j})  \right)  \dxt&\geq 0\,,
  \end{split}
		\end{align}
		and then send $j\to\infty$ .
We use that, since $\ointe {\chi}{\tau_j}  \to \chi $
  and $\psi \leq \chi$  a.e.\ in $Q$, the sequence $(\mathds{1}_{A_j})_j$ of the characteristic functions of the sets $A_j$ converges a.e.\  in $Q$
  and strongly in $L^1(Q)$
  to the function identically  equally to $1$. Therefore, 
  \[
  \begin{aligned}
  & 
   \iint_Q \mathds{1}_{A_j} \partial_t \linte \chi{\tau_j}
			(\psi {-}\ointe\chi{\tau_j}  ) \dxt && \longrightarrow &&  \iint_Q  \partial_t \chi
			(\psi {-}\chi  ) \dxt,
\\
& 	   \iint_Q \mathds{1}_{A_j}	\nabla\ointe {\chi}{\tau_j}{\cdot}\nabla \psi  \dxt  && \longrightarrow &&  \iint_Q  \nabla\chi{\cdot}\nabla \psi  \dxt,
\\
&
\iint_Q  \mathds{1}_{A_j} \left( \breve W(\psi
  )  {+} \invbreve W'(\uinte {\chi}{\tau_j})(\psi
  {-}\ointe\chi{\tau_j})  \right)  \dxt  && \longrightarrow &&\iint_Q \left( \breve W(\psi
  )  {+} \invbreve W'(\chi)(\psi
  {-}\chi)  \right)  \dxt \,,
  \end{aligned}
  \]
  where we have also used that $\invbreve{W}(r) = -\frac \ell 2 r^2$. 
To handle the remaining terms, we again resort to the Ioffe theorem  \cite{Ioff77LSIF}, which gives
\begin{align*}
&
   \limsup_{j\to\infty} \left( {-} \iint_Q  \mathds{1}_{A_j} |\nabla \ointe {\chi}{\tau_j}|^2 \dxt \right)= - \liminf_{j\to\infty}  \iint_Q  \mathds{1}_{A_j} |\nabla \ointe {\chi}{\tau_j}|^2 \dxt 
   \leq  - \iint_Q | \nabla {\chi}|^2 \dxt \,,
   \\
   & 
   \limsup_{j\to\infty} \left( {-}\iint_Q \mathds{1}_{A_j} \breve{W} (\ointe {\chi}{\tau_j}) \dxt \right)  \leq -  \iint_Q \breve W (\chi) \dxt\,,
\end{align*}
and, likewise 
\[
\begin{aligned}
&
 \limsup_{j\to\infty} \iint_Q  \mathds{1}_{A_j} 
            \frac12 {a}'(\ointe {\chi}{\tau_j})
            \Cm{\uinte {\uu}{\tau_j}}{\uinte {\uu}{\tau_j}}
  (\psi
  {-}\ointe\chi{\tau_j})  \dxt 
  \\
   &  =- \liminf_{j\to\infty} \iint_Q  \mathds{1}_{A_j} 
           \frac12  {a}'(\ointe {\chi}{\tau_j})
             \Cm{\uinte {\uu}{\tau_j}}{\uinte {\uu}{\tau_j}}
  (\ointe\chi{\tau_j}
  {-}\psi)   \dxt  \\ & \leq 
  -\iint_Q  \mathds{1}_{A_j} 
            \frac12 {a}'(\chi)
            \Cm{\uu}{\uu}
  (\chi
  {-}\psi)   \dxt \,.
  \end{aligned}
\]
To apply Ioffe's theorem, here, 
 we have also relied on the fact that $( \breve{W} (\ointe {\chi}{\tau_j}))_j$ is bounded from below and 
     $ a'(\ointe {\chi}{\tau_j})( \ointe {\chi}{\tau_j}{-}\psi) \geq 0$  a.e.\ in $Q$.
%
  Combining all these convergences, we arrive at
  \begin{equation}
  \begin{aligned}
    \iint_Q &  \left( \chi_t(\psi{-}\chi) {+}\nabla\chi{\cdot}\nabla(\psi-\chi)  {+} \tfrac12  a'(\chi)(\psi{-}\chi  ) \CC \e(\uu) {:} \e(\uu){+} \invbreve W'(\chi) (\psi {-}\chi)   \right) \dx\dt\\
    & \qquad \qquad \qquad 
					+\iint_Q ( \breve W(\psi ){-}\breve W(\chi) ) \dd x\dt \geq 0
									\label{weakChiIneq_3-proof}
						\end{aligned}
						\end{equation}
      for all $\psi \in L^\infty(0,T; H^1(\Omega))$ with $\psi \leq \chi $ a.e.~in $Q$. 
            Choosing $ \varphi = \psi-\chi $, we   get \eqref{weakChiIneq_3}. 
\end{proof}

\EEE

\section{Proof of Theorem \ref{thm:2}}
\label{s:proof-thm2}
Our proof of the enhanced regularity \eqref{reg-strong-sols} will be based on  estimates that have a local-in-time character only and 
rely on   a Gronwall-type argument. \EEE
Since there is, apparently, no time-discrete version of  local-in-time Gronwall estimates, we will not resort to time discretization  for proving the existence of 
strong solutions, but instead 
\begin{compactitem}
\item devise a suitable approximation of system \eqref{PDEsystem}, namely  system \eqref{galerkinSchauder} ahead,
\item prove existence of solutions to \eqref{galerkinSchauder}  via the Schauder fixed point theorem,
\item perform on it the rigorous regularity estimates.  
\end{compactitem}
\par
Such regularity estimates will be  at first \emph{formally} performed on the original system  \eqref{PDEsystem} in Sec.\ \ref{ss:4.1} below. This will allow us to pinpoint 
how system \eqref{PDEsystem} needs to be approximated in such a way that the calculations of 
 Proposition \ref{prop:formal-aprio} can be rendered rigorous. Hence, in Sec.\ \ref{ss:4.2} we will  set up the approximate system \eqref{galerkinSchauder}
 by combining Galerkin discretization and Yosida regularization. In Sections \ref{ss:4.3}  and \ref{ss:4.3bis} \EEE  we will address
 the existence of  local-in-time solutions to the associated Cauchy problem, 
 and rigorously perform the, previously formal, enhanced regularity estimates.  \EEE
Finally, in Sec.\ \ref{ss:4.4}  we will conclude the proof of Thm.\  \ref{thm:2} by taking the limit in system  \eqref{galerkinSchauder}. 
\subsection{Formal a priori estimates}
\label{ss:4.1}
 Before carrying out the enhanced a priori estimates, it is 
convenient to  rewrite the flow rule \eqref{chiEq} as 
\begin{equation}
\label{rewritten-flow-rule}
\begin{aligned}
\chi_t+I'_{(-\infty,0]}(\chi_t)+\omega & = \chi - \invbreve{W}'(\chi) -\frac12 a'(\chi) \e(\uu) \CC \e(\uu) \qquad \text{a.e.\ in } Q
\\
& \text{with } \omega = -\Delta\chi + \breve{W}'(\chi) +\chi\,,
\end{aligned}
\end{equation}
where we have formally replaced $\partial I_{(-\infty,0]}(\chi_t)$ by $ I_{(-\infty,0]}'(\chi_t)$, hereafter abbreviated as $I'(\chi_t)$, 
and
resorted to the convex/concave decomposition \eqref{cvx-cnc-as-conseq}  of $W$. 
 Although in the present setting the convex contribution $\breve W$ may be nonsmooth, for notational simplicity
 we will formally write 
 $\breve{W}'(\chi)$,  $\breve{W}''(\chi)$. 
 In fact, in Section \ref{ss:4.3} ahead we will make all estimates rigorous by replacing $\breve W$ by a (version of) its Yosida regularization. 
 \par
 The following result  
 (with the caveat 
 that all calculations can be rendered rigorously when    $\breve W$   is suitably regularized)
 collects elementary estimates  that will nonetheless have a key role in the ensuing calculations; note that 
 \eqref{est-omega-BS} is in the spirit of the well-known Brezis-Strauss result \cite{BreStra73}.  In the proof we will  use that $\breve W'(0)=0$, which we can always suppose up
 to a translation.
 \begin{lemma}
 \label{l:props-omega}
 There exists $S_0>0$ such that  for almost all $t\in (0,T)$
 \begin{subequations}
 \label{est-omega}
 \begin{align}
 &
  \label{est-omega-BS}
 \left(\|\chi(t)\|_{H^2(\Omega)}{+} \|\breve{W}'(\chi(t))\|_{L^2(\Omega)}  \right) \leq S_0 \|\omega(t)\|_{L^2(\Omega)}\,,
\\
 &
  \label{est-omega-dert}
\|\chi_t(t)\|_{H^1(\Omega)} \leq  S_0   \|\omega_t(t)\|_{L^2(\Omega)}\,,
\\
 &
  \label{est-omega-dertt}
 \|\chi_{tt}(t)\|_{H^1(\Omega)} \leq  S_0   \left( \|\omega_{tt}(t)\|_{L^2(\Omega)}  + \|\breve{W}'''(\chi(t))\|_{L^\infty(\Omega)} \|\chi_t(t)\|_{L^3(\Omega)}^2 \right)\,.  \EEE
\end{align}
 \end{subequations}
 \end{lemma}
 \begin{proof}
 $\vartriangleright \eqref{est-omega-BS}:$
 We calculate 
\begin{align*}
    \| \omega(t)\|_{L^2(\Omega)}^2 &= \| \Delta \chi (t)\|_{L^2(\Omega)}^2 + \| \breve W ' (\chi(t))  \|_{L^2(\Omega)}^2
    + 
   \| \chi(t) \|_{L^2(\Omega)}^2 
   \\
   & \qquad  +2 \int_\Omega\breve W ' (\chi(t)) \chi(t) \dd x 
 - 2 \int_\Omega \Delta \chi(t) (\breve W ' (\chi(t)) + \chi(t)) \dx 
    \\
    & \stackrel{(1)}{\geq} 
    c \|  \chi(t) \|_{H^2(\Omega)}^2 + \| \breve W ' (\chi(t))\|_{L^2(\Omega)}^2 +
     2 \int_\Omega  ( \breve  W'' (\chi(t)) + 1)| \nabla \chi(t)|^2 \dx  \\
     & \stackrel{(2)}{\geq}  c \|  \chi(t) \|_{H^2(\Omega)}^2 + \| \breve W ' (\chi(t))\|_{L^2(\Omega)}^2\,,
  \EEE
\end{align*}
 where {\footnotesize (1)} follows the fact that $\breve W'(0)=0$ and 
 {\footnotesize (2)}   
 from the convexity of $\breve W$.
\par\noindent 
 $\vartriangleright \eqref{est-omega-dert}:$ Differentiating in time $ \omega = -\Delta\chi + \breve{W}'(\chi) +\chi$ and testing it by $\chi_t$ we obtain
 \[
\|\chi_t\|_{H^1(\Omega)}^2   \leq   \int_\Omega \left( |\chi_t|^2{+} |\nabla\chi_t|^2\right) \dd x +  \int_\Omega \breve{W}''(\chi) |\chi_t|^2 \dd x 
= \int_\Omega \omega_t\chi_t \dd x \leq \frac12 \|\omega_t\|_{L^2(\Omega)}^2 +\frac12 \|\chi_t\|_{L^2(\Omega)}^2\,,
 \]
whence \eqref{est-omega-dert}.
\par\noindent 
  $\vartriangleright \eqref{est-omega-dertt}:$  
We differentiate $ \omega = -\Delta\chi + \breve{W}'(\chi) +\chi$  twice  in time and test it by $\chi_{tt}$,  thus obtaining
 \[
 \begin{aligned}
 \|\chi_{tt}\|_{H^1(\Omega)}^2    & 
 \leq   \int_\Omega \left( |\chi_{tt}|^2{+} |\nabla\chi_{tt}|^2\right) \dd x +  \int_\Omega \breve{W}''(\chi) |\chi_{tt}|^2 \dd x 
 \\
 &  = \int_\Omega \omega_{tt}\chi_{tt} \dd x + \int_\Omega \breve{W}'''(\chi)|\chi_t|^2  \chi_{tt}  \dd x 
 \\
 & 
 \leq \|\omega_t\|_{L^2(\Omega)}^2 + \|\breve{W}'''(\chi) \|_{L^\infty(\Omega)} \| |\chi_t|^2\|_{L^{3/2}(\Omega)}^2 + \frac12 \|\chi_{tt}  \|_{L^3(\Omega)}^2\,,
 \end{aligned}
 \]
 whence \eqref{est-omega-dertt}\,.
\EEE 
 \end{proof}

\begin{proposition}
\label{prop:formal-aprio} 
Assume  Hypotheses {\bf \ref{h:1-strong}} \& {\bf \ref{h:2-strong}},  let $\Omega$ fulfill \eqref{omega-smooth}. 
Then,
there exists a  time  $\widehat{T} \in (0,T]$ and a constant $S_1>0$ such that for all $t\in (0,\widehat{T})$
      \begin{align}\label{eq:proploc}
    \begin{split}
                & \| \f u_t(t) \|_{H^2(\Omega)}^2+  
  \| \chi(t) \|_{H^2(\Omega)}^2     +\int_0^t \left(  \| \f u_t (s)\|_{H^3(\Omega)}^2{+} \|  \chi_t(s) \|_{H^1(\Omega)}^2 \right) \ds 
 \\
& \leq{}
S_1 
 \Big( 1{+}
  \|\vv_0   \|_{H^2(\Omega)}^2  
 {+}   \| \f u_0 \| _{H^3(\Omega)}^2  {+}  \| \chi_0\|_{H^2(\Omega)}^2{+}  \| |\partial \breve{W}^{\circ}|(\chi_0) \|_{L^2(\Omega)}^2
 {+} \| \uu_0\|_{H^2(\Omega)}^8 \Big) \,. \EEE
      \end{split}
  \end{align}
  Furthermore,  there exists a constant $\widehat{S}_1>0$ such that \EEE
\begin{equation}
\label{comparisonH2}
\| \uu_{tt} \|_{L^2 (0,\widehat{T}; H^1(\Omega))} \leq \widehat{S}_1 \,. \EEE
\end{equation}
\end{proposition}
 \renewcommand{\Vm}{\mathbb{C}}
 \par 
 Throughout the proof, we will repeatedly use the following estimate  for all $ t \in [0,T]$
\begin{equation}
\label{2TFC}
 \| \zz(t) \|_{\boldsymbol{X}}^p =\left \| \zz(0)+ \int_0^t \zz_t \ds  \right \|_{\boldsymbol{X}}^p \leq   2^{p-1}   \|  \zz(0) \| _{\boldsymbol{X}}^{p} +  
 (2t)^{p-1}  \int_0^t \|  \zz_t(s) \| _{\boldsymbol{X}}^{p}  \dd s\,,  
\end{equation} 
which follows, by Jensen's inequality and the elementary inequality 
$(a{+}b)^p \leq 2^{p-1} (a^p{+}b^p)$ for all $a,b\in [0,+\infty)$, 
for every $\zz  \in W^{1,p} (0,T;\boldsymbol{X})$, $p\geq1$ (where  $\boldsymbol{X}$ is a Banach space  with the Radon-Nikod\'ym property). 
 \EEE
\begin{proof}
We split the argument in the following claims.
\medskip

\par\noindent
\emph{\textbf{Claim $0$:} The evolution of the mean of $\f u$ is only determined by the given data, \textit{i.e.}, 
\begin{align*}
    \int_\Omega \f u (t)   \dx &= \int_\Omega \f u_0 \dx + t\int_\Omega \f v_0\dx +  \int_0^t\int_\Omega  (t{-}r) \EEE  \f f(r) \dr\,.
\end{align*}}

 Integrating in space the momentum balance \eqref{weakUEq} we  infer 
$
    \partial_t \int_\Omega \f u_t \dx = \int_\Omega \f f \dx \,, $
so that,  integrating in time we get
\begin{equation}
\label{integral-ut}
\int_\Omega \f u _t \dx (t) = \int_\Omega \f v_0\dx + \int_0^t\int_\Omega \f f \dx \ds  
\end{equation}
and thus, integrating again over $(0,t)$, we obtain
\[
    \int_\Omega \f u (t)   \dx = \int_\Omega \f u_0 \dx + t\int_\Omega \f v_0\dx + \int_0^t\int_0^s\int_\Omega \f f \dx \dr\ds \,.
\]
Via Fubini's theorem, we find $\int_0^t\int_0^s\int_\Omega \f f \dx \dr\ds = \int_0^t\int_\Omega \f f(r) \dx\int_{t-r}^{t} \ds\dr = \int_0^t\int_\Omega  (t{-}r) \EEE \f f(r) \dr $\,.
\medskip

\par\noindent
\emph{\textbf{Claim $1$:}  There exists a constant $S_{1,1}>0$ such that for almost all $t\in (0,T)$}
\begin{equation}
\label{claim1-added}
\begin{aligned}
&
\frac{\dd}{\dd t}  \| \f u_t(t)  \|_{L^2(\Omega)}^2  +  \| \f u_t(t)  \|_{H^1(\Omega)}^2 
\\
&
\leq 
S_{1,1}\left( \| \f v_0\|_{L^2(\Omega)}^2 {+} \| \bff \|_{L^2(Q)}^2 {+}  \Big(\|\chi(t) \|_{L^\infty(\Omega)}^{2\rho{+}4} {+}1 \Big)
\Big( \|  \eps(\uu_0)\|_{L^2(\Omega)}^2 {+}  \int_0^t \|  \eps(\uu_t)\|_{L^2(\Omega)}^2 \dd s 
\Big) \right)\,.
 \end{aligned}
\end{equation}
We test \eqref{weakUEq} by $\f u_t$. Taking into account  \eqref{asstensors} and \eqref{ass-b}, by  the Poincar\'e-Korn  inequality we  find
for almost all $t\in (0,T)$
\[
\begin{aligned}
\int_{\Omega} b(\chit(t) ) \VV\eps (\uu_t(t) ){:} \eps (\uu_t(t) ) \dd x &  \geq c  \| \f u_t(t)  \|_{H^1(\Omega)}^2  -C \left|  \int_\Omega {\f u}_t (t)   \dx\right|^2 
\\
&
\stackrel{(*)}{\geq}  c  \| \f u_t(t)  \|_{H^1(\Omega)}^2  -C \left|  \int_\Omega \f v_0   \dx\right|^2  -C \| \bff \|_{L^1(Q)}^2 \,,
\end{aligned}
\]
where {\footnotesize (*)} follows from \eqref{integral-ut}. We thus obtain
\begin{align}
&
\nonumber
\frac{\dd }{\dd t}  \frac{1}{2} \| \f u_t(t)  \|_{L^2(\Omega)}^2  +  c \| \f u_t(t)  \|_{H^1(\Omega)}^2  
\\
&
\nonumber
\leq C \|\f v_0\|^2_{L^1(\Omega)} + C \| \bff \|_{L^2(Q)}^2  + 
\left|  \int_\Omega \bff{\cdot}\uu_t \dd x  \right|  + \left|   \int_\Omega  a(\chi)  \mathbb{C}\eps (\uu){:}\eps (\uu_t) \dd x  \right| 
\\
&
\nonumber
\leq \frac c2  \| \f u_t(t)  \|_{H^1(\Omega)}^2  + C \|\f v_0\|_{L^1(\Omega)}^2 + C' \| \bff \|_{L^2(Q)}^2  + \|a(\chi(t))\|_{L^\infty(\Omega)}\| \eps (\uu(t) )\|_{L^2(\Omega)} 
\| \eps (\uu_t(t) ) \|_{L^2(\Omega)} 
\intertext{and, estimating $ \|a(\chi)\|_{L^\infty(\Omega)} \leq C (\|\chi\|_{L^\infty(\Omega)}^{\rho+2} {+}1 )$ via \eqref{ultimate-growth-conditions} and 
$\| \eps (\uu(t) )\|_{L^2(\Omega)} $ via \eqref{2TFC}, we continue the above chain of inequalities with}
&
\nonumber
\leq \frac {3c}4    \| \f u_t(t)  \|_{H^1(\Omega)}^2  + C \|\f v_0\|_{L^1(\Omega)}^2 + C' \| \bff \|_{L^2(Q)}^2  
\\
&
\nonumber
\qquad 
+ C' (\|\chi(t)\|_{L^\infty(\Omega)}^{\rho+2} {+}1)^2 \left( 2   \|  \eps(\uu_0) \| _{L^2(\Omega)}^{2} +  
 2t  \int_0^t \|  \eps(\uu_t(s))   \| _{L^2(\Omega)}^{2}  \dd s \right) \,,
\end{align}
whence \eqref{claim1-added}. 
\medskip \EEE

\par\noindent
\emph{\textbf{Claim $2$:}  There exist a constant $S_{1,2}>0$  and $\overline\beta >1$
(indeed,  $\overline\beta = (4\rho{+}12)$,  with $\rho:=\max\{ p,q\}$
 and $p,q$ from 
 \eqref{ass-a-3-strong} and \eqref{ass-b-strong}, respectively)
 such that   for every $t \in [0,T]$}  
\begin{align}\label{ineqBeforGron}
\begin{split}
&   \| \f u_t(t)  \|_{H^2(\Omega)}^2  +\int_0^t  \| \f u_t \|_{H^3(\Omega)}^2  \ds  \EEE
\\
&  \leq S_{1,2}
 \Big( 
  \|\vv_0   \|_{H^2(\Omega)}^2  
 {+}   \| \f u_0 \| _{H^3(\Omega)}^2  {+}   \int_0^t\| \f f \|_{H^1(\Omega)}^2\ds  
 \\
 &   \qquad \qquad \qquad 
 {+} \EEE
  \int_0^t 
  \big(  \| \chi \|_{H^2(\Omega)}^{\overline \beta}  {+} 1  \EEE \big) {\times}  
   \big(\| \f u_t\|_{H^2(\Omega)}^2{+} \int_0^s \| \uu_t\|_{H^3(\Omega)}^2 \dta {+}1\big) \dd s  \Big)   \,.
 \end{split}
\end{align}
\EEE
     We  test equation~\eqref{weakUEq} by $ \di (\Vm{:} \varepsilon (\di (\Vm{:}\varepsilon(\f u_t))))$  and integrate in space, thus obtaining
     \begin{equation}
     \label{general-initial-est}
     I_1+I_2+I_3 =I_4 \qquad \text{with } 
     \begin{cases}
     I_1=\int_\Omega  \f u_{tt} \,  \di (\Vm {:} \varepsilon (\di (\Vm{:}\varepsilon(\f u_t)))) \dx,
     \smallskip
     \\
     I_2 = -\int_\Omega \di ( b(\chi) \Vm \eps{(\f u_t)} )\, \di (\Vm{:} \varepsilon (\di (\Vm{:}\varepsilon(\f u_t)))) \dd x,
          \smallskip
     \\
     I_3  = -\int_\Omega \di ( a(\chi) \CC\eps{(\f u)} ) \, \di (\Vm{:} \varepsilon (\di (\Vm{:}\varepsilon(\f u_t)))) \dd x,
          \smallskip
     \\
     I_4=   \int_\Omega  \ff  
     \, \di (\Vm {:} \varepsilon (\di (\Vm{:}\varepsilon(\f u_t)))) \dx.
     \end{cases}
     \end{equation}
 \EEE Then, for the first term we deduce
  \begin{align*}
I_1={}&- \int_\Omega \nabla \f u_{tt} {:} \Vm \eps({\di ( \Vm {:} \eps(\f u_t)) })\dx  + \int_{\partial \Omega }\f u_{tt} {\otimes} \f n {:} \Vm \eps({\di ( \Vm {:} \eps(\f u_t)) }) \dS 
\\
={}&- \int_\Omega \Vm \eps(\f u_{tt} ) {:} \nabla {\di ( \Vm {:} \eps(\f u_t)) }\dx  + \int_{\partial \Omega }\f u_{tt} {\otimes} \f n {:} \Vm \eps({\di ( \Vm {:} \eps(\f u_t)) }) \dS 
\\
={}& \int_\Omega \di \left ( \Vm \eps(\f u_{tt} ) \right )\cdot \di ( \Vm {:} \eps(\f u_t))\dx  + \int_{\partial \Omega }\f u_{tt} {\otimes} \f n {:} \Vm \eps({\di ( \Vm {:} \eps(\f u_t)) }) \dS 
\\&- \int_{\partial \Omega} (  \di ( \Vm {:} \eps(\f u_t)) {\otimes} \f n )\Vm \eps(\f u_{tt} ) \dS \,,
\end{align*}
 integrating by parts and exploiting the symmetry of $\Vm$. \EEE
The two boundary terms vanish:  this will be proved rigorously for the 
approximate system \eqref{galerkinSchauder}. \EEE
Thus, we may infer
\begin{align}
I_1 = \frac{1}{2}\frac{\dd }{\dd t} \int_\Omega \left | \di \left ( \Vm \eps(\f u_t) \right ) \right |^2 \dx  \,.\label{utested1}
\end{align}
In order to calculate  
$I_2$,  we resort to the 
product rule, yielding
\begin{equation}
\label{product-rule}
 \di ( b(\chi) \Vm \eps(\f u_t) ) =( \Vm {:}\eps( \f u_t ) ) \nabla \chi b'(\chi) + b(\chi) \di ( \Vm \eps(\f u_t) ) \,.
 \end{equation}
 Therefore, 
 \begin{subequations}
 \label{calculations4I2}
\begin{align}
\begin{split}
I_2 =- \int_\Omega&\left [ ( \Vm {:}\eps( \f u_t ) ) \nabla \chi b'(\chi) + b(\chi) \di ( \Vm \eps(\f u_t) )  \right ] {\cdot} \di (\Vm {:} \varepsilon (\di (\Vm{:}\varepsilon(\f u_t))))\dx  
\\
={}& \int_\Omega b(\chi) \left [ \Vm {:} \eps  \di ( \Vm \eps(\f u_t ) \right ){:} \varepsilon (\di (\Vm{:}\varepsilon(\f u_t)))] 
\dx  
\\
&+ \int _\Omega b'(\chi) \di ( \Vm \eps(\f u_t) ) {\otimes} \nabla \chi {:} \Vm {:} \eps(\di (\Vm{:}\varepsilon(\f u_t))) \dx  
\\&+ \int_\Omega  \nabla \left (( \Vm {:}\eps( \f u_t ) ) \nabla \chi b'(\chi) \right ) {:} \Vm {:} \varepsilon (\di (\Vm{:}\varepsilon(\f u_t))) \dx  
\\
&-\int_{\partial \Omega}\di ( b(\chi) \Vm \eps(\f u_t) ) {\otimes} \f n {:}\Vm {:} \varepsilon (\di (\Vm{:}\varepsilon(\f u_t)))   \dS  \doteq I_{2,1}+
 I_{2,2} + I_{2,3}+I_{2,4}  \,.
\end{split}
\label{utested2}
\end{align}
 Now, we remark that, thanks to \eqref{ass-b}
\begin{equation}
\label{estI21}
I_{2,1}   \geq b_0 \int_\Omega  \left (\Vm {:} \eps  \di ( \Vm \eps(\f u_t ) \right ){:} \varepsilon (\di (\Vm{:}\varepsilon(\f u_t)))) 
\dx  \geq b_0 \eta_{\Vm}\int_\Omega \left| \varepsilon (\di (\Vm{:}\varepsilon(\f u_t))) \right|^2 \dd x, \EEE 
\end{equation}
 where the last inequality follows from the positive-definiteness of $\Vm$, cf.\ \eqref{asstensors}, 
whereas we estimate 
\[
\begin{aligned}
|I_{2,2} |+|I_{2,3}|  \leq  & 
C  |\Vm | \left \| \varepsilon (\di (\Vm{:}\varepsilon(\f u_t)))\right \|_{L^2(\Omega)} \left \| b'(\chi)\di  (\Vm{:}  \eps{(\f u_t)} ){\otimes} \nabla \chi \right \|_{L^2( \Omega)}
\\&\quad +C
\left \| \varepsilon (\di (\Vm{:}\varepsilon(\f u_t)))\right \|_{L^2(\Omega)} \left \| b'(\chi) (\Vm{:}  \eps{(\f u_t)} )\nabla \chi \right \|_{H^{1}( \Omega)} \,,
\end{aligned}
\]
  where $  |\Vm |$ denotes the tensor norm of $\Vm$. 
\EEE  
The boundary term $I_{2,4}$ vanishes again due to the homogeneous Neumann boundary conditions;  again, this argument  will be made rigorous
for our approximation scheme, cf.\ the proof of Prop.\ \ref{prop:loc} later on. \EEE
\end{subequations}

Similarly,   by the chain rule and 
 an integration-by-parts we obtain \EEE
  \label{calculations4I3}
\begin{align}
\begin{split}
I_3=- \int_\Omega &
\left [\CC {:} \eps(\f u) \nabla \chi a'(\chi) + a(\chi) \di ( \CC \eps(\f u))\right ]{\cdot}    \di (\Vm {:} \varepsilon (\di (\Vm{:}\varepsilon(\f u_t))))\dx  
\\
={}& \int_\Omega a(\chi) \left (\Vm {:} \eps(  \di ( \CC \eps(\f u) )) \right ){:} \varepsilon (\di (\Vm{:}\varepsilon(\f u_t)))) 
\dx  
\\
&+ \int _\Omega a'(\chi) \di ( \CC \eps(\f u) ) {\otimes} \nabla \chi {:} \Vm {:} \eps(\di (\Vm{:}\varepsilon(\f u_t))) \dx  
\\&+ \int_\Omega  \nabla \left (( \CC {:}\eps( \f u ) ) \nabla \chi a'(\chi) \right ) {:} \Vm {:} \varepsilon (\di (\Vm{:}\varepsilon(\f u_t))) \dx  
\\
&-\int_{\partial \Omega}\di ( a(\chi) \CC \eps(\f u) ) {\otimes} \f n {:}\Vm {:} \varepsilon (\di (\Vm{:}\varepsilon(\f u_t)))   \dS 
\\\leq {}& 
\frac{b_0 \eta_\Vm\EEE}{2} \int_\Omega | \eps(\di ( \Vm {:} \eps(\f u_t)) |^2 \dx  +C\| a(\chi) \|_{L^\infty(\Omega)}^2  \| \Vm {:} \eps(  \di ( \CC \eps(\f u) ))  \|_{L^2(\Omega)}^2 \EEE
\\
&+ C \left ( \| a'(\chi)\|_{L^\infty(\Omega)}^2 \|   \di ( \CC \eps(\f u) ) {\otimes} \nabla \chi\|_{L^2(\Omega)}^2+ \| \nabla \left (( \CC {:}\eps(\f u ) ) \nabla \chi a'(\chi) \right ) \| _{L^2(\Omega)}^2  \right )\\
&-\int_{\partial \Omega}\di ( a(\chi) \CC \eps(\f u) ) {\otimes} \f n {:}\Vm {:} \varepsilon (\di (\Vm{:}\varepsilon(\f u_t)))   \dS \,,
\end{split}
\label{utested3}
\end{align}
 with  $b_0$ and $\eta_\Vm$  the constants from  \eqref{estI21}. \EEE
For the boundary term, we observe again that it vanishes,  as shown rigorously in the proof of the upcoming Proposition  \ref{prop:loc}. \EEE

 Finally, arguing in the same way as for $I_1$ we conclude that 
\begin{equation}
 \label{calculations4I4}
\begin{aligned}
I_4={}& -\int_\Omega  \Vm \eps(\f f ) : \nabla \left( \di ( \Vm {:} \eps(\f u_t))\right) \dx  + \int_{\partial \Omega }\f f {\otimes} \f n {:} \Vm \eps({\di ( \Vm {:} \eps(\f u_t)) }) \dS 
\\
\leq{}& \frac{b_0 \eta_{\Vm}\EEE}{4}\int_\Omega | \eps\left( \di ( \Vm {:} \eps(\f u_t))\right)  |^2 \dd x  +  C \EEE \| \f f \|_{H^1(\Omega)}^2 \,.
\end{aligned}
\end{equation}
Note, again,  that the boundary term vanishes   cf.\ the proof of Prop. \ref{prop:loc}. \EEE


Combining  \eqref{general-initial-est}   with \eqref{utested1}, \eqref{calculations4I2}, \eqref{calculations4I3}, and \eqref{calculations4I4},   we infer the estimate

\begin{align*}
 \frac{1}{2}\frac{\mathrm d }{\mathrm d t}& \int_\Omega \left | \di \left ( \Vm \eps(\f u_t) \right ) \right |^2 \dx   +  \frac{b_0\eta_{\Vm}}{4}  
 \int_\Omega  \left| \eps(  \di ( \Vm \eps(\f u_t) )) \right |^2 \dx   
\\
\leq{} &
C
\left \| \varepsilon (\di (\Vm{:}\varepsilon(\f u_t)))\right \|_{L^2(\Omega)} \left \| b'(\chi) (\Vm{:}  \eps(\f u_t) )\nabla \chi \right \|_{H^{1}( \Omega)}
\\&+C \left \| \varepsilon (\di (\Vm{:}\varepsilon(\f u_t)))\right \|_{L^2(\Omega)} \|  b'(\chi)\|_{L^\infty (\Omega)} \left \| (\Vm{:}  \eps(\f u_t) ){\otimes} \nabla \chi \right \|_{L^2( \Omega)}
%
\\&+C\| a(\chi) \|_{L^\infty(\Omega)}^2 \|\eps{ (\f u)}\|_{H^2(\Omega)}^2 
\\
&+ C \left ( \| a'(\chi)\|_{L^\infty(\Omega)}^2 \|   \di ( \CC \eps(\f u) ) {\otimes} \nabla \chi\|_{L^2(\Omega)}^2+ \| \nabla \left (( \CC {:}\eps(\f u ) ) \nabla \chi a'(\chi) \right ) \| _{L^2(\Omega)}^2  \right )\\
&+C  \left \| \varepsilon (\di (\Vm{:}\varepsilon(\f u_t)))\right \|_{L^2(\Omega)} \left \| a'(\chi) (\CC{:}  \eps(\f u) )\nabla \chi \right \|_{H^{1}( \Omega)} +    C  \| \f f \|_{H^1(\Omega)}^2 \EEE
\\
\stackrel{(\star)}{\leq}{}& 
C  ( \|\chi \|_{L^\infty(\Omega)}^{2(q{+}1)}{+}1) \EEE \left (  \| \varepsilon(\f u_t)\|_{W^{1,3}(\Omega)}^2 \| \nabla \chi\|_{L^6(\Omega)}^2+ \| \chi\|_{H^2(\Omega)}^2 \| \varepsilon(\f u_t)\|_{L^\infty(\Omega)}^2  \right ) 
\\
&+C   (\|\chi \|_{L^\infty(\Omega)}^{2q}{+}1)  \EEE  \| \varepsilon(\f u_t)\|_{L^\infty(\Omega)}^2 \| \nabla \chi\|_{L^4(\Omega)}^4
+C(\| \chi \|_{L^\infty(\Omega)}^{2(p+2)}+1)  \| \eps{(\f u)}\|_{H^2(\Omega)}^2 
\\
&+ C   (\|\chi \|_{L^\infty(\Omega)}^{2(p{+}1)}{+}1)  \EEE \
 \left (  \| \varepsilon(\f u)\|_{W^{1,3}(\Omega)}^2 \| \nabla \chi\|_{L^6(\Omega)}^2+ \| \chi\|_{H^2(\Omega)}^2 \| \varepsilon(\f u)\|_{L^\infty(\Omega)}^2  \right ) \\&
+C  (\|\chi \|_{L^\infty(\Omega)}^{2p} {+}1) \EEE \ \| \varepsilon(\f u)\|_{L^\infty(\Omega)}^2 \| \nabla \chi\|_{L^4(\Omega)}^4  +  C  \| \f f \|_{H^1(\Omega)}^2 
\,, \EEE
\end{align*} 
 where for {\footnotesize $(\star)$} we have also used the growth conditions \eqref{ass-a-3-strong} and \eqref{ass-b-strong}. 
%
%
%
Integrating in time and inserting the Gagliardo--Nirenberg inequalities in three dimensions 
\begin{align*}
\| \zeta \| _{L^\infty(\Omega)} + \| \zeta \|_{W^{1,3}(\Omega)} \leq c \| \zeta \|_{H^2(\Omega)}^{1/2} \| \zeta \| _{H^1(\Omega)}^{1/2}\,, \qquad \|\zeta \|_{W^{1,4}(\Omega)}  \leq c \| \zeta \|_{H^2(\Omega)}^{3/4} \| \zeta \| _{H^1(\Omega)}^{1/4}\,,
\end{align*}
for all $ \zeta \in H^2(\Omega)$, we obtain
\begin{align}\label{ineqforu}
\begin{split}
     \frac{1}{2}& \int_\Omega \left | \di \left ( \Vm \eps(\f u_t(t)) \right ) \right |^2 \dx  +  \frac{b_0\eta_{\Vm}}{4}  \EEE\int_0^t   \int_\Omega   \left| \eps(  \di ( \Vm \eps(\f u_t) )) \right |^2 \EEE
\dx   \ds 
\\
\leq{}&
C\int_0^t    (\|\chi \|_{L^\infty}^{2\rho+2} {+}1) \EEE  \left( \| \varepsilon(\f u_t)\|_{H^2(\Omega)}\| \varepsilon(\f u_t)\|_{H^1(\Omega)} +
\| \varepsilon(\f u)\|_{H^2(\Omega)}\| \varepsilon(\f u)\|_{H^1(\Omega)}\right ) \|  \chi\|_{H^2(\Omega)}^2
 \ds 
\\
&+ C\int_0^t  (\|\chi \|_{L^\infty}^{2\rho} {+}1) \EEE \left  (  \| \varepsilon(\f u_t)\|_{H^2(\Omega)}\| \varepsilon(\f u_t)\|_{H^1(\Omega)} +  \| \varepsilon(\f u)\|_{H^2(\Omega)}\| \varepsilon(\f u)\|_{H^1(\Omega)}\right  ) \| \chi\|_{H^2(\Omega)}^3\| \chi\|_{H^1(\Omega)}
\ds \\&\quad  +  C  \int_0^t \big\{  \| \f f \|_{H^1(\Omega)}^2 + ( \| \chi\|_{L^\infty}^{2\rho +4}+1) \| \eps{(\uu)}\|_{H^2(\Omega)}^2 \big\} \ds
\,, \EEE
\\
\stackrel{(\star\star)}\leq{} &   
 \mu \int_0^t \| \uu_t\|_{H^3(\Omega)}^2 \dd s  + C\|\vv_0\|_{H^2(\Omega)}^2 \EEE
 +C \int_0^t \big\{  \| \f f \|_{H^1(\Omega)}^2 + ( \| \chi\|_{L^\infty}^{2\rho +4}+1) \| \varepsilon(\f u)\|_{H^2(\Omega)}^2 \big\}  \ds
 \\
 & \quad +C \int_0^t     (\|\chi \|_{L^\infty}^{4\rho+4} {+}1) \EEE  \left ( \| \chi \|_{H^2(\Omega)}^4 +   \| \chi \|_{H^2(\Omega)}^8   \right ) \left (  \| \varepsilon(\f u_t)\|_{H^1(\Omega)}^2+\| \varepsilon(\f u)\|_{H^1(\Omega)}^2\right ) \ds \,,
\end{split}
\end{align}
 for a positive constant $\mu$ to be specified later, and \EEE
 recalling that  $\rho=\max\{ p,q\}$. For {\footnotesize $(\star\star)$} 
 we have estimated $ \| \chi \|_{H^1(\Omega)}$ via  $\| \chi \|_{H^2(\Omega)}$,
  estimated  $ \| \eps{  \di ( \Vm \eps({\f u_t} ))}  \|_{L^2(\Omega)}$ via 
 $\| \uu_t\|_{H^3(\Omega)}$
 and $
\|  \di \left ( \Vm \eps(\f v_0) \right ) \|_{L^2(\Omega)} $ via  $\| \vv_0\|_{H^2(\Omega)}$, 
  and 
  used Young's inequality. Again  by \eqref{2TFC}, we observe  that 
$$
 \| \varepsilon(\f u)\|_{H^j(\Omega)}^2 
\leq    2\| \eps(\f u_0) \| _{H^j(\Omega)}^2 +   2T   \int_0^t \| \eps(\f u_t)\|_{H^j(\Omega)}^2 \ds \qquad \text{for } j\in \{1,2\}\,.
$$
\par
 We now add \eqref{ineqforu} and \eqref{claim1-added} integrated over $(0,t)$, thus obtaining \EEE 
\begin{equation}\label{ineqBeforGron-ALMOST}
\begin{aligned}
&    \| \uu_t(t)\|_{L^2(\Omega)}^2 +    \int_\Omega \left | \di \left ( \Vm \eps(\f u_t(t)) \right ) \right |^2 \dx 
+ c  \int_0^t   \left( \| \f u_t  \|_{H^1(\Omega)}^2  {+}   \int_\Omega   \left| \eps(  \di ( \Vm \eps(\f u_t) )) \right |^2 \dd x \right) \dd s \EEE 
 \\
 & \leq{}  \mu \int_0^t \| \uu_t\|_{H^3(\Omega)}^2 \dd s  \EEE 
  \\
  &
 \quad 
+ C\left(
 \| \vv_0  \|_{H^2(\Omega)}^2  
 + \int_0^t ( \| \chi\|_{L^\infty}^{2\rho +4}+1) \left( \| \eps(\f u_0) \| _{H^2(\Omega)}^2 +     \int_0^s \| \eps(\f u_t)\|_{H^2(\Omega)}^2 \dta\right ) 
 +   \| \f f \|_{H^1(\Omega)}^2\ds  \right)
 \\
 &
 \quad +C \int_0^t (\|\chi \|_{L^\infty(\Omega)}^{4\rho+4} {+}1)  (   \| \chi \|_{H^2(\Omega)}^8  {+}1 )  \left (\|\eps( \f u_t)\|_{H^1(\Omega)}^2{+}    \int_0^s \| \eps(\uu_t)\|_{H^1(\Omega)}^2 \dta{+}  \| \eps(\uu_0)  \|_{H^1(\Omega)}^2  \EEE  \right )  \ds \,.
\end{aligned}
\end{equation}
 Now, it follows from 
the elliptic regularity estimates \eqref{key-elliptic-regul-est-NEW} from 
Corollary \ref{corB:ellipt} that there exists $\widehat{C}_{\mathrm{ER}}>0$ such that  
\[
  \| \uu_t(t)\|_{H^2(\Omega)}^2  \leq \widehat{C}_{\mathrm{ER}} \left( \| \uu_t(t)\|_{L^2(\Omega)}^2 {+}    \int_\Omega \left | \di \left ( \Vm \eps(\f u_t(t)) \right ) \right |^2 \dx   \right)\,.
\] \EEE
Likewise, we have 
\[
   \| \uu_t(t)\|_{H^3(\Omega)}^2  \leq  \widehat{C}_{\mathrm{ER}} \left( \| \uu_t(t)\|_{H^1(\Omega)}^2 {+}   \int_\Omega   \left| \eps(  \di ( \Vm \eps(\f u_t) )) \right |^2 \dd x  \right)\,.
\] 
Hence, we choose $\mu$ in \eqref{ineqBeforGron-ALMOST} in such a way as to absorb $\int_0^t \| \uu_t\|_{H^3(\Omega)}^2 \dd s $ into the left-hand side, thus obtaining
\[
\begin{aligned}
&   \| \uu_t(t)\|_{H^2(\Omega)} 
+ c  \int_0^t    \| \f u_t  \|_{H^3(\Omega)}^2   \dd s \EEE 
 \\
 & \leq{}  C\left(
 \| \vv_0  \|_{H^2(\Omega)}^2  
 + \int_0^t ( \| \chi\|_{L^\infty(\Omega)}^{2\rho +4}+1) \left( \| \eps(\f u_0) \| _{H^2(\Omega)}^2 +     \int_0^s \| \eps(\f u_t)\|_{H^2(\Omega)}^2 \dta\right ) 
 +   \| \f f \|_{H^1(\Omega)}^2\ds  \right)
 \\
 &
 \quad +C \int_0^t (\|\chi \|_{L^\infty(\Omega)}^{4\rho+4} {+}1)  (   \| \chi \|_{H^2(\Omega)}^8  {+}1 )  \left (\|\eps( \f u_t)\|_{H^1(\Omega)}^2{+}    \int_0^s \| \eps(\uu_t)\|_{H^1(\Omega)}^2 \dta{+}  \| \eps(\uu_0)  \|_{H^1(\Omega)}^2  \EEE  \right )  \ds \,.
\end{aligned}
\] 
Therefore,  \EEE  estimating 
\begin{equation}
\label{useful-in-the-end}
\begin{cases}
\|\chi \|_{L^\infty(\Omega)} \leq C \|\chi \|_{H^2(\Omega)}
\\
\|\eps({\f u_t}) \|_{H^2(\Omega)} \leq C \|\uu_t \|_{H^3(\Omega)},
\end{cases}
\end{equation} 
we arrive at 
  \eqref{ineqBeforGron}. 
\medskip

\par\noindent
\emph{\textbf{Claim $3$:}  there exist a constant $S_{1,3}>0$ and   $\underline\beta >1$
 such that}  
\begin{equation}
\label{testedchi}
    \begin{aligned}
        &
         \| \omega(t)\|_{L^2(\Omega)}^2 +   \| \chi(t)\|_{H^2(\Omega)}^2  +\int_0^t \left( \|\chi_t\|_{L^2(\Omega)}^2{+}  \|\nabla \chi_t\|_{L^2(\Omega)}^2 \right) \dd s 
        \\ & 
\leq S_{1,3}(1{+}   \| \chi_0\|_{H^2(\Omega)}^2{+}  \| |\partial \breve{W}^{\circ}|(\chi_0) \|_{L^2(\Omega)}^2 
 {+} \| \uu_0\|_{H^2(\Omega)}^8) 
\\
& \quad  +S_{1,3} \int_0^t  \left( \| \omega\|_{L^2(\Omega)}^8  {+} \|  \uu_t \|_{H^2(\Omega)}^8 {+}  \int_0^s \|  \uu_t \|_{H^2(\Omega)}^8 \dd \tau  {+} \|\chi\|_{H^2(\Omega)}^{\underline\beta} \right) \dd s \,.   
\end{aligned}
\end{equation}
We test \eqref{rewritten-flow-rule} 
 by $\omega_t$ and integrate in space and over the time interval $(0,t)$,
$t\in (0,T)$. 
 Thus, we obtain 
\begin{equation}
\label{claim3-1}
\begin{aligned}
&
\frac12 \| \omega(t)\|_{L^2(\Omega)}^2  +\int_0^t \left( \|\chi_t\|_{L^2(\Omega)}^2{+}  \|\nabla \chi_t\|_{L^2(\Omega)}^2 \right) \dd s 
\\
& \quad 
+\ddd{\int_0^t  \int_\Omega \{ \breve{W}''(\chi)|\chi_t|^2 {+} I'(\chi_t) \chi_t  {+}I''(\chi_t) |\nabla\chi_t|^2 {+}\breve{W}''(\chi)I'(\chi_t) \chi_t \} \dd x \dd s}{$I_0 \geq 0$}{}   
\\
&   =
\int_0^t \int_\Omega \left( \chi - \invbreve{W}'(\chi) -\frac12 a'(\chi) \e(\uu) \CC \e(\uu) \right)\omega_t \dd x \dd s \doteq I_1\,.
\end{aligned}
\end{equation}
Indeed, by the convexity of $\breve{W}$ and $I_{(-\infty,0]}$, the first and third 
contributions to $I_0$  are non-negative; \EEE likewise,   the monotonicity of $I'$ and the fact that $I'(0)=0$ ensure that the second and fourth term in $I_0$ is positive. We integrate $I_1$ by parts, thus obtaining \EEE
\begin{subequations}
\label{treatment-I1}
\begin{align}\label{partintime}
\begin{split}
    \\
    I_1= {}&\int_0^t \omega   \left(  \frac12a''(\chi) \EEE \chi_t \Cm{\f u}{\f u} + a'(\chi) \Cm{\f u}{\f u_t} + \invbreve{W}''(\chi)\chi_t- \chi_t \right) \dd x \dd s  \\
    &-\int_\Omega\omega(t) \left( \frac12 a'(\chi(t))\EEE  \Cm{\f u(t) }{\f u(t)} + \invbreve{W}'(\chi(t))- \chi(t)  \right) \dx\\
    &+\int_\Omega\omega(0) \left( \frac12 a'(\chi_0)\EEE  \Cm{\f u_0 }{\f u_0} + \invbreve{W}'(\chi_0)- \chi_0  \right)  \dx
    \doteq I_{1,1} + I_{1,2} +I_{1,3}\,.
\end{split}
\end{align}
By H\"older's and Young's inequalities we have 
\begin{equation}
\label{treatment-I1-1}
\begin{aligned}
 |I_{1,1}|   \leq{}& \frac12 |\mathbb{C}| \int_0^t \|\omega \|_{L^2(\Omega)}  \|  a''(\chi) \EEE \|_{L^\infty(\Omega)} \| \chi_t\|_{L^6(\Omega)}\| \eps({\f u})\|_{L^6(\Omega)}^2\ds \\& 
  +\int_0^t\| \omega \|_{L^2(\Omega)}\left(  \frac12 |\mathbb{C}|  \| a'(\chi)\|_{L^\infty(\Omega)} \EEE
  \| \eps({\f u}) \|_{L^4(\Omega)}\| \eps({\f u_t}) \|_{L^4(\Omega)} +  \|  \invbreve{W}''(\chi)-1\|_{L^\infty(\Omega)} \| \chi_t\|_{L^2(\Omega)}\right)\ds 
    \\
    \stackrel{(1)}{\leq{}}& \frac{1}{4}\int_0^t \| \chi_t\|_{H^1(\Omega)}^2 \ds 
     + C  \int_0^t \| \omega\|_{L^2(\Omega)}^2 \ds 
    +C \int_0^t  \|\omega \|_{L^2(\Omega)}^2 (\|\chi\|_{L^\infty(\Omega)}^{2p}{+}1) \| \eps({\f u})\|_{L^6(\Omega)}^4 \dd s 
      \\
      & \quad   +C \int_0^t  (\|\chi\|_{L^\infty(\Omega)}^{2p+2}{+}1) 
      \| \eps({\f u}) \|_{L^4(\Omega)}^2\| \eps({\f u_t}) \|_{L^4(\Omega)}^2 \dd s 
       +\frac14 \int_0^t  
       \| \chi_t\|_{L^2(\Omega)}^2 
   \dd s  
\\
    \stackrel{(2)}{\leq{}}&
    \frac{1}{2}\int_0^t \| \chi_t\|_{H^1(\Omega)}^2 \ds 
     + \int_0^t \| \omega\|_{L^2(\Omega)}^2 \ds  
    \\
    & \quad
    +C \int_0^t  \left( \| \omega\|_{L^2(\Omega)}^8  {+} \| \eps({\f u_t}) \|_{L^4(\Omega)}^8  {+} \| \eps({\f u}) \|_{L^6(\Omega)}^8 {+} \|\chi\|_{L^\infty(\Omega)}^{m} \right)   \ds \,,
\end{aligned}
\end{equation}
 where for {\footnotesize (1)} we have resorted to the growth properties \eqref{ass-a-3-strong} and \eqref{ultimate-growth-conditions} of $a''$ and $a'$, 
  and used that $\invbreve{W}''(\chi) \equiv -\ell$.  \EEE Moreover,  {\footnotesize (2)} again follows from Young's inequality; therein, $m= \max\{8p  ,4p+4 \}= 8p $ .

Secondly, we observe via Young's inequality that
\[
\begin{aligned}
|I_{1,2}|
    \leq{} \frac{1}{4}\| \omega(t) \|_{L^2(\Omega)}^2 + \| \invbreve{W}'(\chi(t)) {+}  \chi(t)\|_{L^2(\Omega)}^2  + C \| \eps(\f u(t))\|_{L^4(\Omega)}^4 \,.
\end{aligned}
\]
From \eqref{2TFC} we gather
$
    \| \eps(\f u(t))\|_{L^4(\Omega)}^4 
\leq
 \| \uu_0\|_{H^2(\Omega)}^4 +  C \int_0^t\| \eps( \uu_t)\|_{L^4(\Omega)}^4 \dd s \,.
$
 Recalling  that   $ \invbreve{W}'(\chi) = -\ell \chi $, 
we find
\[
\begin{aligned}
     \| \invbreve{W}'(\chi(t)) {-}  \chi(t)\|_{L^2(\Omega)}^2   \leq {}& (\ell{+}1)^2  \| \chi(t)\|_{L^2(\Omega)}^2 
\leq 2 (\ell{+}1)^2 \left(\| \chi_0\|_{L^2(\Omega)}^2  {+}  \int_0^t \int_\Omega \chi_t \chi \dx   \dd s  \right)
\\
\leq{}& \frac{1}{4}\int_0^t \| \chi_t\|^2_{L^2(\Omega)}\ds + C \| \chi_0\|^2_{L^2(\Omega)} +C \int_0^t \| \chi\|_{L^2(\Omega)}^2 
\,.
\end{aligned}
\]
All in all, we conclude 
\begin{equation}
\label{treatment-I1-2}
\begin{aligned}
|I_{1,2}|
    \leq{}& \frac{1}{4}\left( \| \omega(t) \|_{L^2(\Omega)}^2 +   \int_0^t \|\chi_t\|_{L^2(\Omega)}^2  \dd s\right) \\& + C \left(\| \chi_0\|_{L^2(\Omega)}^2{+} \| \uu_0\|_{H^2(\Omega)}^4 {+} \int_0^t\|  \eps(\uu_t)\|_{L^4(\Omega)}^4 \dd s + \int_0^t \| \chi\|_{L^2(\Omega)}^2  \right)\,.    
\end{aligned}
\end{equation} \EEE
%
 \par Finally, we have 
\begin{equation}
\label{treatment-I1-3}
\begin{aligned}
|I_{1,3}|
    \leq{} \frac{1}{4}\| \omega(0) \|_{L^2(\Omega)}^2 + C \left(\| \uu_0\|_{H^2(\Omega)}^4 + \| \invbreve{W}'(\chi_0)\|_{L^2(\Omega)}^2  {+}  \| \chi_0\|_{L^2(\Omega)}^2 \right)\ \leq C.
\end{aligned}
\end{equation}
\end{subequations}
\par
 Combining \eqref{claim3-1} with \eqref{treatment-I1}, 
  and again using that $ \|\chi \|_{L^\infty(\Omega)} \leq C \|\chi \|_{H^2(\Omega)}$ \EEE
   we obtain 
\begin{equation}
    \label{ineqforchi}
    \begin{aligned}
        &
        \| \omega(t)\|_{L^2(\Omega)}^2  +\int_0^t \left( \|\chi_t\|_{L^2(\Omega)}^2{+}  \|\nabla \chi_t\|_{L^2(\Omega)}^2 \right) \dd s 
        \\ & 
\leq C(1{+} \| \chi_0\|_{L^2(\Omega)}^2{+} \| \omega(0)\|_{L^2(\Omega)}^2 {+} \| \uu_0\|_{H^2(\Omega)}^4) 
\\
& \quad  +C \int_0^t  \left(
\| \omega\|_{L^2(\Omega)}^8  {+} \| \eps({\f u_t}) \|_{L^4(\Omega)}^8  {+} \| \eps({\f u}) \|_{L^6(\Omega)}^8 {+} 
 \|\chi\|_{H^2(\Omega)}^{8 p}  \EEE
 \right) \dd s \,.
\end{aligned}
\end{equation}
 By \eqref{2TFC} we have
$
 \| \eps(\f u(t)) \|_{L^6(\Omega)}^8  \leq   2^7  \| \eps(\f u_0) \| _{L^6(\Omega)}^8  +    2^7 T^7 \int_0^t \| \eps(\f u_t)\|_{L^6(\Omega)}^8  \ds$.
 Furthermore, 
 we use that 
$
\|\eps({\f u_t}) \|_{H^1(\Omega)} \leq C \|\uu_t \|_{H^2(\Omega)}$. 
In order to conclude \eqref{testedchi}, it remains to observe that, 
by \eqref{chidata-strong}, 
\begin{equation}
\label{est-omega-0}
 \| \omega(0)\|_{L^2(\Omega)} \leq 	 \|\chi_0\|_{H^2(\Omega)}  +\| |\partial \breve{W}^{\circ}|(\chi_0) \|_{L^2(\Omega)}+  \|\chi_0\|_{L^2(\Omega)} \,, 
 \end{equation}
and to remark that 
the $L^2$-norm of $\omega$ 
does bound the 
$H^2$-norm of $\chi$,  cf.\ \eqref{est-omega-BS}. 
 All in all, 
 we arrive at \eqref{testedchi} with $\underline{\beta}=8p$. 
%

\medskip
\par\noindent 
\emph{\textbf{Claim $4$:}  there exists a constant $S_{1,4}>0$ such that for   $\beta :=  \max\{ \overline\beta, \tfrac12 \underline\beta\} $ there holds}
\begin{equation}
\label{ineqrelative}
\begin{aligned}
&   \| \f u_t(t)  \|_{H^2(\Omega)}^2  +\int_0^t  \| \f u_t \|_{H^3(\Omega)}^2  \ds  
  +\| \omega(t)\|_{L^2(\Omega)}^2 +   \| \chi(t)\|_{H^2(\Omega)}^2  +\int_0^t  \|\chi_t\|_{H^1(\Omega)}^2\dd s 
\\
&  \leq S_{1,4}
 \Big( 1{+}
  \|\vv_0   \|_{H^2(\Omega)}^2  
 {+}   \| \f u_0 \| _{H^3(\Omega)}^2  {+}  \| \chi_0\|_{H^2(\Omega)}^2{+} \| |\partial \breve{W}^{\circ}|(\chi_0) \|_{L^2(\Omega)}^2 \EEE
 {+} \| \uu_0\|_{H^2(\Omega)}^8 {+}    \int_0^t\| \f f \|_{H^1(\Omega)}^2\ds    \Big)
 \\
 & \quad 
 + S_{1,4}  \int_0^t  \left\{
 \| \chi \|_{H^2(\Omega)}^{2\beta}  {+}  \| \f u_t\|_{H^2(\Omega)}^8
  {+} \left(  \int_0^s \|  \uu_t \|_{H^3(\Omega)}^2 \dd \tau \right)^2    {+} \| \omega\|_{L^2(\Omega)}^8 \right\} \dd s \,.
   \end{aligned}
 \end{equation}
It suffices to add estimates \eqref{ineqBeforGron} and \eqref{testedchi}: as for the left-hand side of  \eqref{ineqBeforGron}, we use that 
\[
\begin{aligned}
&  \int_0^t 
  \big(  \| \chi \|_{H^2(\Omega)}^{\overline\beta}  {+} 1   \big) {\times}  
   \big(\| \f u_t\|_{H^2(\Omega)}^2{+} \int_0^s \| \uu_t\|_{H^3(\Omega)}^2 \dta {+}1\big) \Big) \dd s 
 \\ &   \leq  C \left( T{+} \int_0^t  \left\{
 \| \chi \|_{H^2(\Omega)}^{2 \overline\beta}  {+}  \| \f u_t\|_{H^2(\Omega)}^4
   {+} \left(  \int_0^s \|  \uu_t \|_{H^3(\Omega)}^2 \dd \tau \right)^2   \right\} \dd s 
  \right)\,,
   \end{aligned}
\]
whereas we trivially observe that 
\[
 \int_0^t \int_0^s \|  \uu_t \|_{H^2(\Omega)}^8 \dd \tau  \dd s \leq T  \int_0^t \|  \uu_t \|_{H^2(\Omega)}^8 \dd s 
\]
for the corresponding term on the right-hand side of  \eqref{testedchi}. Then, \eqref{ineqrelative} ensues.


\medskip
\par\noindent 
\emph{\textbf{Conclusion of the proof:}}

With the choice
\[
\psi(t):=  \| \f u_t(t)  \|_{H^2(\Omega)}^2  +\int_0^t  \| \f u_t \|_{H^3(\Omega)}^2  \ds  
  +\| \omega(t)\|_{L^2(\Omega)}^2 +   \| \chi(t)\|_{H^2(\Omega)}^2 +1,
  \]
 we observe  that for $\beta> 4$, the estimate \eqref{ineqrelative} yields
\begin{equation}
\label{psi-ineq}
\psi(t) \leq \widetilde{S}_1  \psi(0) +\int_0^t \widetilde{S}_1   \psi(s)^\beta \dd s  \qquad \text{for all } t \in [0,T^*]
\end{equation}
  for some suitable constant 
  $
\widetilde{S}_1  $ also encompassing    $\| \uu_0\|_{H^2(\Omega)}^8 $, $ \| \f u_0 \| _{H^3(\Omega)}^2$, $ \|\vv_0   \|_{H^2(\Omega)}^2 $,
$\| \chi_0\|_{H^2(\Omega)}^2$, 
  $\| |\partial \breve{W}^{\circ}|(\chi_0) \|_{L^2(\Omega)}^2$, 
and 
 $  \int_0^T\| \f f \|_{H^1(\Omega)}^2\ds $. \EEE
\EEE Let us define $\phi(t):= \widetilde{S}_1  \psi(0) +\int_0^t \widetilde{S}_1   \psi(s)^\beta \dd s $. Then, taking the inequality \eqref{ineqrelative} to the power $\beta$
one has
\[
\phi'(t) = \psi^\beta(t)  \leq \left(\widetilde{S}_1  \psi(0) +\int_0^t \widetilde{S}_1   \psi(s)^\beta \dd s \right)^\beta =   \phi^\beta (t) 
\qquad \text{for all } t \in [0,T^*]\,.
\]
Via the usual comparison arguments for ODEs,  
from $\phi' \leq \phi^\beta$ we conclude 
that 
\[
 \psi(t) \leq \phi(t) \leq  \left( 
 \frac1{\displaystyle \phi^{1{-}\beta}(0){-}(\beta{-}1)t } \right)^{1/(\beta{-}1)} \qquad\text{for all } t < \frac1{\beta{-}1} \phi^{1{-}\beta}(0) = \frac1{\beta{-}1} \psi^{1{-}\beta}(0)\,.
\]
Therefore, we may conclude, e.g.,  that 
\[
\psi(t)\leq \frac1{2^{\beta-1}} \phi(0) = \frac{ \widetilde{S}_1}{2^{\beta-1}}  \psi(0) \qquad \text{for } t \in \left(0,  \frac1{2(\beta{-}1)} \psi^{1{-}\beta}(0)\right]\,. 
\] \EEE
In this way, 
 we conclude  estimate~\eqref{eq:proploc}. 
\par
Eventually, \eqref{comparisonH2} follows from \eqref{eq:proploc}, arguing by  comparison in the momentum balance.  This concludes the proof.
\end{proof} \EEE

 The following sections will be devoted to the rigorous justification of Proposition \ref{prop:formal-aprio}. \EEE
\subsection{Regularization and Galerkin approximation}
\label{ss:4.2}
 We will approximate system \eqref{PDEsystem} by 
\begin{enumerate}
\item  
Regularizing the  possibly \emph{nonsmooth}  (cf.\ Hypothesis \ref{h:1-strong})  convex contribution $ \breve W$ to $W$,
in order to rigorously carry out the estimates leading to Claim $3$ in the proof of Prop.\  \ref{prop:formal-aprio}. In fact, we will need to replace $\breve W$ by a \EEE
regularised version  $W_\delta  \in \rmC^3 (\R)$, $\delta \in(0,1)$, \EEE   such that
\begin{align}
\label{properties-W-delta}
\begin{cases}
    0\leq \breve W''_\delta (r)\leq \frac1\delta
    \\
   | W'''_\delta (r) |  \leq  \frac{C}{\delta^3} \EEE
   \end{cases}
     \text{ for all }r\in \R \text{ and }\lim_{\delta \to 0} \breve{W}_\delta (r) = \breve W(r) \text{ for all }r\in \dom \breve W\,.
\end{align}
  Likewise, we will replace the  the indicator function  $ I_{(-\infty,0]}$ by its \emph{smoothened}  Moreau-Yosida approximation $I_\delta$. 
\item
Adding an  elliptic time-regularizing term to the damage flow rule,  tuned by a second parameter $\nu>0$  that will need to scale suitably w.r.t.\ $\delta$,
 cf.\ \eqref{scaling-condition}  below. 
 \EEE
\item Adopting a \EEE  Galerkin discretization for the momentum balance,  consisting of eigenfunctions of a selfadoint operator, see below.
\end{enumerate}
 \par 
 In order to obtain a   smooth approximation
$\breve{W}_\delta$ of $\breve W$ and $I_\delta$ of $ I_{({-}\infty, 0]}$,
we shall apply the construction detailed in 
 Section \ref{s:appB} ahead, and based on 
 the results in  \cite[Sec.\,3]{Gilardi-Rocca}, 
  to  the operators  $\betaup= \partial\breve W$ and $\betaup = \partial I_{({-}\infty, 0]}$.  Let us now delve into the Galerkin discretization of the  momentum balance. \EEE

\subsubsection*{Galerkin approximation}
We are going to use a Galerkin scheme to discretize the elasticity subsystem in space.
 Hereafter, we will use the notation 
\[
\Lquo:= \{ \vv  \in L^2(\Omega;\R^d)\, : \ \int_\Omega \vv \dx =0 \}, \qquad \Hquo = H^1(\Omega;\R^d) {\cap} \Lquo\,.
\] \EEE
For the approximation of the elasticity equation, we use an $L^2(\Omega)$-orthonormal Galerkin basis consisting of eigenfunctions $\f y_1 , \, \f y_2 , \, \ldots$
of the differential operator corresponding to the boundary value problem
\begin{align}\label{boundaryvalueproblem}
\begin{split}
-\di (\Vm \eps{( \f y)})  &= \f h \quad\text{in } \Omega \, , \quad
\int_\Omega \f y \dx =0 \,,
\quad
  \f n {\cdot} \Vm \eps( \f y)  = 0\quad\text{on } \partial\Omega \,.
\end{split}
\end{align}
The above problem  is a symmetric strongly elliptic system that possesses, by the Lax-Milgram lemma,  a unique weak solution $\f y \in \Hquo$ for any $\f h \in (\Hquo)^*$. 
 Its solution operator is thus a compact selfadjoint operator in $\Lquo$. Hence there exists an orthogonal basis of eigenfunctions  $\f y_1 , \, \f y_2 , \, \ldots$ in $\Lquo$.
The regularity result of Proposition \ref{propB:ellipt}  ahead  ensures that 
the  eigenfunctions   $\f y_1 , \, \f y_2 , \, \ldots$  are, indeed, in $ H^3(\Omega;\R^d)  $. Therefore, the space spanned by them, and by $\f y_0 \equiv \f 1$, \EEE
 \[
  V_n : = \spa \left \{\f 1,   \f y_1, \dots , \f y_n \right \} \subset  H^3(\Omega;\R^d)  \,.
  \] 
 We will  need to consider both the  orthogonal
projection
$\mathbb{P}_{H^3}^n
: H^3(\Omega;\R^d) \longrightarrow V_n$ and $\mathbb{P}_{H^2}^n
: H^2(\Omega;\R^d) \longrightarrow V_n$. With slight abuse, we will drop the subscript $H^k$, $k\in \{2,3\}$, in their notation.


\subsubsection*{The approximate system}
Combining the regularization for the damage model with the Galerkin-discretization for the elasticity equations, we end up with the regularized--discretized system 
\begin{subequations}\label{eq:regularized}
\begin{align}
 \int_\Omega \f u _{tt}{\cdot} \f z  + \left( b({\chi}) \Vm \varepsilon(\f u_t )  {+} a({\chi} )\CC\varepsilon({\f u})\right)  {:} \eps(\f z )\dx &   && 
  \label{regularized1}
\\
\qquad \qquad \qquad 
= \int_\Omega \f f {\cdot} \f z \dx &  && \text{for all }\f z \in V^n,  \text{a.e.\ in } (0,T), \EEE
\\
\nu  \omega_{tt}  +\omega+ \chi_t +I'_{\delta } ( {\chi_t} ) +\tfrac{1}{2}a'(\chi) \Cm{\f u}{\f u} +\invbreve{W}'({\chi})-{\chi}&=  0     &&  \text{a.e.\ in } Q, \EEE 
\label{regularized2}\\
-\Delta \chi + \breve{W}_\delta'(\chi)+\chi  &= \omega    &&  \text{a.e.\ in } Q\,,
 \EEE \label{regularized3}
 \\
 \partial_{\pmb n}\chi & =0 &&  \text{a.e.\ on } \Sigma\,.
  \label{regularized=BC}
\end{align}
\end{subequations}

\subsection{Existence 
 for  the regularized approximate  system}
\label{ss:4.3}
\renewcommand{\eps}[1]{\varepsilon ({#1})}

 First of all, let us show that the Cauchy problem for  system \eqref{eq:regularized},
  supplemented with the initial data $(\mathbb{P}^n( \f u_0), 
 \mathbb{P}^n (\f v_0), 
\chi_0 )$ and with an additional initial datum for $\omega_t$, \EEE 
  does admit a local-in-time 
\emph{strong} solution (here `strong' refers to the fact that  \eqref{regularized2}--\eqref{regularized=BC} are satisfied pointwise). 
\begin{proposition}
\label{prop:loc-exist-approx}
Let $(\uu_0,\vv_0,\chi_0) \in H^3(\Omega;\R^d){\times}  H^2(\Omega;\R^d) {\times} H^2(\Omega)$ fulfill Hyp.\ 
\textbf{\ref{h:2-strong}}.  Let $\varpi_0 \in L^2(\Omega)$ be given.
\par
For every  $\delta,\, \nu >0$  there exists $\widetilde T = \widetilde T(\delta,\nu) \in (0,T]$ \EEE  such that for every   $n \in \N$  system \eqref{eq:regularized}
admits a solution $(\uu,\chi)$ with the regularity 
\begin{equation}
\label{reg-solution-local-in-time}
\begin{aligned}
&\uu\in H^1(0,\widetilde{T};H^3(\Omega;\R^d))\cap W^{1,\infty}(0,\widetilde{T};H^2(\Omega;\R^d)) 
 \cap H^2(0,\widetilde{T};  H^1(\Omega;\R^d)), 
  \\
			&\chi\in   W^{1,\infty}(0,\widetilde{T};H^2(\Omega))\cap W^{2,\infty}(0,\widetilde{T};H^1(\Omega))\,, \EEE  
			\end{aligned}
\end{equation}
satisfying the initial conditions 
\[
\begin{cases}
\uu(0)= \mathbb{P}^n( \f u_0), 
\\
 \vv(0)= \mathbb{P}^n (\f v_0), 
 \\
  \chi(0)= \chi_0,  
  \\
  \omega'(0)= \varpi_0
  \end{cases} \qquad \text{a.e.\ in } \Omega\,.
\]
\end{proposition}
\noindent 
In fact, as a consequence of the a priori estimates from Proposition \ref{prop:loc}, we will improve the above local existence result and show that 
the final time $\widetilde T$  neither depends    on $\delta$ nor  on $\nu$. 

In order to prove the existence of solutions for the discretized-regularized system, we will apply  Schauder's fixed-point argument. \EEE
More precisely, for fixed $ \bar{\chi} \in L^\infty(Q) $   we will solve  the Cauchy problem 
\begin{subequations}\label{galerkinSchauder}
\begin{align}
 \int_\Omega \f v _{t}\cdot \f z  + \left( b({\bar\chi}) \Vm \varepsilon(\f u_t )  + a({\bar\chi} )\CC\varepsilon({\f u})\right)  : \eps{\f z }\dx &= \int_\Omega \f f \cdot \f z \dx   && 
  \label{GalerkinSchauder1}
\\
    \text{for all }\f z \in V^n,   & &&  \text{a.e.\ in } (0,T),
    \nonumber
\\
\f u_t&=\f v   &&   \text{a.e.\ in } (0,T),
 \label{GalerkinSchauder12}\\
\nu  \omega_{tt}  +\omega+ \chi_t +I'_{\delta } ( {\chi_t} ) +\tfrac{1}{2}a'(\bar\chi) \Cm{\f u}{\f u}   +\invbreve{W}'({\bar\chi})-{\bar\chi}&=  0     &&  \text{a.e.\ in }Q,  
\label{GalerkinSchauder2}\\
-\Delta \chi + \breve{W}_\delta'(\chi)+\chi  &= \omega    &&  \text{a.e.\ in } Q\,,  \label{GalerkinSchauder3}
\\
  \partial_{\pmb n}\chi & =0 && \ \text{a.e.\ on } \Sigma\,,
  \label{GalerkinSchauder4}
\end{align}
%
%
\end{subequations}
and prove that the solution operator $\bar \chi \mapsto \chi$ admits a fixed point as soon as $(x,t)\mapsto \bar\chi (x,t)$ is defined on a cylinder $\Omega{\times}
(0,\widetilde T)$ with sufficiently small $\widetilde T$. 

\subsubsection*{The fixed point argument: solving the momentum balance}

Firstly, we solve the discretized momentum balance \eqref{GalerkinSchauder1}--\eqref{GalerkinSchauder12} for fixed $\bar \chi\in L^\infty(Q)$. 
For notational simplicity,
we will consider as a solution operator 
   the mapping $\bar\chi\mapsto \uu$, disregarding  the solution component $\vv$.
\begin{lemma}
\label{l:solvability-u}
 Let 
 $\bar \chi\in L^\infty(Q)$ be fixed. 
 For  every pair $ (\f u_0, \f v_0) \in H^3(\Omega;\R^d) {\times}  H^2(\Omega;\R^d)$ fulfilling \eqref{compat-below}
  there exists a
unique  solution
\begin{equation}
\label{maximal-solution}
  ( \f u, \f v)\in H^1(0,T; V^n{\times} V^n) 
  \end{equation}
   to the Cauchy  problem for system \eqref{GalerkinSchauder1}--\eqref{GalerkinSchauder12}, supplemented with the initial conditions
\[ ( \f u(0), \f v(0))=(\mathbb{P}^n (\f u_0), \mathbb{P}^n  (\f v_0)).
\] 
Moreover, there exists a   function  $\zeta_{\uu} : [0,\infty)^4 \to [0,\infty)$,   monotonously increasing 
w.r.t.\  all of its arguments, 
 such that
\begin{equation}
\label{uLinfty}
\begin{aligned}
& \| \f u\|_{L^\infty(0,T;V^n)} + \| \f v\|_{L^\infty(0,T;V^n)} +
  \| \f v_t\|_{L^2(0,T;V^n)}
  \\
  &
  \leq \zeta_\uu\Big( \| \f f\|_{L^2(0,T;H^1(\Omega)}, \| \bar \chi\|_{L^\infty(Q)}, \| \f u_0\|_{H^3(\Omega)}, \| \f v _0\|_{H^2(\Omega)}\Big) \,,
  \end{aligned}
\end{equation}
and the 
 solution operator 
$
 \mathcal{T}_{\f u}: L^\infty(Q) \to  H^1(0,T; V^n)$  defined by $  \bar \chi \mapsto  \uu$,
 is continuous. 
\end{lemma}
\begin{proof}
    A classical existence theorem (see~\cite[Chapter I, Theorem 5.2]{hale}) 
    ensures that, 
   for every $n\in\N$, 
    there exists a time $T_n^*$ such that there exists a (unique) maximal  solution   $(\f u, \f v) $, in the sense of Carath\'e{}odory,   to the Cauchy problem for  \eqref{GalerkinSchauder1}--\eqref{GalerkinSchauder12}
with 
\[
  ( \f u, \f v)\in \AC([0, \tau]; V^n{\times} V^n) \qquad \text{ for all $0 <\tau< T_n^*$.}
  \]
   With straightforward arguments, based on  the norm-equivalence of all finite-dimensional norms, 
  we obtain that 
  \[
  \begin{aligned}
 &  \| \f u\|_{L^\infty(0,T_n^*;V^n)} +  \| \f v\|_{L^\infty(0,T_n^*;V^n)} + \| \f v_t\|_{L^2(0,T_n^*;V^n)} \\
 & \leq 
\zeta_{\uu}  \Big( \| \f f\|_{L^2(0,T;H^1(\Omega)}, \| \bar \chi\|_{L^\infty(Q)}, \| \mathbb{P}^n( \f u_0)
  \|_{V^n}, \| \mathbb{P}^n (\f v _0)\|_{V^n} \Big) \,.
  \end{aligned}
  \]
  Since
  \begin{equation}
  \label{property-of-projections}
   \| \mathbb{P}^n( \f u_0)
  \|_{V^n}
  \leq c \| \f u_0\|_{H^3(\Omega)}, \qquad 
   \| \mathbb{P}^n (\f v _0)\|_{V^n} \leq c \| \f v_0\|_{H^2(\Omega)},
   \end{equation} 
   the right-hand side in the above estimate does not depend on $n$    and thus the pair $ ( \f u_n, \f v_n)$ extends to the whole interval $[0,T]$. 
   Estimate \eqref{uLinfty}  is then a  consequence of   
   \eqref{property-of-projections} and of the monotonicity of the function $\zeta_{\uu}  $. 
   \par
      The continuity of the solution operator follows  from  estimate \eqref{cont-dep-below}
      below. To prove it, we consider
      system \eqref{GalerkinSchauder1}--\eqref{GalerkinSchauder12},   corresponding to  two given  functions $\bar\chi_1$, $\bar\chi_2$, subtract
       \eqref{GalerkinSchauder1} with  $\bar\chi=\bar\chi_2$ from   \eqref{GalerkinSchauder1} for    $\bar\chi=\bar\chi_1$, \EEE  and test  the obtained relation by $\widehat\vv:= \vv_1{-}\vv_2 = \partial_t \widehat \uu$, with $\widehat \uu :=  \uu_1{-}\uu_2$. 
      Integrating in time and taking into account that
      $\uu_1 (0)= \uu_2 (0)$ and  $ \vv_1 (0)= \vv_2 (0)$ 
we obtain for all $t\in [0,T]$
       \[
       \begin{aligned}
       &
\int_\Omega \tfrac12 |\widehat\vv(t)|^2 \dd x + b_0 \eta_{\Vm} \int_0^t \int_\Omega |\eps{\widehat\vv}|^2 \dd x \dd s 
\\
&
       \leq 
       \int_0^t |\Vm|  \Big\{ 
        \|\eps{\vv_2}\|_{L^2(\Omega)} \| \eps{\widehat\vv}\|_{L^2(\Omega)} \| b(\bar\chi_1){-} b(\bar\chi_2) \|_{L^\infty(\Omega)} 
        \\
        &
        \qquad   \qquad  \qquad   {+}
         \| a(\bar\chi_1)\|_{L^\infty(\Omega)} \| \eps{\widehat\uu}\|_{L^2(\Omega)} \| \eps{\widehat\vv}\|_{L^2(\Omega)}
          \\
        &
        \qquad   \qquad  \qquad 
         {+}  \|\eps{\uu_2}\|_{L^2(\Omega)} \| \eps{\widehat\vv}\|_{L^2(\Omega)} \| a(\bar\chi_1){-} a(\bar\chi_2) \|_{L^\infty(\Omega)} 
         \Big\}
        \dd s
        \\
        &
        \stackrel{(1)}\leq \frac{ b_0 \eta_{\Vm}}2 \| \eps{\widehat\vv}\|_{L^2(\Omega)}^2 +K_1 T  \int_0^t s \left(  \int_0^s 
          \|\eps{\widehat\vv}\|_{L^2(\Omega)}^2   \dd r \right)  \dd s + K_2 \| \bar\chi_1{-} \bar\chi_2\|_{L^\infty(\Omega{\times}(0,T))}^2\,,
       \end{aligned}
       \]
       where {\footnotesize (1)} follows from Young's inequality and from  estimating  $ \| \eps{\widehat\uu(s)}\|_{L^2(\Omega)}^2 \leq s \int_0^s  
       \| \eps{\widehat\vv}\|_{L^2(\Omega)}^2 \dd r.  $
      The constant $K_2$ depends on  $|\Vm|$, on 
       $\max_{|r|\leq M} (|a'(r)|{+}|b'(r)| )$ (with $M =\| \bar\chi_1\|_{L^\infty(Q)} {+}\| \bar\chi_2\|_{L^\infty(Q)} $), and on 
      $ \sup_{t\in [0,T]}  \left(\|\eps{\uu_2(t)}\|_{L^2(\Omega)} {+} \|\eps{\vv_2(t)}\|_{L^2(\Omega)}  \right) $, cf.\ \eqref{uLinfty}. Likewise, the constant $K_1$ also
       depends on 
      $\|\bar\chi_1\|_{L^\infty(\Omega{\times}(0,T))} $.  All in all, with the Gronwall Lemma we conclude that 
      \begin{subequations}
        \label{cont-dep-below}
        \begin{align}
        \int_0^t  \| \eps{\widehat\vv}\|_{L^2(\Omega)}^2 \dd s \leq  \kappa_2 \| \bar\chi_1{-} \bar\chi_2\|^2_{L^\infty(Q)} \exp (\kappa_1 T^2 )
       \intertext{with $\kappa_i= 2 ( b_0 \eta_{\Vm} )^{-1} K_i$ and, a fortiori,  we have for some constant $\kappa_3$}
     \sup_{t\in [0,T]}  \| \eps{\widehat\vv(t)} \|_{L^2(\Omega)}  \leq  \kappa_3 \| \bar\chi_1{-} \bar\chi_2\|_{L^\infty(Q)} \,.
        \end{align}
        \end{subequations}
\end{proof}

\subsubsection*{The fixed point argument: solving the damage flow rule}

We now  solve the approximate flow rule
 \eqref{GalerkinSchauder2}--\eqref{GalerkinSchauder4} for fixed $\bar \chi\in L^\infty(Q)$ and with $\uu = 
   \bar \uu:  = \mathcal{T}_{\f u} (\bar\chi)$. 
   The statement of Lemma \ref{l:4chi} mirrors that of Lemma \ref{l:solvability-u} and, again with slight abuse, we will consider as a solution operator 
   the map $(\bar\chi,\bar\uu)\mapsto \chi$, disregarding  the solution component $\omega$. \EEE
\begin{lemma}
\label{l:4chi}
Let  $\bar \chi \in L^\infty(Q)$  and $ \bar  \uu =  \mathcal{T}_{\f u} (\bar\chi) \in L^\infty(0,T;V^n)$. Set 
\[
\bar
h := \bar\chi -\invbreve{W}'({\bar\chi})  - \tfrac{1}{2}a'(\bar\chi) \Cm{\bar \uu }{\bar \uu }\,,
\]
  and consider the PDE system
\begin{equation}
\label{GalerkinSchauder-DAMAGE}
\begin{aligned}
\nu  \omega_{tt}  +\omega+ \chi_t +I'_{\delta } ( {\chi_t} ) &=  \bar h      &&  \text{a.e.\ in } Q, 
\\
-\Delta \chi + \breve{W}_\delta'(\chi)+\chi  &= \omega    &&  \text{a.e.\ in } Q \,,  
\end{aligned}
\end{equation}
supplemented with the  boundary condition \eqref{GalerkinSchauder4}. 
\par
Then, for every $\chi_0 \in H^2 (\Omega)$ fulfilling \eqref{chidata-strong} and $\varpi_0 \in  L^2(\Omega)$
there exists a unique solution 
\[
 \chi\in \boldsymbol{X}:= W^{1,\infty}(0,T; H^2(\Omega))  \cap W^{2,\infty}(0,T;H^2(\Omega)),   \qquad \omega \in W^{2,\infty}(0,T;L^2(\Omega)),  \EEE
\]
to system \eqref{GalerkinSchauder-DAMAGE}
satisfying  $\chi(0) = \chi_0$ and $\omega_t(0)= \varpi_0$.
\par
Moreover, there exists a   function  $\zeta_{\chi} : [0,\infty)^5 \to [0,\infty)$,   monotonously increasing 
w.r.t.\  all of its arguments, 
 such that
\begin{equation}
\label{chiLinfty}
\begin{aligned}
& 
 \| \chi\|_{W^{1,\infty}(0,T;H^2(\Omega)) {\cap}  W^{2,\infty}(0,T;H^1(\Omega))} + \nu \|\omega \|_{W^{2,\infty}(0,T;L^2(\Omega))} \EEE
  \\
  &
  \leq \zeta_\chi\Big( \frac1\delta,  \| \bar \chi\|_{L^\infty(Q)}, \| \bar \uu\|_{L^\infty(0,T; V^n)},
   \| \chi_0\|_{H^2(\Omega)},
   \| |\partial \breve{W}^{\circ}|(\chi_0) \|_{L^2(\Omega)},
   \nu^{1/2} \| \varpi _0\|_{L^2(\Omega)}\Big)  \EEE \,,
  \end{aligned}
\end{equation}
\EEE and the 
 solution operator 
$
 \mathcal{T}_{\chi} $ mapping  $ (\bar \chi ,\bar u ) \mapsto  \chi $
 is continuous from 
$ L^\infty(Q)\times L^\infty(0,T; V^n) $ to 
 $\boldsymbol{X} $  endowed with the weak$^*$ topology.
\end{lemma}
\begin{proof}
It is rather standard 
    to prove the existence of solutions, e.g.\ by time discretization. That is 
     why,  we focus here mainly on deriving the necessary \textit{a priori} estimates to deduce the regularity 
      $ \chi\in \boldsymbol{X}$ \EEE
for 
      the solution, and estimate \eqref{chiLinfty}. 
    We test equation~\eqref{GalerkinSchauder3} by $\omega_t$, which provides the estimate 
    \[
    \begin{aligned}
       \frac{1}{2}   \frac{\mathrm{d}}{\dt}  \left( \nu \| \omega_t (t)\|_{L^2(\Omega)}^2 {+} \|\omega(t)\|_{L^2(\Omega)}^2 \right)
       &  \leq \| \chi_t {+} I_\delta'(\chi_t) \|_{L^2(\Omega)}  \|\omega_t \|_{L^2(\Omega)} + \| \bar h\|_{L^2(\Omega)}  \|\omega_t \|_{L^2(\Omega)}
       \\
       & \stackrel{(1)}\leq  \left( 1{+} \frac1\delta \right)  \|\chi_t \|_{L^2(\Omega)}  \|\omega_t \|_{L^2(\Omega)} +
    \|\bar h\|_{L^2(\Omega)}  \|\omega_t\|_{L^2(\Omega)}
             \\
       & \stackrel{(2)}\leq  \frac12 \|\bar h\|_{L^2(\Omega)}^2 +    \left( S_0{+} \frac{S_0}\delta {+}\frac12  \right)\EEE  \|\omega_t \|_{L^2(\Omega)}^2\,,
       \end{aligned}
       \]
        where {\footnotesize (1)} follows from the fact that $
       \| I_\delta'(\chi_t)\|_{L^2 (\R)} \leq \frac1\delta   \|\chi_t \|_{L^2(\Omega)} $ 
       by the Lipschitz continuity of $I_\delta'$ and the fact that $I_\delta '(0)=0$,
        cf.\ \eqref{prop-delta-1} ahead, while {\footnotesize (2)}  is a consequence of 
        \eqref{est-omega-dert}.
       Then, with the Gronwall lemma we obtain that 
       \[
    \sup_{t\in [0,T]}   \big(  \nu^{1/2} \EEE
     \|\omega_t(t) \|_{L^2(\Omega)} {+} \|\omega(t)\|_{L^2(\Omega)} \big) \leq \widetilde\zeta
   \Big(\frac1\delta, \| \bar h\|_{L^\infty(0,T; L^2(\Omega))},
   \| \omega(0)\|_{L^2(\Omega)},   \nu^{1/2} \EEE  \| \varpi _0\|_{L^2(\Omega)}\Big) 
       \]
       for some $\widetilde\zeta :  [0,\infty)^4 \to [0,\infty)$,   increasing 
w.r.t.\  all  arguments. 
Taking into account estimates  
\eqref{est-omega},  estimating $\| \bar h\|_{L^\infty(0,T; L^2(\Omega))}$ via
$ \| \bar \chi\|_{L^\infty(Q)} $ and $  \| \bar \uu\|_{L^\infty(0,T; H^2(\Omega))} $, and estimating $\|\omega(0)\|_{L^2(\Omega)}$ via 
\eqref{est-omega-0}, 
 we find 
  \[
    \sup_{t\in [0,T]}   \big( 
 \|\chi(t)\|_{H^2(\Omega)}   {+}   \|\chi_t(t) \|_{H^1(\Omega)} \big) \leq  \overline{\zeta}_\chi^{\delta,\nu} \EEE
       \]
       with the constant   $\overline{\zeta}_\chi^{\delta,\nu}$ depending on the same quantities as in  \eqref{chiLinfty}. 
A comparison argument in 
$\omega_t= -\Delta \chi_t + \breve W''_\delta ( \chi)\chi_t+\chi_t$
(recalling that $\| \breve W''_\delta ( \chi)\|_{L^\infty(\Omega{\times}(0,T))} \leq \frac1\delta$), 
 then allows us to conclude an estimate for 
$-\Delta \chi_t $ in $L^\infty (0,T;L^2(\Omega))$.  
Therefore, $\chi_t $ is estimated in  $L^\infty (0,T;H^2(\Omega))$.
Arguing by comparison in \eqref{GalerkinSchauder12}, \EEE  we ultimately deduce an estimate for 
$\omega_{tt} $ in $L^\infty (0,T;L^2(\Omega))$.  
A fortiori, by  
\eqref{est-omega-dertt} 
and taking into account that $\sup_{t\in [0,T]}
 \|\breve{W}'''(\chi(t))\|_{L^\infty(\Omega)} \|\chi_t(t)\|_{L^3(\Omega)}^2  \leq C$, 
we infer an estimate for $\chi_{tt} $ in $L^\infty (0,T;H^1(\Omega))$.  
Thus, \EEE suitably adapting the right-hand side term   $\overline{\zeta}_\chi^{\delta,\nu}$, \EEE  we conclude 
estimate \eqref{chiLinfty}.
\par
In order to have the solution operator $ \mathcal{T}_{\chi}$ well defined, let us verify that, for given $ \bar h  $ and data $\chi_0 $, $\varpi_0$,   the initial boundary-value
problem for \eqref{GalerkinSchauder-DAMAGE} admits a unique solution. Indeed, let $(\chi_i,\omega_i)$, $i=1,2$, two solution pairs.
Set $\widehat\chi= \chi_1-\chi_2$ and t $\widehat\omega= \omega_1-\omega_2$. 
We subtract system \eqref{GalerkinSchauder-DAMAGE} for $\omega_2$ from \eqref{GalerkinSchauder-DAMAGE} for $\omega_1$, 
thus obtaining 
\begin{equation}
\label{GalerkinSchauder-DAMAGE-diff}
\begin{cases}
 \nu \EEE  \wh\omega_{tt}  +\wh\omega+ \wh\chi_t +I'_{\delta } ( \partial_t\chi_1 ) -I'_{\delta } ( \partial_t\chi_2)  &=  0     
\\
-\Delta \wh\chi+ \breve{W}_\delta'(\chi_1) - \breve{W}_\delta'(\chi_2)+\wh\chi  &= \wh\omega  
\end{cases}
  \qquad   \text{a.e.\ in }Q \,. 
\end{equation}
 We test the first equation by $\wh \omega_t$, while we differentiate in time the second equation and test it by $\wh\chi_t$. Adding the resulting relations and integrating in time and space, we obtain
 \begin{align}
 \nonumber
 &
\frac{\nu}2 \|\wh \omega_t(t)\|_{L^2(\Omega)}^2 + \frac12  \| \wh\omega_t(t)\|_{L^2(\Omega)}^2 +\int_0^t \|\wh \chi_t \|_{H^1(\Omega)}^2  \dd s  
\leq I_1+I_2
\intertext{where, using that $I'_{\delta } $ is Lipschitz continuous with Lipschitz constant $\frac1\delta$,  we estimate}
& 
 \nonumber
I_1= \int_0^t \| I'_{\delta } ( \partial_t\chi_1 ) {-}I'_{\delta } ( \partial_t\chi_2) \|_{L^2(\Omega)}\|\wh \omega_t\|_{L^2(\Omega)} \dd s 
\leq \frac12 \int_0^t \|\wh \chi_t \|_{L^2(\Omega)}^2  \dd s    + \frac1{2\delta}\EEE \int_0^t \|\wh \omega_t \|_{L^2(\Omega)}^2  \dd s,
\intertext{while we have }
&
 \nonumber
I_2 = 
\int_0^t  \int_\Omega \left({-}\breve{W}_\delta''(\chi_1) \partial_t \chi_1 {+}\breve{W}_\delta''(\chi_2) \partial_t \chi_2 \right)
\wh \chi_t  \dd x \dd s   
\\
&   \nonumber
\qquad \leq  - \int_0^t \int_\Omega  \breve{W}_\delta''(\chi_1)  |\wh \chi_t|^2 \dd x \dd s 
+   \int_0^t  \| \breve{W}_\delta''(\chi_2) {-} \breve{W}_\delta''(\chi_1) \|_{L^2(\Omega)} \, \| \partial_t\chi_2 \|_{L^\infty(\Omega)}  \|\wh \chi_t\|_{L^2(\Omega)} \dd s
\\
&   \nonumber \qquad \stackrel{(1)}\leq \frac14 \int_0^t   \|\wh \chi_t\|_{L^2(\Omega)}^2 \dd s + C \int_0^t  \int_0^s   \|\wh \chi_t\|_{L^2(\Omega)}^2 \dd r \dd s \,.
\end{align}
Indeed, {\footnotesize (1)} follows from the convexity of $\breve{W}_\delta$,  from Young's inequality (with  the constant $C$ depending on 
$\|\chi_2 \|_{W^{1,\infty}(0,T;H^2(\Omega))}$), and from estimating   
\[
 \| \breve{W}_\delta''(\chi_2(s)) {-} \breve{W}_\delta''(\chi_1(s)) \|_{L^2(\Omega)}  \leq C \|\widehat \chi(s)\|_{L^2(\Omega)} \leq  C \int_0^s   \|\wh \chi_t(s)\|_{L^2(\Omega)} \dd s \,.
\]
All in all, we obtain
\[
\frac{\nu}2 \|\wh \omega_t(t)\|_{L^2(\Omega)}^2 + \frac12  \| \wh\omega_t(t)\|_{L^2(\Omega)}^2 +\frac14 \int_0^t \|\wh \chi_t \|_{H^1(\Omega)}^2  \dd s  
\leq  \frac1{2\delta} \EEE \int_0^t \|\wh \omega_t \|_{L^2(\Omega)}^2  \dd s + C \int_0^t  \int_0^s   \|\wh \chi_t\|_{L^2(\Omega)}^2 \dd r \dd s\,,
\]
and via  the Gronwall Lemma we conclude the desired uniqueness $\wh \chi=\wh \omega \equiv 0$ a.e.\ in $Q$. 
\par
Finally, let us sketch the proof of the continuity of $ \mathcal{T}_{\chi}$. Consider
$(\bar\chi_n,\bar\uu_n)_n \subset  L^\infty(Q)\times L^\infty(0,T; V^n) $ such that 
$
\bar\chi_n \to \bar\chi_\infty \text{ in } L^\infty(Q)$ and $ \bar\uu_n\to  \bar \uu_\infty  \text{ in } L^\infty(0,T; V^n)\,.
$
Due to estimate \eqref{chiLinfty},  the corresponding sequence $(\chi_n =\mathcal{T}_{\chi}(\bar\chi_n,\bar\uu_n))_n  $
is bounded in $ \Spx = W^{1,\infty}(0,T; H^2(\Omega))  \cap   W^{2,\infty}(0,T;H^1(\Omega))$. \EEE Likewise, the associated  sequence
 $(\omega_n)_n$ is bounded in $W^{2,\infty}(0,T; L^2(\Omega)) $. \EEE
 Hence,
 there exist a pair $(\chi_\infty,\omega_\infty) $ 
 and a subsequence $(n_k)_k$ 
 such that 
$\chi_{n_k} \weaksto \chi_\infty $ in $\Spx$ and $\omega_{n_k} \weaksto \,\omega_\infty $ in  $W^{2,\infty}(0,T; L^2(\Omega)) $. \EEE We standardly check that $(\omega_\infty,\chi_\infty)$ solve 
the Cauchy problem for system \eqref{GalerkinSchauder-DAMAGE} with 
$\bar h_\infty = \bar\chi_\infty  -\invbreve{W}'({\bar\chi_\infty })  - \tfrac{1}{2}a'(\bar\chi_\infty ) \Cm{\bar \uu_\infty }{\bar \uu_\infty  }$.
 Thus,  $\chi_\infty =  \mathcal{T}_{\chi}  ( \bar\chi_\infty,\bar \uu_\infty  )$. 
Since the limit is uniquely identified, a posteriori we have convergence along the whole sequence $(\chi_n )_n$. We have thus shown that 
\[
\bar\chi_n \to \bar\chi_\infty \text{ in } L^\infty(Q) \text{ and } \bar\uu_n\to  \bar \uu_\infty  \text{ in } L^\infty(0,T; V^n)
 \ \Longrightarrow \  \mathcal{T}_{\chi}  ( \bar\chi_n , \bar{\f u}_n) \weaksto  \, \mathcal{T}_{\chi}  ( \bar\chi_\infty, \bar{\f u}_\infty ) \text{ in } 
\Spx\,.
\]
This finishes the proof.
%
%
%
%
\end{proof}
\par
Let us now introduce the operator
\[
\mathcal{T}: L^\infty (Q) \to L^\infty (Q), \qquad \overline\chi \mapsto \mathcal{T}(\bar\chi):= \mathcal{T}_{\chi}(\mathcal{T}_{\uu}(\bar\chi))\,,
\]
and, for 
given $\mathsf{T} \in (0,T]$ and 
$\mathsf{R}>0$,  the notation 
\[
B_{\mathsf{R}}^\infty(\mathsf{T}):=\{ \bar \chi \in L^\infty(\Omega \times (0,\mathsf{T})) \mid \| \bar \chi - \chi_0 \|_{L^\infty(\Omega \times (0,\mathsf{T}))}\leq \mathsf{R} \}\,.
\]
With our next result we will show that there exists $\widetilde T \in (0,T]$  such that, if we restrict  $\mathcal{T}$ a closed ball in 
$ L^\infty (\Omega{\times}(0,\widetilde{T}))$,  $\mathcal{T}$  maps 
$B_{\mathsf{R}}^\infty(\mathsf{T})$
into itself and indeed admits a fixed point, which in fact provides a local-in-time solution to 
the Cauchy problem for 
system  \eqref{eq:regularized}. This concludes the \textbf{\underline{proof of Proposition \ref{prop:loc-exist-approx}}}. 
\begin{lemma}\label{lem:shorttimeexistence}
 Let $ (\f u_0, \f v_0,
\chi_0) \in H^3(\Omega;\R^d) {\times}  H^2(\Omega;\R^d) {\times}
 H^2(\Omega)$ 
 fulfill Hypothesis
 \ref{h:2-strong}.  Let $\varpi_0\in L^2(\Omega)$ be given.
 \par
 Then, for a  suitably chosen \EEE
 $\mathsf{R}>0$ 
 there exists  $\widetilde{T} = \widetilde{T}(\delta,\nu) \in   (0,T]$ such that 
 the operator $\mathcal{T}$ admits a fixed point in  $B_{\mathsf{R}}^\infty(\widetilde T)$. 
\par
As a consequence, the Cauchy problem for 
system  \eqref{eq:regularized} admits a solution $(\uu,\chi,\omega)$ as in \eqref{reg-solution-local-in-time}.
\end{lemma}
\begin{proof}
Combining the continuity  properties of the operator $\mathcal{T}_{\uu}$ with those of $\mathcal{T}_\chi$, we easily check 
that $\mathcal{T}: L^\infty (Q) \to L^\infty (Q)$ is continuous. 
\par
It follows from estimates \eqref{uLinfty} and \eqref{chiLinfty} that there exists a function $\zeta: [0,\infty)^5 \to [0,\infty)$, 
 increasing 
w.r.t.\  all  arguments, such that for every $\bar \chi \in L^\infty(Q)$ there holds
\[
\begin{aligned}
 \| \mathcal{T}(\chi)\|_{W^{1,\infty}(0,T;H^2(\Omega)) {\cap}  W^{2,\infty}(0,T;H^1(\Omega))}  
  \leq \zeta \Big( \frac1\delta,  \frac1\nu, \EEE\| \f f\|_{L^2(0,T;H^1(\Omega)}, \| \bar \chi\|_{L^\infty(Q)}, m_0\Big) \,,
  \end{aligned}
  \]
  where we have set 
  $
  m_0:=     \| \f u_0\|_{H^3(\Omega)} + \| \f v _0\|_{H^2(\Omega)} +
   \| \chi_0\|_{H^2(\Omega)} +
   \| |\partial \breve{W}^{\circ}|(\chi_0) \|_{L^2(\Omega)} +
    \| \varpi _0\|_{L^2(\Omega)}$. 
  Since $W^{1,\infty}(0,T;H^2(\Omega)) {\cap} H^{2}(0,T;L^2(\Omega))$ compactly embeds in $L^\infty(\Omega\times (0,T))$,  we conclude that the operator
  $\mathcal{T}$ is compact. 
  \par
\EEE Finally, let us choose 
\[
\mathsf{R}> 
\zeta_0:= \zeta \Big( \frac1\delta,  \frac1\nu, \EEE \| \f f\|_{L^2(0,T;H^1(\Omega)}, 0,  m_0\Big).
    \]
For any  $\widetilde T \in (0,T] $, 
for every $t\in [0,\widetilde T]$ and $\bar\chi \in  B^\infty_R(\widetilde T)$ we have 
\[
 \begin{aligned}
    \| \mathcal{T}(\bar\chi)(t)- \chi_0\|_{L^\infty(\Omega)} 
 &    =
       \| \chi(t)- \chi_0\|_{L^\infty(\Omega)} 
       \\
       & = \left\|  \int_0^t \chi _t \ds \right\| _{L^\infty(\Omega)} 
       \\
       &
       \leq  t \left\|  \chi _t  \right\| _{L^\infty(\Omega{\times} (0,t))}
\\ &       \leq C_{H^2,L^\infty}  \widetilde{T} \| \chi\|_{W^{1,\infty}(0,\widetilde{T};H^2(\Omega))}
    \\ &        \leq  C_{H^2,L^\infty}  \widetilde{T}  \zeta \Big( \frac1\delta,  \frac1\nu, \EEE   \| \f f\|_{L^2(0,T;H^1(\Omega)}, \| \bar \chi\|_{L^\infty(\Omega\times (0,\widetilde T))}, m_0\Big) 
      \\ &      \stackrel{(1)}\leq  C_{H^2,L^\infty}  \widetilde{T}  \zeta \Big( \frac1\delta,   \frac1\nu, \EEE  \| \f f\|_{L^2(0,T;H^1(\Omega)}, R {+} \|  \chi_0\|_{L^\infty(\Omega)}, m_0\Big)\,,
 \end{aligned}
 \]
 where $C_{H^2,L^\infty} $ is the constant for the continuous embedding $H^2(\Omega) \subset L^\infty(\Omega)$, and 
 {\footnotesize (1)} follows from the estimate 
 \[
 \| \bar \chi\|_{L^\infty(\Omega\times (0,\widetilde T))} \leq \| \bar \chi{-}\chi_0\|_{L^\infty(\Omega\times (0,\widetilde T))} + \|\chi_0\|_{L^\infty(\Omega\times (0,\widetilde T))}  \leq 
 R+  \|\chi_0\|_{L^\infty(\Omega\times (0,\widetilde T))}\,,
 \]
 and the monotonicity of $\zeta$. 
 Hence,
 upon choosing 
 \[
 \widetilde T
 \leq R \, C_{H^2,L^\infty}^{-1}   \,  \zeta \Big( \frac1\delta,   \frac1\nu, \EEE  \| \f f\|_{L^2(0,T;H^1(\Omega)}, R {+} \|  \chi_0\|_{L^\infty(\Omega)}, m_0\Big)^{-1}  
 \]
 we have that 
 \EEE
 $\mathcal{T}(B^\infty_R(\widetilde T)) \subset B^\infty_R(\widetilde T) $. Therefore, we are in a position to 
 apply 
 Schauder's fixed point theorem. This concludes the proof.
\end{proof}


\subsection{ A priori estimates for  the regularized approximate  system}
\label{ss:4.3bis}
 With the following result we rigorously prove the estimates of Prop.\ 
\ref{prop:formal-aprio} for the local-in-time solutions $(\uu,\chi,\omega)$ of the approximate system 
\eqref{galerkinSchauder} (for better readability, we choose to omit the dependence on the parameters
$n$ and $\delta$ in their notation).  Since  estimate \eqref{eq:proplocreg}
 below holds for a constant independent of  $n\in \N$
and  $\delta,\, \nu>0$, we  deduce  that the local-in-time solution
$(\uu,\chi)$
 found in Prop.\ \ref{prop:loc-exist-approx} exists up to a time 
$\widehat{T}$ independent of such parameters.
%
%
%
%

\begin{proposition}[Enhanced local-in-time estimates for the approximate solution]
\label{prop:loc} 
Assume Hypotheses \textbf{\ref{h:1-strong}} and \textbf{\ref{h:2-strong}}, and  let $\Omega$ fulfill  condition \eqref{omega-smooth}. 
Then, there exist a  time  $\widehat{T} \in (0,T]$ such that 
\begin{enumerate}
\item
   for every $n\in \N$
and  $\delta,\, \nu>0$  the solution $(\uu , \chi) $ from Proposition \ref{prop:loc-exist-approx}
 extends to the interval $[0,\widehat{T}]$ with the regularity
 \begin{equation}
 \label{new-reg-0hatT}
 \begin{aligned}
 &
\uu\in H^1(0,\widehat{T};H^3(\Omega;\R^d))\cap W^{1,\infty}(0,\widehat{T};H^2(\Omega;\R^d)),
 \\
 &
		\chi\in   L^{\infty}(0,\widehat{T};H^2(\Omega))\cap H^{1,\infty}(0,\widehat{T};H^1(\Omega))
		\EEE 
		\end{aligned} 
 \end{equation}
  and we have that $\omega =-\Delta \chi + \breve{W}_\delta'(\chi)+\chi  \in  W^{2,\infty}(0,\widehat{T};L^2(\Omega))$;
 \item
 there exists and a function $\boldsymbol{\zeta}: [0,\infty)^5 \to [0,\infty)$, increasing w.r.t. all its arguments, 
 such that for every $n\in \N$
and  $\delta,\, \nu>0$ there holds for all $t\in (0,\widehat{T})$ \EEE
      \begin{align}\label{eq:proplocreg}
      \begin{split}
                &  \| \f u_t(t) \|_{H^2(\Omega)}^2+  
  \| \chi(t) \|_{H^2(\Omega)}^2   +   \|  \omega(t) \|_{L^2(\Omega)}^2  +       \nu \| \omega_t(t)\|_{L^2(\Omega)}^2 +       \nu \| \chi_t(t)\|_{H^1(\Omega)}^2   \EEE
 \\
 & \qquad  +\int_0^t \left(  \| \f u_t (s)\|_{H^3(\Omega)}^2 
  {+} \|  \chi_t(s) \|_{H^1(\Omega)}^2 
  \right) \ds 
 \\
&    \leq{}
\boldsymbol{\zeta}\Big (  \| \f u_0\|_{H^3(\Omega)} ,\| \f v _0\|_{H^2(\Omega)},
   \| \chi_0\|_{H^2(\Omega)} ,
   \| |\partial \breve{W}^{\circ}|(\chi_0) \|_{L^2(\Omega)} ,
    \nu^{1/2} \EEE  \| \varpi _0\|_{L^2(\Omega)}  \Big)  \,.      \end{split}
  \end{align}
   Furthermore, there exists a constant $C$ such that for every $n\in \N$
and  $\delta,\, \nu>0$
  \begin{equation}
  \label{u-dertt-0whT}
  \| \uu_{tt} \|_{L^2(0,\widehat{T};H^1(\Omega))} + \| \breve{W}_\delta'(\chi)
  \|_{L^\infty(0,\widehat{T};L^2(\Omega))}  \leq   C\,. 
  \end{equation}  \EEE
  \end{enumerate}
\end{proposition}
 Clearly,  in view of Lemma \ref{l:props-omega}, the regularity 
$\omega \in   W^{2,\infty}(0,\widehat{T};L^2(\Omega))$ 
leads  to  additional regularity 
for $\chi$. However, we shall  
not  emphasize  it, as it   will not carry over to the limit as $\delta,\, \nu \down 0$. \EEE


%
\begin{proof}
 Here, we revisit  the various claims in the proof of Prop. \ref{prop:formal-aprio} and show how the related calculations can be made rigorous.
\par\noindent
\emph{\textbf{Claim $1$:} The evolution of the mean of $\f u$ is only determined by the given data, \textit{i.e.}, 
\begin{align*}
    \int_\Omega \f u (t)   \dx &= \int_\Omega \f u_0 \dx + t\int_\Omega \f v_0\dx +  \int_0^t\int_\Omega  (t{-}r) \EEE \f f(r) \dr\,.
\end{align*}}
This claim follows exactly as in the proof of Proposition~\ref{prop:formal-aprio} 
by choosing the basis function $\f 1$ as test function in the Galerkin discretization~\eqref{regularized1}.

\medskip

\par\noindent
 \emph{\textbf{Claim $2$:}  there exists a constant $S_{1,1}>0$   such that    estimate 
\eqref{claim1-added}
holds.}
\\
It follows exactly as in the proof of  Proposition~\ref{prop:formal-aprio}. \EEE

\medskip

\par\noindent
\emph{\textbf{Claim $3$:}  there exist a constant $S_{1,2}>0$  and $\overline\beta >1$
 such that    estimate \eqref{ineqBeforGron} holds.}
\par\noindent
In order to rigorously prove  this claim, we use the special choice of the Galerkin basis. 
First of all, we observe that testing~\eqref{regularized1} by $ \di (\Vm{:}\eps{\di(\Vm{:}\eps{\f u_t})})$ is possible, since the choice of our Galerkin basis ensures that $\di (\Vm{:}\eps{\di(\Vm{:}\eps{\f u_t})}) \in V^n$ for $\f u_t \in V^n$. 
Moreover, we observe that the following boundary conditions are fulfilled due to the choice of the Galerkin basis 
\begin{align}    \f n {\cdot} \Vm \eps{\di (\Vm{:}\eps{\f u_t})} = 0 \,,\quad  \f n {\cdot} \Vm \eps{\f u_{tt}} = 0 \quad \text{on }\partial \Omega {\times} (0,T)
\,.  \notag
\end{align}
This follows from the fact that $\f u_t$ and $\f u _{tt}$ are just linear combinations of the basis functions and for all basis functions $\f y_i\in V^n$ with $ i\in \{1,\ldots, n\}$ it holds by construction that $\f n {\cdot} \Vm \eps{\f y_i} = 0$. Moreover, the basis functions are eigenfunctions of the operator~\eqref{boundaryvalueproblem}, so  that for any  $\f y_i\in V^n $ also $\di (\Vm{:}\eps{\f y_i}) = \lambda_i \f y_i$ fulfills the associated boundary condition, \textit{i.e.}, $$\f n {\cdot} \Vm \eps{\di (\Vm{:}\eps{\f y_i})} =\lambda_i  \f n {\cdot} \Vm \eps{\f y_i} = 0 \quad\text{on }\partial \Omega {\times} (0,T)\,.  $$ 
With this observation, all  the formal calculations of \textit{\textbf{Claim} 2} can be performed  rigorously and all boundary terms are null. 

\medskip

\par\noindent
\emph{\textbf{Claim $4$:}  there exist a constant $S_{1,3}>0$ and   $\underline\beta >1$
 such that for almost all $t\in (0,T)$}  
 \begin{equation}
\label{testedchireg}
    \begin{aligned}
        &
      \nu \| \omega_t(t)\|_{L^2(\Omega)}^2  + 
          \| \omega(t)\|_{L^2(\Omega)}^2 +   \| \chi(t)\|_{H^2(\Omega)}^2  +\int_0^t \left( \|\chi_t\|_{L^2(\Omega)}^2{+}  \|\nabla \chi_t\|_{L^2(\Omega)}^2 \right) \dd s 
        \\ & 
\leq S_{1,3}(1{+}     \| \varpi_0\|_{L^2(\Omega)}^2 \EEE  {+}  \| \chi_0\|_{H^2(\Omega)}^2{+}  \| |\partial \breve{W}^{\circ}|(\chi_0) \|_{L^2(\Omega)}^2 \EEE
 {+} \| \uu_0\|_{H^2(\Omega)}^8) 
\\
& \quad  +S_{1,3} \int_0^t  \left( \| \omega\|_{L^2(\Omega)}^8  {+} \|  \uu_t \|_{H^2(\Omega)}^8 {+}  \int_0^s \|  \uu_t \|_{H^2(\Omega)}^8 \dd \tau  {+} \|\chi\|_{H^2(\Omega)}^{\underline\beta} \right) \dd s \,.   
\end{aligned}
\end{equation}
\medskip 
\par\noindent 
 In fact, thanks to  Lemma~\ref{lem:shorttimeexistence},   for a solution to the system~\eqref{eq:regularized} we have the regularity property
 $\omega\in W^{2,\infty}(0,T;L^2(\Omega))$
 (although  we have 
 $
  \| \omega\|_{W^{2,\infty}(0,\widehat{T};L^2(\Omega))} \leq C(\delta,\nu)
$
for a  positive constant $C(\delta,\nu)$, with $C(\delta,\nu) \uparrow +\infty$ as $\delta,\, \nu \downarrow 0$).
 Hence, $\omega_t$ is an admissible test function for  equation~\eqref{regularized2}. \EEE
The same arguments as in the proof of Proposition~\ref{prop:formal-aprio} can now be followed step by step in order to derive the estimate~\eqref{testedchireg}. 
\medskip 
\par\noindent 
Combining the estimates from Claim $2$ and $3$,
 we obtain the analogue of 
inequality~\eqref{ineqrelative},  with  the same constants $ S_{1,4} $ and $\beta$, 
 but with the additional term $ \nu \| \omega_t(t)\|_{L^2(\Omega)}^2 $  on the left-hand side. 
Recall that  $ \nu \| \omega_t(t)\|_{L^2(\Omega)}^2  \geq c \nu  \| \chi_t(t)\|_{H^1(\Omega)}^2 $ 
by Lemma \ref{l:props-omega}. 
Hence,  the very  same local-in-time Gronwall-type estimate  as in the proof of Proposition~\ref{prop:formal-aprio} allows us 
to deduce estimate \eqref{eq:proplocreg}.  
Estimate \eqref{u-dertt-0whT} 
for $\uu_{tt}$
then follows in view of \eqref{uLinfty}, by the equivalence of all finite-dimensional norms, while the bound for $\breve{W}_{\delta}'(\chi)$ follows from that for
$\omega $ in $L^{\infty}(0,\widehat{T};L^2(\Omega))$, arguing by comparison. 
  
\medskip 
\par\noindent 
Since the involved constants are independent of $n\in \N$  and of $\delta,\, \nu >0$,  by a standard prolongation argument
we obtain that the solution found in Prop.\ \ref{prop:loc-exist-approx} extends  to an interval $(0,\widehat{T})$
 independent of all parameters, on which 
estimate  \eqref{eq:proplocreg} holds. 
This concludes the proof.
\end{proof}

\subsection{Limit passage in the regularized system and conclusion of the proof of Theorem \ref{thm:2}}
\label{ss:4.4}  
 We split the argument in some steps.
Let us mention in advance that we shall resort to Proposition \ref{prop:consistency}: thus, we will show  that the limiting pair $(\uu,\chi)$ 
fulfills the damage flow rule pointwise a.e.\ in $Q$ by proving  the variational inequality  \eqref{weakChiIneqSub} and  the energy-dissipation inequality \eqref{UEDI}. 
\par
For the compactness argument below we recall that, for a 
given   reflexive  space $\Spx$,   convergence in the  space $\rmC^0([0,S]; \Spxw)$  is, by definition,   convergence in  $\rmC^0([0,S]; (\Spx,d_\mathrm{weak}))$, where the metric $d_\mathrm{weak}$ induces the weak topology on a closed bounded set of $\Spx$.
\medskip

\par\noindent
\emph{\textbf{Step $1$: compactness.}}   \EEE
    Since the \textit{a priori} estimate~\eqref{eq:proploc} holds independently of 
    the parameters $n\in \N $ and $\delta,\nu>0$, 
 we may choose two sequences  
 \begin{equation}
 \label{scaling-condition}
  \delta_n \down 0  \text{ and } \nu_n \down 0  \text{ such that } \frac{\nu_n^{1/2}}{\delta_n} \to 0  \text{ as } n \to \infty \,.
  \end{equation}
  We also consider  a sequence $(\varpi_0^n)_n\subset L^2(\Omega)$ of initial data such that 
  \begin{equation}
  \label{scaling-initial-data}
\nu_n^{1/2} \| \varpi_0^n\|_{L^2(\Omega)} \longrightarrow 0  \text{ as } n \to \infty \,.
  \end{equation}
Correspondingly, by Proposition \ref{prop:loc} we find  a sequence
  of solutions 
   $(\uu_{n,\delta_n,\nu_n},\chi_{n,\delta_n,\nu_n})_n $, 
   hereafter simply denoted as $(\uu_n,\chi_n)_n$,  with associated $\omega_n = -\Delta\chi_n +\breve{W}_{\delta_n}'(\chi_n)+\chi_n$,
   ad a quadruple $(\uu,\chi,\omega,\xi)$, for which,   along a (not-relabeled)  subsequence,  \EEE
the weak-convergences associated with  the bounds \eqref{eq:proplocreg}  hold, namely        \begin{subequations}
		\label{weak-cvg-loc}
		\begin{align}
			 \uu_{n}&\weaksto  \uu
				&&\text{ weakly-star in } H^2(0,\widehat{T}; H^2(\Omega;\R^d)) \EEE {\cap} W^{1,\infty}(0,\wh T;H^2(\Omega;\R^d)){\cap} H^{1}(0,\wh T;H^3(\Omega;\R^d)),\\
			\chi_{n}&	\weaksto  \chi
				&&\text{ weakly-star in }L^\infty(0,T;H^2(\Omega))\cap H^1(0,T;H^1(\Omega)),\\
			\omega_n & \weaksto \omega
				&&\text{ weakly-star in }L^\infty(0,T;L^2(\Omega))
    \\
			\breve{W}_{\delta_n}'(\chi_n) &  \weaksto \xi 
				&&\text{ weakly-star in }L^\infty(0,T;L^2(\Omega))
    \,.
		\end{align}
		\end{subequations}
		 Furthermore, 
		by well-known compactness results, we gather the strong convergences
		 \begin{subequations}
		\label{strong-cvg-loc}
		\begin{align}
		\label{strong-cvg-loc-1}
		 \partial_t\uu_{n}&\weaksto  \partial_t\uu
				&&\text{ strongly in }\rmC^0([0,\widehat T];H^2(\Omega;\R^d)_{\mathrm{weak}})\,,\\
			\label{strong-cvg-loc-2}	\chi_{n}&\weaksto  \chi
				&&\text{ strongly in } \rmC^0([0,\widehat T];H^2(\Omega)_{\mathrm{weak}})\,.\\
		\intertext{Finally, from \eqref{eq:proplocreg} we also deduce  an estimate  for $(\nu_n^{1/2}  \omega_n )_n \subset W^{1,\infty} (0,\widehat T;L^2(\Omega))$, so that}
		& \label{strong-cvg-loc-3}
		\nu_n \omega_n \to 0 &&\text{ strongly in } W^{1,\infty} (0,\widehat T;L^2(\Omega))\,.
			\end{align}
		\end{subequations}
		By the weak 
lower semi continuity of the involved 
 norms, 
  we may take the limit in estimate \eqref{eq:proplocreg} and  deduce that its analogue holds
  at least for almost all  $t\in (0,\widehat T)$. Indeed, since  $ \uu_t \in\rmC^0([0,\widehat T];H^2(\Omega;\R^d)_{\mathrm{weak}}) $
   and  $ \chi \in    \rmC^0([0,\widehat T];H^2(\Omega)_{\mathrm{weak}}) $,
   we ultimately have that the pointwise estimates   \eqref{eq:proplocreg} 
   for $\uu_t$ and $\chi$
 hold for all times.  
\medskip

\par\noindent
\emph{\textbf{Step $2$: momentum balance.}}   \EEE
Using the convergences~\eqref{weak-cvg-loc} 
  and \eqref{strong-cvg-loc}, \EEE  it  is a standard manner  to pass to the limit in the Galerkin approximation~\eqref{regularized1}. 
 In this way, we deduce that the pair $(\uu,\chi)$  satisfies the momentum balance   pointwise a.e.\ in $Q$ \EEE
\medskip

\par\noindent
\emph{\textbf{Step $3$: variational inequality  \eqref{weakChiIneqSub}.}}   \EEE
Multiplying the regularized flow rule \eqref{regularized2} by a test function  $\psi  \in \rmC^1_{\mathrm{c}}(0,\widehat T) {\otimes} L^2(\Omega)$ such that $ \psi \leq 0$ a.e. in $Q$ and integrating
in space and time, 
 we find 
\begin{align}\label{weakchireg}
\begin{split}
        \iint_Q \left( \partial_t\chi_n - \Delta \chi_n + \breve W'_\delta (\chi_n) + \invbreve W'(\chi_n) + \frac{1}{2}a'(\chi_n) \Cm{\uu_n}{\uu_n} \right ) \psi -\nu \partial _t\omega _n \partial_t \psi \dx \ds &=\\ - \iint_Q  I'_\delta (\partial _t \chi_n) \psi \dx \ds &\geq 0 \,. 
\end{split}
\end{align}
The last term on the left-hand side vanishes in the limit  as $n\to\infty$  due to \eqref{strong-cvg-loc-3}.  \EEE
\[ \nu_n \EEE \int_Q\partial _t\omega _n \partial \psi \dx \ds \leq   \nu_n \EEE  \| \partial _t\omega _n \|_{L^\infty(0,\widehat T;L^2(\Omega))}  \| \partial_t \psi\|_{L^1(0,\widehat T ; L^2(\Omega))} \longrightarrow 0 \text{ as } n \to \infty\,.
\]
Passing to the limit   in the remaining terms  on the left-hand side of  inequality~\eqref{weakchireg} is now a standard procedure in view  of
convergences \eqref{weak-cvg-loc} and \eqref{strong-cvg-loc}.  In particular, combining the weak convergence of $ \breve{W}_{\delta_n}'(\chi_n) $ 
with 
 the strong   convergence \eqref{strong-cvg-loc-2} for $\chi_n$ we conclude  \EEE
that  $ \xi \in \partial \breve W (\chi)$ a.e.~in $Q$. In the limit of the inequality~\eqref{weakchireg}, we infer that~\eqref{weakChiIneqSub} is fulfilled. 
\medskip

\par\noindent
 \emph{\textbf{Step $4$:  energy-dissipation inequality \eqref{UEDI}.}} First of all, we observe that  an approximate version of \eqref{UEDI} holds for system
\eqref{eq:regularized}. Indeed, testing \eqref{regularized1}  by 
$ \f z = \partial_t \uu_n$, multiplying \eqref{regularized2} multiplied by $\partial_t\chi_n$, adding the obtained relations and  integrating in time leads to 
\begin{equation}\label{energyinreg}
\begin{aligned}
  &  \mathcal{E}_{\delta_n}( \uu_n(t) , \chi_n(t)  , \partial_t\uu_n(t) ) +\int_0^t  \mathcal{D}(\chi_n(s) ,\partial_t\uu_n(s) ,\partial_t \chi_n(s)  ) \dd s + \nu_n \int_0^t \int_\Omega \partial_{tt}\omega_n(s)  \partial_t \chi_n(s) \dd s  \dx 
  \\
  &
  =   \mathcal{E}_{\delta_n}( \uu_n(0) , \chi_n(0)  , \partial_t\uu_n(0) ) +    \iint_Q \f f {\cdot} \partial_t \uu_n \dx  \dd s \,,
\end{aligned}
\end{equation}
featuring the regularized energy and dissipation functionals 
	\begin{align}
	\label{def:en-reg}  
	\cE_{\delta_n}(\uu,\chi,\ut):={}&\io \left\{  \frac12|\ut|^2 {+}  \frac12 \EEE a(\chi) \CC \e(\uu){:} \e(\uu) {+} \frac12|\nabla\chi|^2{+}\breve{W}_{\delta_n}(\chi)
	{+}\invbreve{W}(\chi)
	\right\} \dx
  \,, 
					\\
					\label{def:diss-reg}
	 \cD(\chi,\ut,\chit):={}&\io \left\{ b(\chi)\VV\e(\ut){:}\e(\ut) {+} |\chit|^2{+}I_{\delta_n}(\chit) \right\} \dx
	    \,.
				\end{align}
  For the last term on the left-hand side, we find 
  \[
  \begin{aligned}
  &
           \nu_n  \int_0^t \int_\Omega  \partial_{tt}\omega_n(s)  \partial_t \chi_n(s) \dd s  \dx  \\
           &=\nu_n 
            \int_0^t 
           \int_\Omega \partial _t \left( - \Delta \partial_t \chi_n  {+} \breve W''_{\delta_n} (\chi_n) \partial_t\chi_n {+} \partial_t \chi_n \right) \partial_t \chi_n \dx  \dd s 
      \\ 
       &=\nu_n 
            \int_0^t 
           \int_\Omega \left( - \Delta \partial_{tt} \chi_n  {+} \breve W'''_{\delta_n} (\chi_n)\partial_t\chi_n \partial_t\chi_n {+} \breve W''_{\delta_n} (\chi_n) \partial_{tt} \chi_n {+}
             \partial_t^2 \chi_n \right) \partial_t \chi_n  \EEE \dx  \dd s 
      \\ 
      &=\nu_n 
           \int_0^t  \int_\Omega \frac{\mathrm{d}}{\dt}\frac{1}{2} \left\{  |\nabla \partial_t \chi_n|^2 {+} | \partial_t \chi_n|^2 
     \right\} {+} \breve W'''_{\delta_n} (\chi_n)(\partial_t\chi_n )^3 {+} \breve W''_{\delta_n} (\chi_n) \partial_{tt} \chi_n \partial_t\chi_n \dx \dd s  
           \\
      & =               \int_0^t  \frac{\mathrm{d}}{\dt} \left\{  \frac{\nu_n}{2 }\int_\Omega |\nabla \partial_t \chi_n|^2 {+} | \partial_t \chi_n|^2 {+}
        | \sqrt{\breve W''_{\delta_n} (\chi_n)}\partial_t \chi_n|^2 \dx \right\} \dd s  
      \\
      & \qquad + \frac{\nu_n}{2} \int_0^t \int_\Omega  \breve W'''_{\delta_n} (\chi_n) |\partial_t \chi_n|^2 \partial_t \chi_n \dx \dd s  \,,
  \end{aligned}
  \]
  where we used the fact that  $\breve W''_{\delta_n} \geq 0$ and that 
  \[
 \partial _t \frac{\nu_n}{2 } | \sqrt{\breve W''_{\delta_n} (\chi_n)}\partial_t \chi_n|^2  = \frac{\nu_n}{2 }  \breve W'''_{\delta_n} (\chi_n) |\partial_t \chi_n|^2  \partial_t \chi_n +
 \nu_n \breve W''_{\delta_n} (\chi_n) \partial_{tt}\chi_n \partial_t \chi_n                       \qquad \aein \, Q\,. 
\]
 Thus, 
\eqref{energyinreg} rephrases as
\[
\begin{aligned}
  &  \mathcal{E}_{\delta_n}( \uu_n(t) , \chi_n(t)  , \partial_t\uu_n(t) ) +\int_0^t  \mathcal{D}(\chi_n(s) ,\partial_t\uu_n(s) ,\partial_t \chi_n(s)  ) \dd s + 
  \mathcal{V}_n(\chi_n(t),\partial_t\chi_n(t))
 \\
  &
  =   \mathcal{E}_{\delta_n}( \uu_n(0) , \chi_n(0)  , \partial_t\uu_n(0) )  
 +
   \int_0^t \int_\Omega  \f f {\cdot} \partial_t \uu_n \dx  \dd s
   \\ &\quad+ 
  \mathcal{V}_n(\chi_n(0),\partial_t\chi_n(0))+ \frac{\nu_n}{2} \int_0^t \int_\Omega  \breve W'''_{\delta_n} (\chi_n) |\partial_t \chi_n|^2 \partial_t \chi_n \dx \dd s 
   \,,
\end{aligned}
\]
where we have used the place-holder
\[
 \mathcal{V}_n(\chi_n,\partial_t\chi_n) = 
\frac{\nu_n}{2 } \| \partial_t\chi_n\|_{H^1(\Omega)}^2  + \frac{\nu_n}2 \int_\Omega \left| \sqrt{\breve W''_{\delta_n} (\chi_n)} \partial_t\chi_n \right|^2 \dx \,.
\]
Now,  observe that 
\[
\begin{aligned}
  \mathcal{V}_n(\chi_n(0),\partial_t\chi_n(0))  & = \frac{\nu_n}{2 }  \int_\Omega \partial_t\omega_n (0) \partial_t \chi_n(0) \dd x 
  \\
  & \leq 
   \frac{\nu_n}{2 }  \|  \partial_t\omega_n (0) \|_{L^2(\Omega)}  \|  \partial_t\chi_n (0) \|_{L^2(\Omega)}   \stackrel{(1)}{\leq}
    \frac{\nu_n S_0}{2 }  \|  \partial_t\omega_n (0) \|_{L^2(\Omega)}^2
      \stackrel{(2)}{\longrightarrow } 0 
      \end{aligned}
\]
as $n \to\infty$,
where in (1), we have used \eqref{est-omega-dert} and in (2), we resorted to condition \eqref{scaling-initial-data} for $ \partial_t\omega_n (0) = \varpi_0^n$. Furthermore,
by  \eqref{properties-W-delta}
 we have  $| \breve W'''_{\delta_n} (\chi_n)| \leq \frac1{\delta_n^3}$ a.e.\ in $Q$, therefore  we infer that 
\[ 
\frac{\nu_n}{2} \int_0^t \int_\Omega  \breve W'''_{\delta_n} (\chi_n) |\partial_t \chi_n|^2 \partial_t \chi_n \dx \dd s
\leq \frac12 \nu_n^{1/2} \| \breve W'''_{\delta_n}(\chi_n) \|_{L^\infty(Q)}  \nu_n^{1/2}  \| \partial_t \chi_n\|_{L^3((0,\widehat{T}){\times}\Omega)}^3
\longrightarrow  0 
\]
as $n \to\infty$, 
where the last assertion follows from combining the bound for $\nu_n^{1/2}  \| \partial_t \chi_n\|_{L^\infty(0,\widehat T;H^1(\Omega))}$ from \eqref{eq:proplocreg}, with the scaling condition \eqref{scaling-condition}. In turn, we immediately see that for every $t\in (0,\widehat T]$
\[
\begin{aligned}
 \mathcal{E}( \uu(t) , \chi(t)  , \partial_t\uu(t) )& \leq \liminf_{n\to\infty}
 \mathcal{E}_{\delta_n}( \uu_n(t) , \chi_n(t)  , \partial_t\uu_n(t) )\,,
 \\
 \int_0^t  \mathcal{D}(\chi(s) ,\partial_t\uu(s) ,\partial_t \chi(s)  ) \dd s   &\leq \liminf_{n\to\infty} 
 \int_0^t  \mathcal{D}(\chi_n(s) ,\partial_t\uu_n(s) ,\partial_t \chi_n(s)  ) \dd s\,,
 \\
  \mathcal{E}_{\delta_n}( \uu_n(0) , \chi_n(0)  , \partial_t\uu_n(0) ) &\longrightarrow  \mathcal{E}( \uu_0 , \chi_0  , \vv_0 )\,,
  \\
    \int_0^t \int_\Omega  \f f {\cdot} \partial_t \uu_n \dx  \dd s  &\longrightarrow   \int_0^t \int_\Omega  \f f {\cdot} \partial_t \uu \dx  \dd s  \,.
  \end{aligned}
\]
%
%
All in all, sending $n\to \infty$ in \eqref{energyinreg}  we find 
 that the energy inequality~\eqref{UEDI} holds, in the limit, on $[0,t]$ for all $t \in (0,\widehat T]$. 
 By Proposition \ref{prop:consistency}, 
this completes the proof of  Theorem \ref{thm:2}.
  \QED 
  \EEE

\section{Relative energy inequality}
\label{s:rel}
This Section is devoted to the proof of Theorem~\ref{thm:3}.  The key result is Proposition \ref{prop:weakstrong}, 
where \EEE
 we will compare a weak solution 
$(\uu,\chi)$ (to the initial-boundary value problem for system \eqref{PDEsystem} with  the homogeneous Neumann boundary condition
\eqref{homogNeu}), and a strong solution $(\tu,\tchi)$  in terms of the following quantities: 
\begin{itemize}
\item[-]
 the relative energy 
\begin{align}
\label{defR}
\mathcal{R}( \f u , \chi , \f u _t | \tu, \tchi,\tut) :={}& \int_\Omega \frac{1}{2}| \nabla \chi-\nabla \tchi |^2 + W( \chi )- W (\tchi)- W'(\tchi)(\chi -\tchi) + \frac{\ell}{2}| \chi-\tchi|^2 \dx
\\
\nonumber
&+ \int_\Omega  \frac{1}{2} a(\chi ) \Cm{\f u-\tu}{\f u -\tu}  + \frac{1}{2}| \f u _t -\tut|^2 \dx\,,
\end{align}
 where $\ell \geq 0$ is such  that $r\mapsto W(r) + \tfrac{\ell}{2} |r|^2$ is convex, cf.\ \eqref{assW-2-new}, and  
\item[-] the relative dissipation
\begin{align}\label{defW}
\mathcal{W}( \chi, \f u _t, \chi_t | \tut,\tchi_t) :={}& \int_\Omega | \chi_t - \tchi_t|^2 + b(\chi) \mathbb{V}\varepsilon(\f u_t- \tut) {:} \varepsilon(\f u_t-\tut )\dx 
\,,
\end{align}
 where we have omitted the terms $\int_\Omega  I _{(-\infty, 0])} ( \chi_t) + I _{(-\infty, 0]) }( \tchi_t))  \dx $ as they will be  null as soon as they are evaluated along a  weak and a strong solution. \EEE
\end{itemize}
 Indeed, $\mathcal{R}$ and $\mathcal{W}$ will be  involved in the Gronwall-type inequality \eqref{REI} below, which will be the core ingredient in the proof of Thm.\ \ref{thm:3}. 
\EEE
\begin{proposition}\label{prop:weakstrong}
    Let Hypothesis~\textbf{\ref{h:4}} be fulfilled and let $(\f u , \chi)$ be a weak solution to the Cauchy  problem  for 
    system \eqref{PDEsystem} in the sense of Definition~\ref{def:weakSol} and $(\tu,\tchi)$ a strong solution in the sense of  Definition~\ref{def:strongSol}.  Then the relative energy-inequality
    \begin{equation}
    \label{REI}
    \tag{$\mathrm{REI}$}
    \begin{aligned}
        \mathcal{R}( \f u , \chi , \f u _t | \tu, \tchi,\tut) (t) + &  {}\int_0^t\left[\mathcal{W}( \chi, \f u _t, \chi_t | \tut,\tchi_t){-} \int_\Omega a'(\chi) \tchi_t \Cm{\f u {-}\tu}{\f u {-}\tu} \dx \right]\mathrm{e}^{\int_s^t\mathcal{K}(\tu,\tchi)\dta}  \ds 
        \\  & \leq \mathcal{R}(\f u(0),\chi(0),\f u_t(0)| \tu(0),\tchi(0),\tu_t(0)) \mathrm{e}^{\int_0^t\mathcal{K}(\tu,\tchi)\ds}
    \end{aligned}
    \end{equation}
    holds for a.e.~$t\in(0,T)$, where $\mathcal{K} $ is given by 
    \begin{align}
    \hspace{-0,5cm}    \mathcal{K}(\tu,\tchi) :=  C_{\mathrm{REI}} \EEE \left(\| \tchi_t  \|_{L^{3/2}(\Omega)}
         {+}  \| \varepsilon(\tut)\|_{L^3(\Omega)}^2 \EEE
        {+}  \ell^2  {+}\| \eps{\tu}\|_{L^\infty(\Omega)}^2 
         {+} \| \varepsilon(\tu) \| _{L^3(\Omega)}^2 {+}  \| \varepsilon(\tu) \| _{L^6(\Omega)}^4 \right) 
    \end{align}
     for some positive constant $ C_{\mathrm{REI}} >0$ only depending on the problem data. \EEE
\end{proposition}
\begin{proof}

For the elastic energy density $\frac{1}{2} a(\chi)  \mathbb{C}\varepsilon(\f u ) {:} \varepsilon (\f u) $, we observe by some calculations
\begin{align*}
\frac{1}{2}\int_\Omega a(\chi) \mathbb{C}\varepsilon(\f u{-}\tu)  {:}  \varepsilon(\f u{-}\tu)\dx 
={}& \frac{1}{2}\int_\Omega a(\chi) \mathbb{C}\varepsilon(\f u)  {:} \varepsilon(\f u){+} a( \tchi)  \mathbb{C}\varepsilon(\tu)  {:}  \varepsilon(\tu)  \dx 
\\
&- \frac{1}{2} \int_\Omega 2  a( \chi )  \mathbb{C}\varepsilon(\f u)  {:}  \varepsilon(\tu) - (a(\chi) - a(\tchi)) \mathbb{C}\varepsilon(\tu) {:} \varepsilon(\tu)\dx \,.
\end{align*}
 We now evaluate the second line between $0$ and $t$. We have  
\begin{align*}
    - \frac{1}{2} \int_\Omega& \left[ 2  a( \chi )  \mathbb{C}\varepsilon(\f u)  {:}  \varepsilon(\tu) {-} (a(\chi) {-} a(\tchi)) \mathbb{C}\varepsilon(\tu) {:} \varepsilon(\tu)\dx \right]\Big |_0^t \\&=- \int_0^t \int_\Omega  \Big[  a'(\chi) \chi_t \Cm{\f u}{\tu}{-} \frac{1}{2}\partial_t ( a(\chi) {-}a(\tchi)) \Cm{\tu}{\tu} \Big] \dx \ds
    \\
   &\quad  - \int_0^t \int_\Omega \Big[ a(\chi) \Cm{\f u_t}{\tu} {{+}} a(\chi) \Cm{\f u}{\tu_t} {-} 
   ( a(\chi){-}a(\tchi)) \Cm{\tu}{\tu_t} \Big] \dx \ds
   \intertext{and with algebraic manipulations we easily obtain}
   & = - \int_0^t \int_\Omega \Big[ \frac{1}{2}\chi_t a'(\tchi) \Cm{\tu}{\tu} {{+}} \frac{1}{2}\tchi_t a'(\chi) \Cm{\f u}{\f u} \Big]\dx \ds 
   \\& \quad - \ito  \Big[ a(\chi) \Cm{\f u}{\tu_t} {{+}} a(\tchi) \Cm{\tu}{\f u_t }  \Big]\dx\ds 
   \\
& \quad+ \int_0^t \int_\Omega \Big[ \frac{1}{2} \partial_t (a(\chi) {-} a(\tchi) ) \mathbb{C}\varepsilon(\tu) {:} \varepsilon(\tu) {-} a'(\chi)\chi_t \mathbb{C}\varepsilon(\f u ) {:} \varepsilon(\tu) 
\Big] 
\dx \ds 
\\
&\quad  - \int_0^t \int_\Omega  \Big[ a(\chi) \Cm{\f u_t}{\tu} {{+}} a(\chi) \Cm{\f u}{\tu_t} {-} 
( a(\chi){-}a(\tchi)) \Cm{\tu}{\tu_t} \Big] \dx \ds
   \\&\quad + \int_0^t\int_\Omega \Big[
\frac{1}{2}\chi_ta'(\tchi)\Cm{\tu}{\tu}{+} \frac12  \tchi_t a'( \chi)  \Cm{\f u }{\f u} \Big]\dx \ds
\\
&\quad +\int_0^t\int_\Omega 
\Big[ a(\chi) \Cm{\f u }{\tut}{+} a(\tchi) \Cm{\tu}{\f u_t} \Big]
\dx \ds
\\& = - \int_0^t \int_\Omega \Big[ \frac{1}{2}\chi_t a'(\tchi) \Cm{\tu}{\tu} {+} \frac{1}{2}\tchi_t a'(\chi) \Cm{\f u}{\f u} \Big] \dx \ds 
   \\& \quad - \ito \Big[ a(\chi) \Cm{\f u}{\tu_t} {+} a(\tchi) \Cm{\tu}{\f u_t } \Big] \dx\ds 
   \\
& \quad+ \int_0^t \int_\Omega \Big[ \frac{1}{2} \partial_t (a(\chi) {-} a(\tchi) ) \mathbb{C}\varepsilon(\tu) {:} \varepsilon(\tu) {-} a'(\chi)\chi_t \mathbb{C}\varepsilon(\f u ){:} \varepsilon(\tu)  
\Big]
\dx \ds 
\\
&\quad  + \int_0^t \int_\Omega \Big[ a(\tchi) \Cm{\tu}{\f u_t}{-}a(\chi) \Cm{\f u_t}{\tu}  {+} 
( a(\chi){-}a(\tchi)) \Cm{\tu}{\tu_t} \Big] \dx \ds
   \\&\quad +\int_0^t\int_\Omega 
\frac{1}{2} \Big[ \chi_ta'(\tchi)\Cm{\tu}{\tu}{+} \tchi_t a'( \chi)  \Cm{\f u }{\f u} \Big] \dx \ds
   \,.
\end{align*}
For the last three lines, we observe
\begin{align*}
  & \frac{1}{2} \partial_t (a(\chi) {-} a(\tchi) ) \mathbb{C}\varepsilon(\tu) {:} \varepsilon(\tu) - a'(\chi)\chi_t \mathbb{C}\varepsilon(\f u ) {:} \varepsilon(\tu)  
\\
&+
  a(\tchi) \Cm{\tu}{\f u_t} -  a(\chi) \Cm{\f u_t}{\tu}  +
  ( a(\chi){-}a(\tchi)) \Cm{\tu}{\tu_t}
   \\&+
\frac{1}{2}\left(\chi_ta'(\tchi)\Cm{\tu}{\tu}+ \tchi_t a'( \chi)  \Cm{\f u }{\f u} \right)
\\
&=  (a(\chi) {-} a(\tchi) ) \mathbb{C}\varepsilon(\tu) {:} \varepsilon(\tut) - a(\chi) \mathbb{C}\varepsilon(\tu ) {:} \varepsilon(\f u_t)+  a(\tchi)  \Cm{\tu}{\f u_t}
\\&+\frac{1}{2}\left[ 
(a'(\chi) \chi_t {-} a'(\tchi) \tchi_t) \Cm{\tu}{\tu} 
{-} 2 a'(\chi )\chi_t \Cm{\f u }{\tu}\right]
\\
&+\frac{1}{2} \left[a'(\chi) \tchi_t \Cm{\f u }{\f u} +a'(\tchi)\chi_t\Cm{\tu}{\tu}\right]\,.
\end{align*}
 All in all, we have calculated
\[
\begin{aligned}
&\frac{1}{2}\int_\Omega a(\chi) \mathbb{C}\varepsilon(\f u{-}\tu)  {:}  \varepsilon(\f u{-}\tu)\dx \Big |_0^t
\\
& = \frac{1}{2}\int_\Omega a(\chi) \mathbb{C}\varepsilon(\f u)  {:} \varepsilon(\f u){+} a( \tchi)  \mathbb{C}\varepsilon(\tu)  {:}  \varepsilon(\tu)  \dx  \Big |_0^t
\\
& \quad 
 - \int_0^t \int_\Omega \Big[ \frac{1}{2}\chi_t a'(\tchi) \Cm{\tu}{\tu} {+} \frac{1}{2}\tchi_t a'(\chi) \Cm{\f u}{\f u} \Big] \dx \ds 
   \\& \quad - \ito \Big[ a(\chi) \Cm{\f u}{\tu_t} {+} a(\tchi) \Cm{\tu}{\f u_t } \Big] \dx\ds
   \\ & \quad +\ito   \Big[(a(\chi) {-} a(\tchi) ) \mathbb{C}\varepsilon(\tu) {:} \varepsilon(\tut) {-} a(\chi) \mathbb{C}\varepsilon(\tu ) {:} \varepsilon(\f u_t){+}  a(\tchi)  \Cm{\tu}{\f u_t} \Big] \dd x \dd s 
   \\ & \quad +\ito  \frac{1}{2}\left[ 
(a'(\chi) \chi_t {-} a'(\tchi) \tchi_t) \Cm{\tu}{\tu} 
{-} 2 a'(\chi )\chi_t \Cm{\f u }{\tu}\right] \dd x \dd s 
   \\ & \quad +\ito \frac{1}{2} \left[a'(\chi) \tchi_t \Cm{\f u }{\f u} +a'(\tchi)\chi_t\Cm{\tu}{\tu}\right] \dd x \dd s \,.
   \end{aligned}
\] \EEE

Concerning the nonlinear potential $W$, we find
\begin{align*}
    \int_\Omega &W( \chi ){-} W (\tchi){-} W'(\tchi)(\chi {-}\tchi)  \dx\Big|_0^t \\
   & = \int_\Omega \big[ W(\chi) {+} W(\tchi) \big] \dx \Big|_0^t - \int_\Omega \big[2 W(\tchi) {+} W'(\tchi) (\chi{-}\tchi) \big] \dx \Big|_0^t 
    \\
    &= \int_\Omega W(\chi) {+} W(\tchi) \dx \Big|_0^t - \int_0^t \int_\Omega   \big[W'(\tchi) \tchi_t  {+} W'(\tchi) \chi_t {+} W''(\tchi) \tchi_t ( \chi{-}\tchi) \big] \dx \ds 
    \\
    & = \int_\Omega W(\chi) {+} W(\tchi) \dx \Big|_0^t - \int_0^t \int_\Omega \big[ W'(\chi) \tchi_t  {+} W'(\tchi) \chi_t \big]  \dx \ds \\
   &\quad  + \int_0^t \int_\Omega\big[ \tchi_t \left(W'(\chi) {-} W'(\tchi) {-}W''(\tchi)  ( \chi{-}\tchi)  \right) \big]  \dx \ds \,.
\end{align*}
For the remaining quadratic terms in the relative energy, we find 
\begin{align*}
   \int_\Omega\big[ \frac{1}{2}&|\nabla \chi{-}\nabla \tchi |^2 {+} \frac{1}{2}| \f u_t {-} \tu_t |^2 {+} \frac{\ell}{2}| \chi {-}\tchi |^2 \big] \dx \Big|_0^t 
   \\={}& \int_\Omega \big[ \frac{1}{2}|\nabla \chi|^2 {+} \frac{1}{2}| \f u_t  |^2   {+} \frac{1}{2}|\nabla \tchi |^2 {+} \frac{1}{2}|  \tu_t |^2  \big]\dx  \Big|_0^t 
   \\
   & + \int_0^t  \int_\Omega  \big[\chi _t \Delta \tchi {-} \nabla \chi {\cdot} \nabla   \tchi_t \big] \dx \ds   -
   \int_0^t\big[ \langle{} \f u_{tt}, \tu_t\rangle {+} \int_\Omega\tu_{tt}{\cdot}\f u_t\dx  \big] \ds  
    \\
   & +\int_0^t \int_\Omega \ell ( \chi_t{-}\tchi_t) (\chi{-}\tchi) \dx \ds \,.
\end{align*} 
Moreover, we   relate  the relative dissipation  \eqref{defW} to the pseudo-potential $\mathcal{D}$  \eqref{def:diss}
(where  $\gamma_1$ is set to $0$  in view of the boundary condition \eqref{homogNeu})
via 
\begin{align*}
    \mathcal{W}(\chi,\f u_t,\chi_t| \tu_t,  \tchi_t) ={}& \mathcal{D}(\chi,\f u_t,\chi_t) + \mathcal{D}(\tchi,\tu_t, \tchi_t) - \int_\Omega
    \big[ 2 \chi_t \tchi_t {+} b(\chi) \mathbb{V} \eps{\f u_t}{:} \eps {  \tu_t} \big]\dx\\& {-}\int_\Omega \big[ b(\tchi)  \mathbb{V} \eps{\tu_t} {:} \eps {  \f u_t}{+} (b(\tchi){-}b(\chi)) \mathbb{V}
     \eps{\tu_t{-}\f u _t }{:}\eps{\tu_t} \big] \dx \,.
\end{align*}

Combining all  the above calculations, we obtain
 (note that,  we have $\gamma_2=0$ in $\calE$ due to \eqref{homogNeu}), 
\begin{align*}
\mathcal{R}( \f u , \chi , \f u _t& | \tu, \tchi,\tut) \Big|_0^t + {}\int_0^t\mathcal{W}( \chi, \f u _t, \chi_t | \tut,\tchi_t) \ds 
\\
={}& \mathcal{E}(\f u , \chi, \f u_t)\Big |_0^t +\int_0^t \mathcal{D}(\chi , \f u _t , \chi_t) \ds   + \mathcal{E}(\tu , \tchi, \tut) \Big|_0^t+ \int_0^t \mathcal{D}(\tchi , \tut , \tchi_t) \ds 
\\
&- \int_0^t \int_\Omega  \chi_t  \colorboxed{green}{\tchi_t 
{-} \Delta \tchi{+} \frac{1}{2}a'(\tchi) \Cm{\tu}{\tu}{+} W'(\tchi))} \dx\ds\\& 
-\int_0^t \colorboxed{blue}{\int_\Omega \chi_t \tchi_t {+}\nabla  \tchi_t{\cdot} \nabla   \chi {+} \frac{1}2a'(\chi) \Cm{\f u}{\f u}\tchi_t{+} W'(\chi)\tchi_t \dx}\ds 
\\
&- \int_0^t  \colorboxed{magenta}{ \langle \f u _{tt} , \tut \rangle {+} \int_\Omega\big[ b(\chi) \mathbb{V}\varepsilon(\f u_t) {:} \varepsilon(\tut) {+} a(\chi)  \Cm{\f u}{\tut} \big] \dx} \ds 
\\
&- \int_0^t  \colorboxed{red}{\int_\Omega \tu_{tt} {\cdot} \f u_t {+} b(\tchi) \mathbb{V}\varepsilon(\tut) {:} \varepsilon(\f u_t) {+} a( \tchi) \Cm{\tu}{\f u_t} \dx} \ds 
\\
&+ \int_0^t \int_\Omega (b(\chi){-} b(\tchi)) \mathbb{V}\varepsilon(\tut){:} \varepsilon(\f u _t {-} \tut) \dx \ds 
\\
&+ \int_0^t \int_\Omega \Big[ \tchi_t\left ( W'(\chi) {-} W'(\tchi){-}W''(\tchi)(\chi{-}\tchi) \right ) {+}  \ell (\chi_t{-}\tchi_t)(\chi{-}\tchi)\Big] \dx \ds 
\\
& + \int_0^t \int_\Omega \Big[ (a(\chi) {-} a(\tchi) ) \mathbb{C}\varepsilon(\tu) {:} \varepsilon(\tut) {-} a(\chi) \mathbb{C}\varepsilon(\tu ) {:} \varepsilon(\f u_t){+}  a(\tchi)  \Cm{\tu}{\f u_t}\Big]
\dx \ds 
\\&+\frac{1}{2}\int_0^t\int_\Omega 
\Big[(a'(\chi) \chi_t {-} a'(\tchi) \tchi_t) \Cm{\tu}{\tu} 
{-} 2 a'(\chi )\chi_t \Cm{\f u }{\tu} \Big]\dx \ds
\\
&{+}\frac{1}{2} \int_0^t \int_\Omega\Big[ a'(\chi) \tchi_t \Cm{\f u }{\f u} {+}a'(\tchi)\chi_t\Cm{\tu}{\tu}\Big]\dx \ds\,.
\end{align*}
 On the one hand, since
 $(\f u , \chi) $ is a weak solution in the sense of Definition~\ref{def:weakSol}, 
 for the damage flow rule we have the one-sided variational inequality \eqref{weakChiIneq}: thus, since $
\tilde{\chi}_t \leq 0$ a.e.\ in $\Omega{\times}(0,T)$, we find that 
  the term in the {\color{blue} \textbf{blue}} box  is negative. In turn, the terms in the 
 {\color{magenta} \textbf{magenta}} box
 equals 
 $\langle \f f , \tu_t \rangle_{H^1(\Omega)} $.  On the other hand, 
 since   $(\tu,\tchi)$ is a strong solution in the sense of Definition~\ref{def:strongSol},  the term in the 
 {\color{green} \textbf{green}} box is null a.e.\ in $\Omega{\times}(0,T)$, whereas the term in the   {\color{red} \textbf{red}} box equals  $\int_\Omega \f f {\cdot}  \uu\dd x $. 
Hence, we  
  find 
\begin{align*}
  \mathcal{R}( \f u , \chi , \f u _t& | \tu, \tchi,\tut) \Big|_0^t + {}\int_0^t\mathcal{W}( \chi, \f u _t, \chi_t | \tut,\tchi_t) \ds \\
  &\leq
 \colorboxed{dblue}{\mathcal{E}(\f u , \chi, \f u_t)\Big |_0^t +\int_0^t \mathcal{D}(\chi , \f u _t , \chi_t) \ds
   - \int_0^t  \langle \f f ,  \tu_t \rangle_{H^1(\Omega)} \dd s}
   \\
   & \quad 
 + \colorboxed{bernstein}{\mathcal{E}(\tu , \tchi, \tut) \Big|_0^t+ \int_0^t \mathcal{D}(\tchi , \tut , \tchi_t) \ds   - \int_0^t \int_\Omega \f f {\cdot} \uu \dd x  \dd s }  
  \\
  &  \quad 
+ \int_0^t \int_\Omega (b(\chi){-} b(\tchi)) \mathbb{V}\varepsilon(\tut){:} \varepsilon(\f u _t {-} \tut) \dx \ds 
\\
& \quad + \int_0^t \int_\Omega \big[ \tchi_t\left ( W'(\chi) {-} W'(\tchi){-}W''(\tchi)(\chi{-}\tchi) \right ) {+}  \ell  (\chi_t{-}\tchi_t)(\chi{-}\tchi)  \big] \dx \ds 
\\
& \quad  + \int_0^t \int_\Omega  \big[ 
(a(\chi) {-} a(\tchi) ) \mathbb{C}\varepsilon(\tu) {:} \varepsilon(\tut) {-} a(\chi) \mathbb{C}\varepsilon(\tu ) {:} \varepsilon(\f u_t){+}  a(\tchi)  \Cm{\tu}{\f u_t}
\big]
\dx \ds 
\\& \quad +\frac{1}{2}\int_0^t\int_\Omega 
\big[(a'(\chi) \chi_t {-} a'(\tchi) \tchi_t) \Cm{\tu}{\tu} 
{-} 2 a'(\chi )\chi_t \Cm{\f u }{\tu} \big] \dx \ds
\\
& \quad  +\frac{1}{2} \int_0^t \int_\Omega \big[ a'(\chi) \tchi_t \Cm{\f u }{\f u} {+}a'(\tchi)\chi_t\Cm{\tu}{\tu} \big] \dx \ds 
\\
&
 \doteq I_1+I_2+I_3+I_4+I_5+I_6+I_7\,.
\end{align*}
  Again, we use the fact that $(\uu,\chi)$ is a weak solution, and thus satisfies the energy-dissipation inequality 
\eqref{UEDI} (with $g\equiv 0$, as we are confining the discussion to homogeneous Neumann boundary conditions), to 
conclude that the term in the  {\color{dblue} \textbf{dark blue}} box  is negative. Analogously, since the strong solution $(\tu,\tchi)$ satisfies the energy-dissipation balance,
we have that the term in the  {\color{bernstein} \textbf{orange}} box  is null. 

We now calculate the integrands of $I_5+I_6+I_7$.  Indeed,  we have that  \EEE 
\[
\begin{aligned}
 & (a(\chi) {-} a(\tchi) ) \mathbb{C}\varepsilon(\tu) {:} \varepsilon(\tut) {-}  
 a(\chi)\mathbb{C}\varepsilon(\tu ) {:} \varepsilon(\f u_t)+ a(\tchi) \Cm{\tu}{\f u_t}
\\ & = (a(\chi){-} a(\tchi))  \mathbb{C}\varepsilon(\tu) {:} \varepsilon(\tut{-}\f u _t ) 
\end{aligned}
\]
as well as
\begin{align*}
&(a'(\chi )\chi_t {-} a'(\tchi) \tchi_t) \Cm{\tu}{\tu} 
- 2 a'(\chi ) \chi_t \Cm{\f u }{\tu}
+ a'(\chi) \tchi_t \Cm{\f u }{\f u} +a'(\tchi)\chi_t\Cm{\tu}{\tu}
\\
& =
a'(\chi) \tchi_t \Cm{\f u {-}\tu}{\f u{-}\tu} +  a'(\chi) \tchi_t \Cm{\f u }{\tu} + a'(\chi) \tchi_t \Cm{\f u {-} \tu}{\tu} \\&\quad +  a'(\chi) \chi_t \Cm{\tu{-}\f u}{\tu} +  a'(\tchi)(\chi_t{-}\tchi_t)\Cm{\tu}{\tu} {-}  a'(\chi)  \chi_t  \Cm{\f u }{\tu} 
\\
& =
a'(\chi) \tchi_t \Cm{\f u {-}\tu}{\f u{-}\tu} +  a'(\chi) (\tchi_t {-} \chi_t) \Cm{\f u }{\tu} \\&\quad+ a'(\chi) (\tchi_t{-}\chi_t) \Cm{\f u {-} \tu}{\tu}   +  a'(\tchi)(\chi_t{-}\tchi_t)\Cm{\tu}{\tu}
\\&=
 a'(\chi )\tchi_t \Cm{\f u {-}\tu}{\f u{-}\tu}  + 2 a'(\chi ) ( \chi_t{-}\tchi_t) \Cm{\tu{-}\f u}{\tu}\\
  &\quad + ( a'(\tchi){-} a'(\chi))(\chi_t{-}\tchi_t) \Cm{\tu}{\tu}\,.
\end{align*}
 Since $a$ is non-decreasing and $\tchi_t \leq 0$ a.e.\ in $\Omega{\times}(0,T)$,  
the first term on the right-hand side has a negative sign. 
Inserting everything back into the relative energy inequality, we find
\begin{align*}
\mathcal{R}( \f u , \chi , \f u _t& | \tu, \tchi,\tut) \Big|_0^t + {}\int_0^t\mathcal{W}( \chi, \f u _t, \chi_t | \tut,\tchi_t)- \frac 12\int_\Omega a'(\chi) \tchi_t \Cm{\f u {-}\tu}{\f u {-}\tu} \dx   \ds 
\\
\leq {}&
\int_0^t \int_\Omega (b(\chi){-} b(\tchi)) \mathbb{V}\varepsilon(\tut){:} \varepsilon(\f u _t {-} \tut) \dx \ds 
\\
&+ \int_0^t \int_\Omega \big[ \tchi_t\left ( W'(\chi) {-} W'(\tchi){-}W''(\tchi)(\chi{-}\tchi) \right ) +  \ell  (\chi_t{-}\tchi_t)(\chi{-}\tchi)\big] \dx \ds 
\\
& + \int_0^t \int_\Omega 
(a(\chi){-} a(\tchi))  \mathbb{C}\varepsilon(\tu) {:} \varepsilon(\tut-\f u _t )  \dx \ds 
\\
& + \frac{1}{2}\int_0^t \int_\Omega \big[ 2 a'(\chi ) ( \chi_t{-}\tchi_t) \Cm{\tu-\f u}{\tu} 
 { +} ( a'(\tchi){-} a'(\chi))(\chi_t{-}\tchi_t) \Cm{\tu}{\tu} \big]  \dx \ds \,.
\end{align*}
The right-hand side will be estimated by the relative energy $\mathcal{R}$.  Indeed,  it holds
\renewcommand{\eps}[1]{\varepsilon ({#1})}
\begin{equation}
\label{almost-final-REI}
\begin{aligned}
\mathcal{R}( \f u , \chi , \f u _t& | \tu, \tchi,\tut) \Big|_0^t + {}\int_0^t\mathcal{W}( \chi, \f u _t, \chi_t | \tut,\tchi_t)- \frac12\int_\Omega a'(\chi) \tchi_t \Cm{\f u {-}\tu}{\f u {-}\tu} \dx   \ds 
\\
& \leq\int_0^t  \| b(\chi) {-} b(\tchi) \|_{L^6(\Omega)} \| \mathbb{V}\varepsilon(\tut)\|_{L^3(\Omega)} \| \varepsilon (\f u_t{ -} \tut ) \|_{L^2(\Omega)}\ds
 \\&
 \quad 
  +\int_0^t  \left\| \tchi_t \int_0^1 W''(\tchi{+}\rho(\chi{-}\tchi))\,
  \mathrm{d}\rho\right \|_{L^{3/2}(\Omega)}
   \| \chi {-} \tchi\|_{L^6(\Omega)}^2 \ds \EEE
   \\
   & \quad   
  +  \ell  \int_0^t   \|\chi_t-\tchi_t\|_{L^{2}(\Omega)}  \| \chi {-} \tchi\|_{L^2(\Omega)} \ds \EEE
\\
& \quad + \int_0^t
  \| \mathbb{C}\varepsilon(\tu) \| _{L^3(\Omega)} \| a(\chi) {-} a(\tchi) \| _{L^6(\Omega)} \| \varepsilon (\f u_t {-}\tut) \|_{L^2(\Omega)} \ds
\\
&\quad +  
 \int_0^t 
  \left\|a'(\chi) \Cm{\tu}{\f u{-}\tu} \right\|_{L^2(\Omega)} \| \chi_t {-} \tchi_t\|_{L^2(\Omega)} \dd s \EEE
  \\
&\quad +  
 \int_0^t 
  \| a'(\tchi){-} a'(\chi) \|_{L^3(\Omega)}  \| \chi_t {-} \tchi_t\|_{L^2(\Omega)}^2  \| \Cm{\tu}{\tu} \|_{L^6(\Omega)}^2   \ds
  \\ & 
  \doteq 
  I_{8}+  I_{9} +I_{10}+ I_{11} +I_{12} +I_{13} \,. \EEE
   \EEE
\end{aligned}
\end{equation}
 Now,  since $b \in \mathrm{C}^1(\R)$
and $\chi,\, \tchi \in L^\infty(\Omega)$, we can estimate
\[
\begin{aligned}
I_8  & \leq c \int_0^t  \| \chi {-} \tchi \|_{L^6(\Omega)}   \| \varepsilon(\tut)\|_{L^3(\Omega)} \| \varepsilon (\f u_t{ -} \tut ) \|_{L^2(\Omega)}\ds
\\
& 
\leq 
\frac14 \int_0^t  \int_\Omega b(\chi) \mathbb{V}\varepsilon(\f u_t{-} \tu_t) {:} \varepsilon(\f u_t{-}\tu_t )\dx \dd s 
+ c  \int_0^t    \| \varepsilon(\tut)\|_{L^3(\Omega)}^2  \| \chi {-} \tchi \|_{L^6(\Omega)}^2  \ds \,,
\end{aligned}
\]
where we have used the lower bound 
$b$, implying
\[ \| \eps{\f u_t{-} \tu_t }\|_{L^2(\Omega)}^2 \leq c \int_\Omega b(\chi) \mathbb{V}\varepsilon(\f u_t{-} \tu_t) {:} \varepsilon(\f u_t{-}\tu_t )\dx \,. \]
Similarly, relying on the fact that $W\in \mathrm{C}^2(\R)$, we check that 
\[
I_9 \leq c \int_0^t  \left\| \tchi_t \right \|_{L^{3/2}(\Omega)}
   \| \chi {-} \tchi\|_{L^6(\Omega)}^2 \ds\,,
\]
while we obviously have 
\[
I_{10} \leq \frac14 \int_0^t  \|\chi_t-\tchi_t\|_{L^{2}(\Omega)}^2 \dd s +c \ell^2 \int_0^t  \| \chi {-} \tchi\|_{L^2(\Omega)}^2 \ds \,.
\] 
Relying now on the fact that  $a \in \mathrm{C}^1(\R)$,
we may estimate 
\[
\begin{aligned}
I_{11}  & \leq   \int_0^t
  \| \varepsilon(\tu) \| _{L^3(\Omega)} \| \chi {-} \tchi \| _{L^6(\Omega)} \| \varepsilon (\f u_t {-}\tut) \|_{L^2(\Omega)} \ds
\\ &   \leq  \frac14 \int_0^t  \int_\Omega b(\chi) \mathbb{V}\varepsilon(\f u_t{-} \tu_t) {:} \varepsilon(\f u_t{-}\tu_t )\dx \dd s 
  + c  \int_0^t
  \| \varepsilon(\tu) \| _{L^3(\Omega)}^2 \| \chi {-} \tchi \| _{L^6(\Omega)}^2 \dd s 
  \end{aligned}
\]
\EEE
The  assumptions on $a'$ imply the estimate
\[
\begin{aligned}
&    \int_\Omega \left|a'(\chi) \Cm{\tu}{\f u{-}\tu} \right|^2 \dx
\\
  &\leq  \| a'(\chi)\|_{L^\infty(\Omega)}^2 | \mathbb{C}|^2 \| \eps{\tu}\|_{L^\infty(\Omega)}^2 \int_\Omega 
      |\eps{\f u {-}\tu}|^2 \dx \\&
    \leq c  \| a'(\chi)\|_{L^\infty(\Omega)}^2  \| \eps{\tu}\|_{L^\infty(\Omega)}^2  \| \eps{ \f u (0)} {-} \tu(0)\|_{L^2(\Omega)}^2 
     + c  \| a'(\chi)\|_{L^\infty(\Omega)}^2  \| \eps{\tu}\|_{L^\infty(\Omega)}^2  \int_0^s \| \eps{\f u_t{-} \tu_t }\|_{L^2(\Omega)}^2\dta
     \\
     &
     \doteq M(\chi,\tu,\uu,\tu_t)
      \,,
\end{aligned}
\]
 and therefore we have 
\[
I_{12} \leq  \frac14 \int_0^t  \|\chi_t-\tchi_t\|_{L^{2}(\Omega)}^2 \dd s + M(\chi,\tu,\uu,\tu_t)\,.
\]
Finally, we estimate 
\[
I_{13}  \leq  \frac14 \int_0^t  \|\chi_t-\tchi_t\|_{L^{2}(\Omega)}^2 \dd s 
 +c \int_0^t \| \eps{\tu}\|_{L^6(\Omega)}^4\|\tchi{-}\chi\|_{L^3(\Omega)}^2 \dd s\,.
\]
Inserting all the above estimates in \eqref{almost-final-REI}, we ultimately deduce  \EEE
%
%
\begin{align*}
\mathcal{R}( \f u , \chi , \f u _t& | \tu, \tchi,\tut) \Big|_0^t + {}\int_0^t\mathcal{W}( \chi, \f u _t, \chi_t | \tut,\tchi_t)- \int_\Omega a'(\chi) \tchi_t \Cm{\f u {-}\tu}{\f u {-}\tu} \dx   \ds 
\\
\leq{}&\int_0^t  \left[ 
 \frac{\underline{b}}{2} \int_\Omega
 \mathbb V {:}\eps{\f u_t {-} \tu_t }){:}\eps{\f u_t {-} \tu_t} \dd x \dd s 
 {+} \frac{1}{2}\| \chi_t {-} \tchi_t\|_{L^2(\Omega)}^2
 \right]
 \ds \\&
 +c
\int_0^t   \left( \| \tchi_t  \|_{L^{3/2}(\Omega)}{+}  
   \| \varepsilon(\tut)\|_{L^3(\Omega)}^2
{+}\ell^2{+}  \| \varepsilon(\tu) \| _{L^3(\Omega)}^2 {+} \| \varepsilon(\tu) \| _{L^6(\Omega)}^4
\right)  \| \chi {-} \tchi\|_{L^6(\Omega)}^2 
\ds 
\\&+ c  \| a'(\chi)\|_{L^\infty(\Omega)}^2  \| \eps{\tu}\|_{L^\infty(\Omega)}^2  \| \eps{  \f u (0)} {-} \tu(0)\|_{L^2(\Omega)}^2 
    \\
    & \quad + c  \| a'(\chi)\|_{L^\infty(\Omega)}^2  \| \eps{\tu}\|_{L^\infty(\Omega)}^2   \int_0^t\int_0^s \int_\Omega b(\chi) \mathbb{V}\varepsilon(\f u_t{-}
     \tu_t) {:} \varepsilon(\f u_t{-}\tu_t )\dx\dta   \dd s
\,.
\end{align*}
Then, estimate \eqref{REI} follows  by Gronwall's inequality. 
\end{proof}
\medskip

\noindent \textbf{Conclusion of the proof of Theorem \ref{thm:3}:} Let $(\uu,\chi)$ and $(\tu,\tchi)$ be a weak and a strong solution pair, respectively, emanating from the same initial data. Then, the right-hand side of estimate \eqref{REI}  is null.  We thus conclude that  $\mathcal{R}( \f u(t) , \chi(t) , \f u _t(t) | \tu(t), \tchi(t),\tut(t)) \equiv 0 $ for almost all $t\in (0,T)$, which obviously yields   $\uu \equiv \tu$ and $\chi \equiv \tchi$. 
\QED  
\bigskip

\noindent
\noindent
{\bf Acknowledgments.}~
This research has been performed in the framework of the MIUR-PRIN Grant 2020F3NCPX 
``Mathematics for industry 4.0 (Math4I4)''. The present paper also benefits from the support of 
the GNAMPA (Gruppo Nazionale per l'Analisi Matematica, la Probabilit\`a e le loro Applicazioni)
of INdAM (Istituto Nazionale di Alta Matematica). E. Rocca also acknowledges the support of Next Generation EU Project No.P2022Z7ZAJ (A unitary mathematical framework for modelling muscular dystrophies).
R. Lasarzik acknowledges support by the Deutsche Forschungsgemeinschaft (DFG, German Research Foundation) under Germany's Excellence Strategy – The Berlin Mathematics Research Center MATH+ (EXC-2046/1, project ID: 390685689).

\bigskip

\appendix

\section{Elliptic regularity results}
\label{s:appA}
 The main result of this section, Corollary \ref{corB:ellipt} below,  collects the two key elliptic regularity estimates for the momentum balance, which are at the core of our analysis of strong solutions. Corollary \ref{corB:ellipt}  follows from the following \EEE
\begin{proposition}
\label{propB:ellipt}
Let $\Vm \in \R^{d\times d\times d\times d}$ fulfill the assumptions~\eqref{asstensors} of Hypothesis~\textbf{\ref{h:1}} and let the domain $\Omega\subset \R^d$ fulfill~\eqref{omega-smooth}. Then,  there exists a constant $C_{\mathrm{ER}} >0$ such that for any 
$\f h \in H^1(\Omega;\R^d)$  
the solution $\f y$ to the boundary-value problem
\begin{align}\label{boundaryvalueproblemProp}
\begin{split}
-\di (\Vm \eps{ \f y})  &= \f h \quad\text{in } \Omega \, , \quad
\int_\Omega \f y \dx =0 \,,
\quad
  \f n {\cdot} \Vm \eps{ \f y}  = 0\quad\text{on } \partial\Omega \,.
\end{split}
\end{align}
fulfills $ \f y \in H^3 (\Omega)$, and there holds
\begin{subequations}
\label{key-elliptic-regul-est}
\begin{align}
\label{key-elliptic-regul-est-1}
    \| \f y \| _{H^2(\Omega)} \leq  C_{\mathrm{ER}} \left( \| \f h \| _{L^2(\Omega)}{+} \| \f y \| _{H^1(\Omega)} \right) \,
    \intertext{as well as }
    \label{key-elliptic-regul-est-2}
    \| \f y \| _{H^3(\Omega)} \leq   C_{\mathrm{ER}}  \left( \| \f h \| _{H^1(\Omega)}{+} \| \f y \| _{H^1(\Omega)} \right) \,.
\end{align}
\end{subequations}
\end{proposition}
In fact, the result of this proposition is a consequence of~\cite[Thm.~3.45] {regularityElast},
compare also to~\cite[Sec.~4.3b]{regularityElast}. 
\par
 We will also resort to the following abstract version of Poincar\'e's inequality, see 
 \cite{Gilly}.
\begin{lemma}
\label{l:Poincare}
Let $\mathsf{V}, \mathsf{H}, \mathsf{W},  \mathsf{Z} $ be four Hilbert spaces
with $\mathsf{V} \Subset \mathsf{H}$ compactly. 
Let $A:  \mathsf{V} \to   \mathsf{W}$ and $B:\mathsf{V} \to \mathsf{Z}$
be linear and continuous operators such that
\begin{compactitem}
\item $\mathrm{Ker}(A) {\cap} \mathrm{Ker}(B) = \{0\}$;
\item there exists a positive constant $C>0$ such that for all $v\in \mathsf{V}$ we have 
\begin{equation}
\label{Poin-bases}
\| v\|_{\mathsf{V}} \leq C \left( \|v\|_{\mathsf{H}} {+} \| Av\|_{\mathsf{W}} \right)\,.
\end{equation}
\end{compactitem}
Then,
\[
\exists\,  M>0 \quad \forall\, v  \in \mathsf{V} \, :  \qquad  \|v\|_{\mathsf{H}} \leq M \left(\|Bv\|_{\mathsf{Z}} {+} \| Av\|_{\mathsf{W}}  \right)\,,
\]
so that $\| v\|_{\mathsf{V}}$ is equivalent to $\|B v\|_{\mathsf{Z}} {+} \| A v\|_{\mathsf{W}} $.
\end{lemma}
\par
We are now in a position to derive the following 
\begin{corollary}
\label{corB:ellipt}
Under the assumptions of Proposition \ref{propB:ellipt},  there exists a constant $C_{\mathrm{ER}} >0$ such that for any 
$\f h \in H^1(\Omega;\R^d)$  
the solution $\f y$ to the boundary-value problem \eqref{boundaryvalueproblemProp} satisfies 
\begin{subequations}
\label{key-elliptic-regul-est-NEW}
\begin{align}
\label{key-elliptic-regul-est-1-n}
  &  \| \f y \| _{H^2(\Omega)} \leq  \widehat{C}_{\mathrm{ER}} \left( \| \f h \| _{L^2(\Omega)}{+} \| \f y \| _{L^2(\Omega)} \right) \,,
  \\ &
    \label{key-elliptic-regul-est-2-n}
    \| \f y \| _{H^3(\Omega)} \leq   \widehat{C}_{\mathrm{ER}}  \left( \| \eps{ \f h} \| _{L^2(\Omega)}{+} \| \f y \| _{H^1(\Omega)} \right) \,.
\end{align}
\end{subequations}
\end{corollary}
\begin{proof}
$\vartriangleright$ \eqref{key-elliptic-regul-est-1-n}:  We apply Lemma \ref{l:Poincare} with the following choices: 
$\mathsf{V} = H^2(\Omega;\R^d)$, $ \mathsf{H}= H^1(\Omega;\R^d)$, 
$ \mathsf{W} = L^2(\Omega;\R^d) = \mathsf{Z} $, and $A\f y = -\di (\Vm \eps{ \f y}) $, $B \f y = y $.  Observe that \eqref{Poin-bases} holds thanks to \eqref{key-elliptic-regul-est-1}.
Then,  $\| \f y \|_{H^2}$ is equivalent to $\|\f y \|_{L^2} {+} \|  \di (\Vm \eps{ \f y})\|_{L^2} $.
\\
$\vartriangleright$ \eqref{key-elliptic-regul-est-2-n}: We now apply  Lemma \ref{l:Poincare} with the very same choices  for $ \mathsf{H} $, $  \mathsf{Z} $, and $B$,
as in the previous lines, while  we set $\mathsf{V} = H^3(\Omega;\R^d)$, $\mathsf{W}= L^2(\Omega;\R^{d\times d})$, and    $A\f y = \eps{\di (\Vm \eps{ \f y}) }$. In this case, 
\eqref{Poin-bases}  reads 
\[
\| \f y \|_{H^3} \leq C \left( \|\f y \|_{H^1} {+} \|  \eps{\di (\Vm \eps{ \f y}) }\|_{L^2} \right)\,,
\]
which holds true thanks to \eqref{key-elliptic-regul-est-2}, taking into account that, again by a  Korn-type inequality, $ \|  \eps{\di (\Vm \eps{ \f y}) }\|_{L^2} $ controls 
$\|\di (\Vm \eps{ \f y})\|_{H^1}$. Then, \eqref{key-elliptic-regul-est-2-n} ensues. 
\end{proof}
\EEE

\section{Smoothening the Yosida approximation}
\label{s:appB}
Following, e.g., the lines of \cite[Sec.\,3]{Gilardi-Rocca}, for a given convex
function $\widehat{\beta}: \R \to \R$ with subdifferential
$\betaup= \partial \widehat\beta: \R \rightrightarrows \R$, and for a fixed
$\delta\in (0,1)$, we define
\begin{equation}
\label{beta-regul}
\betaup_\delta : = \betaup_\delta^{\mathrm{Y}} \star  \varrho_\delta 
\end{equation}
where $ \betaup_\delta^{\mathrm{Y}} $ is the Yosida regularization of the maximal
monotone operator $\betaup$ (we refer to, e.g., \cite{Brez73}) and
\begin{equation}
\label{convol-kernel}
\varrho_\delta(x): = \tfrac1{\delta^{2}} \varrho \left( \tfrac x{\delta^2}\right)
\qquad \text{with }  
\left\{
  \begin{array}{ll}
    \varrho \in \rmC^\infty(\R), 
    \\
    \|\varrho\|_{L^1(\R)} =1,
    \\
    \mathrm{supp}(\varrho) \subset [{-}1,1].
  \end{array}
\right.
\end{equation}
Thus, $\betaup_\delta \in \rmC^\infty(\R)$ and it has been shown in
\cite{Gilardi-Rocca} that
\begin{subequations}
\label{properties-delta-approx}
\begin{equation}
\label{prop-delta-1}
\|\betaup_\delta'\|_{L^\infty(\R)} \leq \frac1\delta,  \qquad  |\betaup_\delta(x){-} \betaup_\delta^{\mathrm{Y}}(x)| \leq \delta 
\text{ for all } x \in \R\,.
\end{equation}
Taking into account the properties of the Yosida approximation we also deduce that 
\begin{equation}
\label{prop-delta-1-bis}
|\betaup_\delta(x)| \leq  |\betaup^o(x)| +\delta  \qquad  \text{ with }  |\betaup^o(x)| = \inf\{|y|\, : \, y \in \betaup(x) \}\,.
\end{equation}
Furthermore, $\betaup_\delta $ admits a convex potential $\widehat\beta_\delta$
satisfying, as a consequence of \eqref{prop-delta-1}, (below
$\widehat\betaup_\delta^{\mathrm{Y}}$ denotes the Yosida approximation of $\widehat \beta$):
\begin{equation}
\label{prop-delta-2}
-\delta |x|  \leq  \widehat{\beta}_\delta^{\mathrm{Y}} (x) -\delta |x| \leq
\widehat{\beta}_\delta(x) \leq  \widehat\beta_\delta^{\mathrm{Y}} (x) +\delta |x| \leq
\widehat\beta(x)+\delta|x|  \  \text{ and } \  \widehat\beta_\delta(x) \to
\widehat\beta(x)  \text{ for all } x \in \R\,. 
\end{equation} 
We also point out that  the following analogue of Minty's trick holds: given $I \subset \R$
and sequence $(v_\delta)_\delta\, v,\, \beta \in L^2 (I;\R)$ such that
$v_\delta\weakto v$ and $\betaup_\delta(v_\delta) \weakto \beta $ in
$ L^2 (I)$,
\begin{equation}
  \label{prop-delta-3}
  \limsup_{\delta\to 0^+} \int_I \betaup_\delta(v_\delta)  v_\delta \dd x
  \leq \int_I \beta    v \dd x  \quad \Longrightarrow \quad \beta \in
  \partial \widehat{\beta}(v) \text{ a.e.\ in } I. 
\end{equation}
\end{subequations}
 We conclude this section with a new result, ensuring an additional estimate for $\betaup_\delta{''}$.
 \begin{lemma}
 \label{l:Robert}
The function 
 $\betaup_\delta$ from \eqref{beta-regul} fulfills 
 \begin{equation}
 \label{estimate-beta2}
 |\betaup_\delta''(x)| \leq \frac{\widehat{C}_\rho}{\delta^3} \qquad \text{for all } x \in \R
 \end{equation}
 with $\widehat{C}_\rho =  \|\varrho'\|_{L^1(\R)} $. 
 \end{lemma} 
\begin{proof}
We have 
\[
\betaup_\delta'(x) = \int_{-\delta^2}^{\delta^2}\varrho_\delta (y) (\betaup_\delta^{\mathrm{Y}})' (x{-}y) \dd y  = -\int_{x-\delta^2}^{x+\delta^2}\varrho_\delta (x{-}y)
 (\betaup_\delta^{\mathrm{Y}})'(y)  \dd y \,.
\]
Therefore, by the first of \eqref{prop-delta-1} we have 
\[
\begin{aligned}
\betaup_\delta''(x) =-\int_{x-\delta^2}^{ x+\delta^2}\varrho_\delta' (x{-}y)
 (\betaup_\delta^{\mathrm{Y}})'(y)  \dd y 
& \leq \frac1\delta \left|  \int_{\R}\varrho_\delta' (x{-}y) \dd y \right| 
\\
& =  \frac1\delta \left|  \int_{\R} \frac1{\delta^4}\varrho' \left(\frac  y{\delta^2}\right) \dd y \right|
 = \frac1{\delta^3}  \left|  \int_{\R} \varrho' ( z) \ \dd z \right|\leq \frac {\widehat{C}_\rho}{\delta^3} \,.
 \end{aligned}
\]
\end{proof}

\EEE

{\small

\markboth{References}{References}

\bibliographystyle{abbrv}


\begin{thebibliography}{10}

\bibitem{ALR}
A.~Agosti, R.~Lasarzik, and E.~Rocca.
\newblock Energy-variational solutions for viscoelastic fluid models.
\newblock {\em Preprint arXiv:2310.13601}, pages 1--40, 2023.

\bibitem{BoBo}
E.~Bonetti and G.~Bonfanti.
\newblock Well-posedness results for a model of damage in thermoviscoelastic
  materials.
\newblock {\em Ann. Inst. Henri Poincar{\'e}, Anal. Non Lin{\'e}aire},
  25(6):1187--1208, 2008.

\bibitem{bosch}
E.~Bonetti and G.~Schimperna.
\newblock Local existence for {Fr{\'e}mond}'s model of damage in elastic
  materials.
\newblock {\em Contin. Mech. Thermodyn.}, 16(4):319--335, 2004.

\bibitem{bss}
E.~Bonetti, G.~Schimperna, and A.~Segatti.
\newblock On a doubly nonlinear model for the evolution of damaging in
  viscoelastic materials.
\newblock {\em J. Differ. Equations}, 218(1):91--116, 2005.

\bibitem{Brez73}
H.~Br{\'e}zis.
\newblock {\em Op\'erateurs maximaux monotones et semi-groupes de contractions
  dans les espaces de {H}ilbert}.
\newblock North-Holland Publishing Co., Amsterdam-London; American Elsevier
  Publishing Co., Inc., New York, 1973.

\bibitem{BreStra73}
H.~Br\'{e}zis and W.~A. Strauss.
\newblock Semi-linear second-order elliptic equations in {$L^{1}$}.
\newblock {\em J. Math. Soc. Japan}, 25:565--590, 1973.

\bibitem{regularityElast}
M.~Costabel, M.~Dauge, and S.~Nicaise.
\newblock {Corner Singularities and Analytic Regularity for Linear Elliptic
  Systems. Part I: Smooth domains.}
\newblock 211 pages, Feb. 2010.

\bibitem{dafermos}
C.~Dafermos.
\newblock The second law of thermodynamics and stability.
\newblock {\em Arch. Ration. Mech. Anal.}, 70:167, 1979.

\bibitem{weakstrongweak}
E.~Emmrich and R.~Lasarzik.
\newblock Weak-strong uniqueness for the general {E}ricksen--{L}eslie system in
  three dimensions.
\newblock {\em Discrete Contin. Dyn. Syst.}, 38:4617--4635, 2018.

\bibitem{feireislstab}
E.~Feireisl.
\newblock Relative entropies in thermodynamics of complete fluid systems.
\newblock {\em Discrete Contin. Dyn. Syst.}, 32:3059, 2012.

\bibitem{Feireislrelative}
E.~Feireisl, B.~J. Jin, and A.~Novotn{\'y}.
\newblock Relative entropies, suitable weak solutions, and weak-strong
  uniqueness for the compressible {Navier}-{Stokes} system.
\newblock {\em J. Math. Fluid Mech.}, 14(4):717--730, 2012.

\bibitem{novotny}
E.~Feireisl and A.~Novotn{\'y}.
\newblock Weak-strong uniqueness property for the full
  {Navier}-{Stokes}-{Fourier} system.
\newblock {\em Arch. Ration. Mech. Anal.}, 204(2):683--706, 2012.

\bibitem{fei}
E.~Feireisl and A.~Novotn{\'y}.
\newblock {\em Singular limits in thermodynamics of viscous fluids}.
\newblock Adv. Math. Fluid Mech. Cham: Birkh{\"a}user, 2nd edition edition,
  2017.

\bibitem{hyper}
E.~Feireisl, E.~Rocca, G.~Schimperna, and A.~Zarnescu.
\newblock On a hyperbolic system arising in liquid crystals modeling.
\newblock {\em J. Hyperbolic Differ. Equ.}, 15:15--35, 2018.

\bibitem{fischer}
J.~Fischer.
\newblock A posteriori modeling error estimates for the assumption of perfect
  incompressibility in the {N}avier--{S}tokes equation.
\newblock {\em SIAM J. Numer. Anal.}, 53:2178, 2015.

\bibitem{fremond}
M.~Fr{\'e}mond.
\newblock {\em Non-smooth thermomechanics}.
\newblock Berlin: Springer-Verlag, 2002.

\bibitem{Gilly}
G.~Gilardi.
\newblock Personal communication.

\bibitem{Gilardi-Rocca}
G.~Gilardi and E.~Rocca.
\newblock Convergence of phase field to phase relaxation models governed by an
  entropy equation with memory.
\newblock {\em Math. Methods Appl. Sci.}, 29(18):2149--2179, 2006.

\bibitem{hale}
J.~Hale.
\newblock {\em Ordinary differential equations}.
\newblock Wiley-Interscience, New York, 1969.

\bibitem{HK-1}
C.~Heinemann and C.~Kraus.
\newblock Existence of weak solutions for {C}ahn-{H}illiard systems coupled
  with elasticity and damage.
\newblock {\em Adv. Math. Sci. Appl.}, 21(2):321--359, 2011.

\bibitem{HK-2}
C.~Heinemann and C.~Kraus.
\newblock Existence results for diffuse interface models describing phase
  separation and damage.
\newblock {\em European J. Appl. Math.}, 24(2):179--211, 2013.

\bibitem{HKRR}
C.~Heinemann, C.~Kraus, E.~Rocca, and R.~Rossi.
\newblock A temperature-dependent phase-field model for phase separation and
  damage.
\newblock {\em Arch. Ration. Mech. Anal.}, 225:177--247, 2017.

\bibitem{HR}
C.~Heinemann and E.~Rocca.
\newblock Damage processes in thermoviscoelastic materials with
  damage-dependent thermal expansion coefficients.
\newblock {\em Math. Methods Appl. Sci.}, 38(18):4587--4612, 2015.

\bibitem{Ioff77LSIF}
A.~D. Ioffe.
\newblock On lower semicontinuity of integral functionals. {I}.
\newblock {\em SIAM J. Control Optimization}, 15(4):521--538, 1977.

\bibitem{KnRoZa13VVAR}
D.~Knees, R.~Rossi, and C.~Zanini.
\newblock A vanishing viscosity approach to a rate-independent damage model.
\newblock {\em Math. Models Methods Appl. Sci.}, 23:565--616, 2013.

\bibitem{diss}
R.~Lasarzik.
\newblock Dissipative solution to the {E}ricksen--{L}eslie system equipped with
  the {O}seen--{F}rank energy.
\newblock {\em Z. Angew. Math. Phy.}, 70:8, 2018.

\bibitem{weakstrong}
R.~Lasarzik.
\newblock Weak-strong uniqueness for measure-valued solutions to the
  {E}ricksen--{L}eslie model equipped with the {O}seen--{F}rank free energy.
\newblock {\em J. Math. Anal. Appl.}, 470:36--90, 2019.

\bibitem{LRS}
R.~Lasarzik, E.~Rocca, and G.~Schimperna.
\newblock Weak solutions and weak-strong uniqueness for a thermodynamically
  consistent phase-field model.
\newblock {\em Atti Accad. Naz. Lincei Rend. Lincei Mat. Appl.}, 33:229--269,
  2022.

\bibitem{leray}
J.~Leray.
\newblock Sur le mouvement d'un liquide visqueux emplissant l'espace.
\newblock {\em Acta Mathematica}, 63:193--248, 1934.

\bibitem{rocca-rossi-deg}
E.~Rocca and R.~Rossi.
\newblock A degenerating {PDE} system for phase transitions and damage.
\newblock {\em Math. Models Methods Appl. Sci.}, 24(7):1265--1341, 2014.

\bibitem{rocca-rossi-full}
E.~Rocca and R.~Rossi.
\newblock ``Entropic'' solutions to a thermodynamically consistent {PDE} system
  for phase transitions and damage.
\newblock {\em SIAM J. Math. Anal.}, 47:2519--2586, 2015.

\bibitem{serrin}
J.~Serrin.
\newblock On the interior regularity of weak solutions of the
  {N}avier--{S}tokes equations.
\newblock {\em Arch. Rational Mech. Anal.}, 9:187--195, 1962.

\bibitem{ThoMie09DNEM}
M.~Thomas and A.~Mielke.
\newblock Damage of nonlinearly elastic materials at small strain: existence
  and regularity results.
\newblock {\em Zeit. Angew. Math. Mech.}, 90(2):88--112, 2010.

\bibitem{Troltzsch-book}
F.~Tr{\"o}ltzsch.
\newblock {\em Optimal control of partial differential equations. {Theory},
  methods and applications}, volume 112 of {\em Grad. Stud. Math.}
\newblock Providence, RI: American Mathematical Society (AMS), 2010.

\end{thebibliography}

}

\end{document}